\documentclass[reqno,11pt,twoside]{article}
\usepackage[hmargin=1.in, vmargin=1.in, marginparwidth=0.8in, marginparsep=0.1in, a4paper, centering]{geometry}
\usepackage[T1]{fontenc}
\usepackage[lining]{ebgaramond}
\usepackage{amsmath,amsthm}
\usepackage[ebgaramond]{newtxmath}
\usepackage{enumerate}
\usepackage[font=small, labelfont=bf, labelsep=colon, margin=0.5in]{caption}
\usepackage{graphicx}

\usepackage{xcolor}
\usepackage{bm}
\usepackage{etoolbox, apptools}
\usepackage{comment}
\renewcommand\footnotemark{}
\usepackage[nobottomtitles,pagestyles]{titlesec}
\titleformat{\section}
{\normalfont\large\bfseries}
{\filcenter\IfAppendix{Appendix }{\S}\thesection.}{1ex}{\filcenter}
\titleformat{\subsection}
{\normalfont\bfseries}
{\filcenter\S\thesubsection.}{1ex}{\filcenter}
\usepackage[colorlinks,allcolors=blue,pagebackref]{hyperref}
\usepackage{pifont}
\renewcommand*{\backrefalt}[4]{%
\ifcase #1 %
\textcolor{red}{No citations}%
\or
\ding{43}~p.~#2%
\else
\ding{43}~pp.~#2%
\fi}
\usepackage{breakcites}
\newcommand{\N}{\mathcal N}
\newcommand{\pa}{\partial\Omega}
\newcommand{\C}{{\mathbb C}}
\newcommand{\R}{{\mathbb R}}
\newcommand{\Z}{{\mathbb Z}}

\DeclareMathOperator*{\Spec}{Spec}
\newcommand{\DtN}{{\mathcal D}}
\newcommand{\myscal}[1]{\left(#1\right)}
\newcommand{\myscals}[1]{\left[#1\right]}
\newcommand{\mydotp}[1]{\left\langle #1\right\rangle}
\newcommand{\Dom}{\mathrm{Dom}}
\newcommand{\Dir}{\mathrm{Dir}}
\newcommand{\Neu}{\mathrm{Neu}}
\newcommand{\Rob}{\mathrm{Rob}}
\newcommand{\Zar}{\mathrm{Zar}}
\newcommand{\dr}{\mathrm{d}}
\newcommand{\er}{\mathrm{e}}
\newcommand{\ir}{\mathrm{i}}
\newcommand{\sr}{\mathrm{s}}
\newcommand{\ar}{\mathrm{a}}
\renewcommand{\tilde}{\widetilde}
\DeclareSymbolFont{alephletters}{OMS}{cmsy}{m}{n}
\SetSymbolFont{alephletters}{bold}{OMS}{cmsy}{b}{n}
\DeclareMathSymbol{\aleph}{\mathord}{alephletters}{"40}
\newcommand{\oro}{\mathrm{\aleph}}
\numberwithin{equation}{section}
\theoremstyle{plain}
\newtheorem{theorem}{Theorem}[section]
\newtheorem{proposition}[theorem]{Proposition}

\newtheorem{corollary}[theorem]{Corollary}
\newtheorem{conjecture}[theorem]{Conjecture}
\theoremstyle{definition}
\newtheorem{example}[theorem]{Example}

\theoremstyle{remark}
\newtheorem{remark}[theorem]{Remark}
\newcommand{\addQEDstyle}[2]{\AtBeginEnvironment{#1}{\pushQED{\qed}\renewcommand{\qedsymbol}{#2}}\AtEndEnvironment{#1}{\popQED}}
\addQEDstyle{remark}{$\blacktriangleleft$}
\addQEDstyle{example}{$\blacktriangleleft$}
\addQEDstyle{definition}{$\blacktriangleleft$}

\definecolor{darkteal}{rgb}{0.0, 0.5, 0.5}

\newcommand{\mydoi}[1]{\href{https://doi.org/#1}{doi: #1}}
\newcommand{\myarXiv}[1]{\href{https://arxiv.org/abs/#1}{arXiv: #1}}
\title{Spectral properties of the Dirichlet-to-Neumann map for the Helmholtz equation
}
\date{\small arXiv v. 1; 13 April 2026}
\author{
Denis S. Grebenkov\thanks{\textbf{D. S. G.}: Laboratoire de Physique de la Mati\`ere Condens\'ee, CNRS -- \'Ecole Polytechnique, Institut Polytechnique de Paris, 91120 Palaiseau, France; \href{mailto:denis.grebenkov@polytechnique.edu}{denis.grebenkov@polytechnique.edu}; \url{https://pmc.polytechnique.fr/pagesperso/dg/}; ORCID: 0000-0002-6273-9164%
}
\and 
Michael Levitin\thanks{%
\textbf{M. L.}: Department of Mathematics and Statistics, University of Reading, 
Pepper Lane, Whiteknights, Reading RG6 6AX, UK;
\href{mailto:M.Levitin@reading.ac.uk}{M.Levitin@reading.ac.uk}; \url{https://www.michaellevitin.net}; ORCID: 0000-0003-0020-3265%
}
\and
Iosif Polterovich\thanks{%
\textbf{I. P.}: D\'e\-par\-te\-ment de math\'ematiques et de statistique, Univer\-sit\'e de Mont\-r\'eal, 
CP 6128 succ Centre-Ville, Mont\-r\'eal QC  H3C 3J7, Canada;
\href{mailto:iosif.polterovich@umontreal.ca}{\nolinkurl{iosif.polterovich@umontreal.ca}}; \url{https://www.dms.umontreal.ca/\~iossif}; ORCID: 0009-0007-0052-6589%
}
}
\begin{document}
\maketitle
\begin{abstract}
The study of the Dirichlet-to-Neumann map and the associated Steklov problem for the Laplace equation has been a central topic in spectral geometry over the past decade. In this survey, we consider a more general framework in which the Laplace equation is replaced by the Helmholtz equation.
We examine how the properties of the Dirichlet-to-Neumann eigenvalues and eigenfunctions depend on the parameter in the Helmholtz equation and describe new phenomena arising when this parameter is nonzero, as  opposed to the Laplace case. In particular, we present various eigenvalue inequalities, analyse spectral asymptotics in different regimes, and investigate nodal domains and other features of eigenfunctions.  We also discuss applications where the Helmholtz parameter plays an essential role, as well as challenges encountered in the numerical computation of the Dirichlet-to-Neumann spectrum. 
\end{abstract}
{\small \tableofcontents}
\section{Introduction}\label{sec:intro}
Let  $\Omega \subset \R^d$ be a bounded domain.
Consider the Steklov-type eigenvalue problem for the Helmholtz equation in $\Omega$, 
\begin{equation}\label{eq:SteklovLambda}
\begin{cases}
-\Delta U  = \Lambda U  \qquad&\text{in  }\Omega,\\
\partial_n U = \sigma  U  \qquad&\text{on  }\pa.\\
\end{cases}
\end{equation}
Here $\partial_n U=\mydotp{\nabla U|_{\pa}, n}$ denotes the normal derivative with respect to the exterior normal $n$, $\Lambda \in \mathbb{R}$ is fixed, and $\sigma \in  \mathbb{R}$ is a spectral parameter. 

For  $\Lambda=0$, this is the classical Steklov problem
\begin{equation}\label{eq:Steklov_def}
\begin{cases}
-\Delta U  = 0 \qquad&\text{in  }\Omega,\\
\partial_n U = \sigma  U  \qquad&\text{on  }\pa,\\
\end{cases}
\end{equation}
that has been introduced more then a century ago \cite{Steklov1902} (see \cite{Kuznetsov14} for a historical discussion). Over the past two decades it has attracted significant  attention in spectral geometry and has become one of the central topics in the subject, see \cite{Girouard17}, \cite{Colbois22} as well as \cite[Chapter 7]{Levitin} and references therein.  In two dimensions, the eigenvalue problem \eqref{eq:Steklov_def} describes the free vibration of a membrane with all the weight localised in the boundary. It is therefore closely linked to the Neumann eigenvalue problem, see \cite{Girouard10b}, \cite{GirHL21} for various related results.

The goal of the present survey is to emphasise the new phenomena and applications  arising in the case $\Lambda \neq 0$.  Recently, there has been a lot of research on the problem \eqref{eq:SteklovLambda} in physics literature  in relation to modelling diffusion processes near a reflecting or partially absorbing boundary, see 
\S \ref{sec:reactions}. In this context, $\Lambda=0$  corresponds to the simplest model of particles with infinite lifetime and no reproduction. The case $\Lambda<0$  corresponds to particles that may disappear inside the domain during diffusion due to, e.g., biological death, radioactive decay or chemical disassembly.  At the same time, the case $\Lambda>0$ corresponds to diffusion with reproduction.  
As we will see, the properties of the eigenvalues and eigenfunctions of \eqref{eq:SteklovLambda} differ significantly depending on the sign of  the parameter 
$\Lambda$. 

The eigenvalue problem \eqref{eq:SteklovLambda} can be also viewed as a spectral problem for the Dirichlet-to-Neumann map
 $\DtN_\Lambda$.  Assuming that $\Lambda$ is not a Dirichlet eigenvalue of $\Omega$, the operator $\DtN_\Lambda$ associates to a
function $f$ on the boundary $\pa$ another function $g=\DtN_\Lambda f$ on $\pa$, which is he normal derivative of the unique solution $U$ of the corresponding 
Helmholtz equation with $U|_{\pa} = f$,  see  \S\ref{subsec:DtNmap} for a formal definition. 
In particular, the eigenvalues of \eqref{eq:SteklovLambda} are precisely the eigenvalues of the Dirichlet-to-Neumann map $\DtN_\Lambda$,
and the eigenfunctions  of $\DtN_\Lambda$ are the restrictions of the eigenfunctions of \eqref{eq:SteklovLambda} to the boundary.
The Dirichlet-to-Neumann map has numerous applications, in particular, for $\Lambda=0$ it can be viewed as the voltage-to-current map appearing in the celebrated Calder\'on's problem (see \cite{Calderon80}, \cite{Sylvester87}). In this case $f$ is viewed as a charge distribution on the boundary $\pa$ and $\DtN_0 f$ determines the electric current on the boundary.  In diffusion problems,  if $f$ represents the imposed concentration of particles on the boundary, $\DtN_\Lambda f$ determines their flux density.

Another important area of applications of the Steklov-type problems is the theory of linear water waves.  It has a long history, going back to the late nineteenth  century \cite{Greenhill}, \cite{Poincare}.  The small oscillations of a fluid in a container are described by the sloshing problem, which is a mixed Steklov--Neumann problem, see \S \ref{sec:mixed}. While  the  two-dimensional sloshing is described using the Laplace equation in the interior, in  three dimensions  the  Helmholtz equation  with a negative parameter arises after the separation of one spatial variable, see e.g. \cite{Peters52}, \cite{Mayrand20}. 

The Dirichlet-to-Neumann operators are closely related to various scattering and diffraction phenomena, 
see \cite{CakCol}, \cite{CakStek}, \cite{Behrndt17}, \cite{AgranovichBook}, and references therein, and to
fractional Laplacians \cite{Caffarelli07}, \cite{Arendt18}. 

We note that the eigenvalue problem \eqref{eq:SteklovLambda} can be formulated  on manifolds with boundary, and many results have been obtained in this context, see, for example, \cite{Girouard22}. To keep things simple, in this survey we stick to the case when $\Omega$ is a bounded Euclidean domain, as most of the spectral properties we are interested  in  are manifested in this geometrically easier setting.

The survey is organised as follows.
In \S\ref{subsec:prelim} and \S\ref{subsec:DtNmap} we present the necessary background notions and introduce the Dirichlet-to-Neumann map. The variational characterisation of the Dirichlet-to-Neumann eigenvalues is discussed in \S\ref{sec:varpr}. In particular, we emphasise the subtleties arising in the case 
$\Lambda>0$, see Remark \ref{rem:spaces}. The Robin--Dirichlet-to-Neumann duality  and some related questions concerning the iospectrality problem for the Dirichlet-to-Neumann maps are considered in \S \ref{subsec:duality}. 

Eigenvalues and eigenfunctions of the Dirichlet-to-Neumann operator  
can be calculated explicitly or expressed in terms of special functions only for a few classes of Euclidean domains such as intervals, balls, spherical shells, and cuboids. These computations
rely on symmetries of domains and separation of variables in the Helmholtz equation. Several examples are presented in  \S\ref{sec:examples}.

In \S \ref{sec:spectralproperties}  we focus on  the spectral properties of the operators $\DtN_\Lambda$. In particular, in \S \ref{sec:DtNDNL} we  give a formula expressing  the Dirichlet-to-Neumann map for  {\it any} value  of the parameter  in terms of  the Dirichlet-to-Neumann map for some {\em fixed} value of the parameter and  the Dirichlet spectral data, see Theorem \ref{thm:Dmatrix}. This result is used  to prove that that the union of the Dirichlet-to-Neumann spectra over all real values of the Helmholtz parameter can be represented  as a family of real-analytic curves, see Theorem \ref{thm:analyticity}. This, in turn, allows us to answer an open question  \cite[Open problem 4.11]{Bucur17}, see Theorem \ref{thm:Robincurves}. In \S\ref{sec:iso} and \S\ref{sec:otherestimates} we also discuss various inequalities for the Dirichlet-to-Neumann eigenvalues and conjecture a bound on the Dirichlet-to-Neumann eigenvalues varying along a fixed analytic eigenvalue branch for $\Lambda\le 0$, see Conjecture \ref{conj:sqroot}.

In \S\ref{sec:asympt} we consider eigenvalue asymptotics in different regimes. In particular, we study the behaviour of the principal eigenvalue as $\Lambda \to 0$, as well as the asymptotics of the spectrum as $\Lambda$ approaches a Dirichlet eigenvalue or tends to $-\infty$. We also discuss  Conjecture \ref{conj:polygons} concerning the asymptotics as $\Lambda\to-\infty$ of the spectrum of the Dirichlet-to-Neumann map on polygons. The Weyl asymptotics as $\sigma \to \infty$ under different regularity assumptions on the boundary is discussed in \S\ref{sec:weyl}.
  
The properties of the  eigenfunctions are discussed in \S\ref{sec:nodal}.  We present Courant's nodal domain theorem  for the interior (bulk) eigenfunctions and discuss the challenges arising in the nodal domain count for the boundary eigenfunctions.  We also explore the strict positivity of the first eigenfunction, as well as the  localisation properties of eigenfunctions near the boundary.
  
In \S\ref{sec:representation} we discuss layer potentials and Green's functions under different boundary conditions and their links to Dirichlet-to-Neumann maps.

Several extensions of the Steklov problem for the Helmholtz equation are considered in \S\ref{sec:extensions}.  In \S\ref{sec:mixed}  we discuss the mixed Steklov--Dirichlet and Steklov--Neumann problems, the latter being of particular importance in view of the sloshing problem that was mentioned earlier.
We also briefly discuss the Dirichlet-to-Neumann operator on exterior domains  as well as the case of a  complex parameter 
$\Lambda$.

In \S \ref{sec:applications} we present various numerical methods that are used to compute the Dirichlet-to-Neumann eigenvalues, such as finite elements, boundary elements and the method of fundamental solutions.  We also discuss domain decomposition techniques.  Finally, \S \ref{sec:reactions} is concerned with applications to diffusion-controlled systems, which provides an important physical motivation  for the subject  of this survey. The paper is concluded by two short appendices containing auxiliary results.

\section{Setting of the problem}\label{sec:basic}
\subsection{Preliminaries}\label{subsec:prelim}
Let $\Omega \subset \R^d$ be a bounded Euclidean domain  (i.e., a connected open set).   
Consider the eigenvalue problems for the Laplace operator with the Dirichlet, Neumann and Robin boundary conditions on the boundary $\pa$:
\begin{equation}\label{eq:Laplacian}
\begin{gathered}
\text{Dirichlet}\\
\left\{
\begin{aligned}
 -\Delta U  &= \lambda U,\\ 
 U|_{\pa}  &= 0,
\end{aligned} 
\right.
\end{gathered}
\qquad\qquad\qquad
\begin{gathered}
\text{Neumann}\\
\left\{
\begin{aligned}
-\Delta U  &= \lambda U,\\ 
\partial_n U  &= 0,
\end{aligned} 
\right. 
\end{gathered}
\qquad\qquad\qquad
\begin{gathered}
\text{Robin}\\
\left\{
\begin{aligned}
-\Delta U &= \lambda U,\\ 
\partial_n U+\gamma U|_{\pa}  &= 0.
\end{aligned} 
\right.
\end{gathered}
\end{equation}
Here, as before,  $\partial_n$ is the normal derivative at the boundary, and $\gamma\in\R$ is a fixed parameter. Note that for $\gamma=0$ the Robin problem coincides with the Neumann one. 

We denote by $-\Delta_\Omega^\Dir$, $-\Delta_\Omega^\Neu$, and $-\Delta_\Omega^{\Rob,\gamma}$ the Dirichlet, Neumann and Robin Laplacians in $\Omega$, according to their boundary conditions in \eqref{eq:Laplacian}. 
While the Dirichlet Laplacian on a bounded domain always has discrete spectrum, one needs to impose some regularity conditions in order for this to hold for the Neumann and Robin Laplacians. In what follows we assume that $\Omega$  has Lipschitz boundary (we say that $\Omega$ is Lipschitz for short), see e.g. \cite[Appendix B.3]{Levitin} and references therein for a formal definition. Essentially, it means that the boundary  $\partial \Omega$ can be locally represented as a graph of a Lipschitz function. In particular, all domains with $C^1$ boundary and  convex domains are Lipschitz in any dimension,  all polygons are Lipschitz, but domains with cusps and slits are not, see Figure  \ref{fig:nonLipschitz}. 
\begin{figure}[!htbp]
\centering
\includegraphics{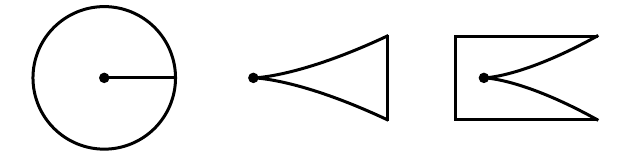}
\caption{
Examples of non-Lipschitz domains, in which the Lipschitz condition is not satisfied at the point shown. From left to right: a disk with a cut; a planar domain with an external cusp; a planar domain with an internal cusp. 
\label{fig:nonLipschitz}}
\end{figure}
For bounded Lipschitz domains, or, more generally, for bounded open sets (not necessarily connected) with Lipschitz boundaries, the Neumann and Robin Laplacian have discrete spectra;  note that these assumptions are sufficient but not necessary.   

We denote the eigenvalues of the problems \eqref{eq:Laplacian} by $\lambda_k^\Dir$, $\lambda_k^\Neu$, and $\lambda_k^{\Rob,\gamma}$, respectively, enumerating them by $k = 1,2,\ldots$, in an increasing order, with the account of multiplicities.   They  are
positive in the Dirichlet case and in the Robin case with $\gamma>0$, and non-negative in the Neumann case (for which $\lambda_1^\Neu=0$). 
In some cases, we will specify explicitly
the domain $\Omega$ by writing, e.g., $\lambda_k^\Dir(\Omega)$ instead of
$\lambda_k^\Dir$.  We also denote by $\Spec(A)$ the spectrum of an
operator $A$, again understood as a multiset with account of multiplicities.

\subsection{The Dirichlet-to-Neumann map}\label{subsec:DtNmap}

Let, as above, $\Omega \subset \R^d$, $d \geq 2$,  be a bounded Lipschitz domain.
In what follows, we use the theory of Sobolev spaces $H^s(\Omega)$, $H_0^s(\Omega)$, and $H^s(\pa)$, $s\in\mathbb{R}$,  on Lipschitz domains and their boundaries.
There is an extensive literature on this subject, for example \cite{McLean}, \cite{Grisvard}, \cite{Adams}.  We particularly recommend a concise exposition in \cite[Appendix]{Chandler12} which is sufficient for our purposes.
For a given Hilbert space $X$ (which may be a Sobolev or a Lebesgue space), we will denote by $\myscal{\cdot,\cdot}_X$ and $\|\cdot\|_X$ the inner product and the norm in $X$.

Fix a parameter $\Lambda\in \C\setminus \Spec\left(-\Delta_\Omega^\Dir\right)$, and consider a non-homogeneous boundary value problem for the Helmholtz equation
\begin{equation}\label{eq:HelmholtzML}
\left\{
\begin{aligned}
-\Delta U &= \Lambda U\qquad \text{in }\Omega,\\
U|_{\pa} &= u,
\end{aligned}
\right.
\end{equation}
where $u\in H^{1/2}(\pa)$ is given. The problem has a unique solution $U\in H^1(\Omega)$, which we will call the \emph{$\Lambda$-harmonic extension} of $u$ (by analogy with the ordinary harmonic extension in the case $\Lambda=0$). We will denote this solution by
$U=\mathcal{E}_\Lambda u\in \mathcal{H}_\Lambda(\Omega)\subset H^1(\Omega)$, where
\[
\mathcal{H}_\Lambda(\Omega)=\left\{U\in H^1(\Omega): -\Delta U-\Lambda U=0\right\}=\left\{\mathcal{E}_\Lambda u: u\in H^{1/2}(\pa)\right\}
\]
is the space of all $\Lambda$-harmonic functions on $\Omega$, see also Proposition \ref{prop:Elambda}. 

We now define the \emph{Dirichlet-to-Neumann operator} (or \emph{map}) $\DtN_\Lambda: H^{1/2}(\pa)\to H^{-1/2}(\pa)$ as
\[
\DtN_\Lambda: u\mapsto  \partial_n \left(\mathcal{E}_\Lambda u\right).
\]
In other words, the linear operator $\DtN_\Lambda$ maps the Dirichlet datum $u=U|_{\pa}$ of a $\Lambda$-harmonic function $U$ into its Neumann datum $\partial_n U = \DtN_\Lambda u$. We refer to \cite{Chandler12}  for the precise definition of the normal derivative $\partial_n: U\mapsto \partial_n U$ as an operator $H^1(\Omega)\to H^{-1/2}(\pa)$.  
An alternative definition of the Dirichlet-to-Neumann map in
terms of boundary triplets is discussed in \cite{BHdS}, see also
\cite{Safarov2008} for an abstract Hilbert space analogue.  
Further extensions include Dirichlet-to-Neumann maps associated to
second-order elliptic operators
\cite{Auchmuty12}, \cite{Auchmuty13}, \cite{Agranovich13}, \cite{Arendt14} or dealing with
other types of boundary conditions
\cite{Shamma72}, \cite{Shamma75}, \cite{Gesztesy08}, \cite{Kennedy08}, \cite{Kennedy10}, \cite{Dambrine16}.  

We have in fact defined a family of operators parametrised by
$\Lambda\in\mathbb{C}\setminus
\Spec\left(-\Delta_\Omega^\Dir\right)$. For the purposes of this
paper, we will be mostly working with real values of $\Lambda$, except in \S\ref{sec:complex}. The
particular case $\Lambda = 0$ corresponds to the operator $\DtN_0$
arising in the original Steklov problem \eqref{eq:Steklov_def}.

We briefly discuss the extension of the definition of the Dirichlet-to-Neumann operator $\DtN_\Lambda$ to the case $\Lambda\in\Spec\left(-\Delta_\Omega^\Dir\right)$. Let 
\[
\mathcal{H}_{\Lambda,0}=\mathcal{H}_\Lambda(\Omega) \cap H^1_0(\Omega)
\]
be the finite-dimensional eigenspace of the Dirichlet Laplacian corresponding to the eigenvalue $\Lambda$. Set
\[
\mathcal{K}_{\Lambda}=\left\{\partial_n U^\Dir: U^\Dir\in\mathcal{H}_{\Lambda,0}\right\}
\]
to be the finite-dimensional linear space of the Neumann traces of the eigenfunctions, and
\[
\mathcal{K}^\perp_{\Lambda}=\left\{ u\in H^{1/2}(\pa): \myscal{u, \partial_n U^\Dir}_{L^2(\pa)}=0\text{ for all }U^\Dir\in\mathcal{H}_{\Lambda,0}\right\}
\]
to be its $L^2(\pa)$-orthogonal complement. Then the non-homogeneous problem \eqref{eq:HelmholtzML} is solvable only if  $u \in\mathcal{K}^\perp_{\Lambda}$ \cite[Theorem 4.10]{McLean}. The solution $U$ in this case is not unique but is defined modulo an addition of  an element from $\mathcal{H}_{\Lambda,0}$. If we treat $\mathcal{E}_\Lambda u$ as a multi-valued map,  and denote by  $\Pi_{\mathcal{K}_{\Lambda}^\perp}$ the projection onto the subspace $\mathcal{K}_{\Lambda}^\perp$, then we can define the operator $u\mapsto \Pi_{\mathcal{K}_{\Lambda}^\perp}\partial_n \mathcal{E}_\Lambda u$ with the domain $H^1(\pa)\cap \mathcal{K}^\perp_{\Lambda}$  as the \emph{Dirichlet-to-Neumann map} for $\Lambda\in\Spec\left(-\Delta_\Omega^\Dir\right)$,  see \cite{Arendt12}, \cite{Arendt14}, \cite{Behrndt15}, \cite{Berkolaiko22} for further details. We note that the additional regularity is required in that definition because we need to ensure that the projector is well-defined. We also used the fact that eigenfunctions of the Dirichlet Laplacian in a Lipschitz  domain $\Omega$ belong to $H^{3/2}(\Omega)$ \cite{JerKen}.

For a bounded domain $\Omega$ with a Lipschitz boundary $\pa$, and for a fixed $\Lambda\in\mathbb{R}$, the operator $\DtN_\Lambda$ is
self-adjoint, semi-bounded below, and has a compact resolvent
\cite{Arendt07}. Additionally, when the boundary $\pa$ is smooth, for any $\Lambda \in
\R \backslash \Spec\left(-\Delta_\Omega^\Dir\right)$, $\DtN_\Lambda$ is an elliptic pseudodifferential
operator of order one, see \cite[Chapter 7.11]{Taylor3}. In
particular, it is a nonlocal operator, in sharp contrast with the
Laplace operator.  For smooth domains, the 
Dirichlet-to-Neumann operator $\DtN_\Lambda$ essentially behaves as
$\left(-\Delta_{\pa}\right)^{1/2}$, up to a pseudodifferential
operator of order zero, where $\Delta_{\pa}$ is the Laplace--Beltrami
operator on the boundary $\pa$. The link between the DtN map and the boundary Laplacian exists also for less regular boundaries, see \cite{Girouard22}.
For construction of the Dirichlet-to-Neumann operators and their
spectral properties for some classes of non-Lipschitz domains, see, e.g.,
\cite{Arendt11}, \cite{Nazarov08}, \cite{Nazarov09}. 

We note that the eigenfunctions of the Dirichlet-to-Neumann operator for a bounded domain $\Omega$ with Lipschitz boundary in fact belong to $H^1(\pa)$ \cite{MitTay}. This explains why $\DtN_\Lambda$ is sometimes considered as an operator on $L^2(\pa)$ with the domain $H^1(\pa)$. Moreover, provided the boundary is smooth,  it can be considered as an operator acting from $C^\infty(\pa)$ to $C^\infty(\pa)$  which does not affect its spectral properties. 

The standard integration by parts formula yields the bilinear form of $\DtN_\Lambda$,
\begin{equation}\label{eq:bilform}
\myscal{\DtN_\Lambda u,v}_{L^2(\pa)}=\myscal{\nabla U,\nabla V}_{L^2(\Omega)} - \Lambda\myscal{U,V}_{L^2(\Omega)},
\end{equation}
valid for all 
\[
u, v\in\Dom(\DtN_\Lambda)=\begin{cases}
H^{1/2}(\pa)&\quad\text{if }\Lambda\not\in\Spec\left(-\Delta_\Omega^\Dir\right),\\
H^1(\pa)\cap  \mathcal{K}^\perp_{\Lambda}&\quad\text{if }\Lambda\in\Spec\left(-\Delta_\Omega^\Dir\right),
\end{cases}
\]
with $U=\mathcal{E}_\Lambda u$ and $V=\mathcal{E}_\Lambda v$. The right-hand side of \eqref{eq:bilform} makes sense for all $U,V\in H^1(\Omega)$, and we denote it by 
\begin{equation} \label{eq:Lambdaprod}
\myscals{U,V}_\Lambda := \myscal{\nabla U, \nabla V}_{L^2(\Omega)}-\Lambda\myscal{U,V}_{L^2(\Omega)}, \qquad U,V\in H^1(\Omega).
\end{equation}

\begin{remark} 
If $\Lambda<0$, then \eqref{eq:Lambdaprod} defines an inner product on $H^1(\Omega)$ inducing the norm $\sqrt{\myscals{U,U}_\Lambda}$, which is equivalent to the standard norm $\|U\|_{H^1(\Omega)}=\sqrt{\myscals{U,U}_{-1}}$.
\end{remark}

The bilinear form \eqref{eq:bilform} induces the quadratic form of $\DtN_\Lambda$, 
\begin{equation}\label{eq:quadform}
\myscal{\DtN_\Lambda u,u}_{L^2(\pa)}=\|\nabla U\|^2_{L^2(\Omega)} - \Lambda\|U\|^2_{L^2(\Omega)}=[U,U]_\Lambda, 
\end{equation}
for all $u\in\Dom(\DtN_\Lambda)$, with $U=\mathcal{E}_\Lambda u$.

For future use, we also present an explicit expression for the $\Lambda$-harmonic extension in terms of the spectral data of the Dirichlet Laplacian. 

\begin{proposition}\label{prop:Elambda} 
We have, for $\Lambda\not\in\Spec\left(-\Delta^\Dir\right)$ and $u\in H^{1/2}(\pa)$,

\[
\mathcal{E}_\Lambda u = \sum_{m=1}^\infty \frac{\myscal{u,\partial_n U^\Dir_m}_{L^2(\pa)}}{\Lambda-\lambda_m^\Dir} U^\Dir_m,
\]
where the equality is understood in the $L^2(\Omega)$-sense. Moreover this also works for $\Lambda\in\Spec\left(-\Delta^\Dir\right)$ as long as $u\in\mathcal{K}^\perp_{\Lambda}$, and we assume the convention $\frac{0}{0}:=0$.
\end{proposition}

\begin{proof} As $U=\mathcal{E}_\Lambda u\in H^1(\Omega)$, it is also in $L^2(\Omega)$ and can be expanded in the basis of Dirichlet eigenfunctions,
\[
U=\sum_{m=1}^\infty\myscal{U,  U^\Dir_m}_{L^2(\Omega)} U^\Dir_m.
\]
Integration by parts yields
\[
\begin{split}
(\Lambda-\lambda_m^\Dir) \myscal{U,  U^\Dir_m}_{L^2(\Omega)}  &= \myscal{-\Delta U,  U^\Dir_m}_{L^2(\Omega)} - \myscal{U,  -\Delta U^\Dir_m}_{L^2(\Omega)}\\
&= \myscal{u, \partial_n U^\Dir_m}_{L^2(\pa)},
\end{split}
\]
and the result follows immediately. 
\end{proof}

\begin{remark} 
It is important to emphasise that the equality in Proposition \ref{prop:Elambda} is only valid in the  $L^2(\Omega)$ sense. 
\end{remark}

\subsection{Eigenvalues and the variational principle}\label{sec:varpr}

Let $\Omega\subset\mathbb{R}^d$ be a bounded Lipschitz domain, and  let us fix $\Lambda\in\mathbb{R}\setminus\Spec\left(-\Delta_\Omega^\Dir\right)$. The spectral problem for the Dirichlet-to-Neumann map $\DtN_\Lambda$,
\[
\DtN_\Lambda u = \sigma u,
\]
coincides with \eqref{eq:SteklovLambda}. 

The properties of the operator $\DtN_\Lambda$ described in the previous section imply that $\DtN_\lambda$ has a discrete  real spectrum
\cite{Arendt12}, \cite{Behrndt15}, i.e. there exists a sequence of eigenvalues 
\[
\sigma_1^{(\Lambda)} \leq \sigma_2^{(\Lambda)} \leq \ldots \nearrow +\infty, 
\]
ordered with multiplicities and accumulating to $+\infty$, 
and a corresponding sequence of  eigenfunctions $u_k^{(\Lambda)}\in\Dom(\DtN_\Lambda)\setminus\{0\}$ such that
\[
\DtN_\Lambda \, u_k^{(\Lambda)} = \sigma_k^{(\Lambda)} \, u_k^{(\Lambda)},  \qquad k \in\mathbb{N},
\]
or, more explicitly,  
with $U_k^{(\Lambda)}:=\mathcal{E}_\Lambda u_k^{(\Lambda)}$, 
\begin{equation}\label{eq:efexplicit}
\begin{cases}
-\Delta U_k^{(\Lambda)} = \Lambda U_k^{(\Lambda)} \qquad&\text{in  }\Omega,\\
\partial_n U_k^{(\Lambda)} = \sigma_k^{(\Lambda)} U_k^{(\Lambda)} \qquad&\text{on  }\pa.\\
\end{cases}
\end{equation}

\begin{remark}
The definition of the Dirichlet-to-Neumann map and its spectral properties remain unchanged if we consider a bounded open set $\Omega\subset \mathbb{R}^d$ with Lipschitz boundary rather than a bounded domain, thus not assuming connectedness of $\Omega$. At the same time, if $\Omega$ is a disjoint union of $m$ bounded Lipschitz domains, $\Omega=\bigsqcup_{j=1}^m \Omega_j$, then the spectrum of the Dirichlet-to-Neumann map on $\Omega$ is the union of its spectra on each of $\Omega_j$. Thus, without loss of generality, we can assume from now on that $\Omega$ is connected. 
\end{remark}

When $\Lambda=0$, the eigenvalues $\sigma_k^{(0)}$ exhibit a natural scaling upon
dilations of the bounded domain $\Omega$ by a factor $\alpha > 0$, i.e.
\begin{equation}\label{eq:scaling0}
\sigma_k^{(0)}(\alpha \Omega) = \frac{1}{\alpha} \sigma_k^{(0)}(\Omega).
\end{equation}
In physical terms, it is consistent with the fact that the eigenvalues
$\sigma_k^{(0)}$ have units length$^{-1}$.  In turn, if $\Lambda \ne 0$, a
dilation of $\Omega$ also affects the Laplace operator and thus
requires rescaling of $\Lambda$,
\begin{equation}\label{eq:scaling}
\sigma_k^{(\Lambda)}(\alpha \Omega) = \frac{1}{\alpha} \sigma_k^{(\Lambda \alpha^2)}(\Omega).
\end{equation}
This property will be important for the discussion of isoperimetric inequalities in \S\ref{sec:iso}.

As is true for any self-adjoint semi-bounded below operator with a discrete spectrum, the eigenvalues $\sigma_k^{(\Lambda)}$ satisfy the variational (minimax) relation
\begin{equation}\label{eq:abs_minimax}
\sigma_k^{(\Lambda)} =  \min_{\substack{\widetilde{\mathcal{L}}\subset H^{1/2}(\pa) \\ \dim \widetilde{\mathcal{L}} = k}}\  \max_{u \in\widetilde{\mathcal{L}}\setminus\{0\}} \frac{\myscal{\DtN_\Lambda u,u}_{L^2(\pa)}}{\|u\|^2_{L^2(\pa)}}, \qquad k\in\mathbb{N}.
\end{equation}
Since we can uniquely identify a function $u\in H^{1/2}(\pa)$ with its $\Lambda$-harmonic extension $U=\mathcal{E}_\Lambda u\in \mathcal{H}_\Lambda(\Omega)$, we can then use \eqref{eq:quadform} in order to rewrite \eqref{eq:abs_minimax} as
\begin{equation} \label{eq:muk_minimax}
\sigma_k^{(\Lambda)} = \min_{\substack{\mathcal{L} \subset \mathcal{H}_\Lambda(\Omega)\\\dim\mathcal{L} = k}}\  \max_{U \in \mathcal{L}\setminus\{0\}} 
\frac{\|\nabla U\|^2_{L^2(\Omega)} - \Lambda \| U\|^2_{L^2(\Omega)}}{\|U|_{\pa}\|^2_{L^2(\pa)}},  \qquad k\in\mathbb{N},
\end{equation}
and in particular 
\begin{equation} \label{eq:muk_minimaxk1}
\sigma_1^{(\Lambda)} =\min_{U \in  \mathcal{H}_\Lambda(\Omega)\setminus\{0\}} 
\frac{\|\nabla U\|^2_{L^2(\Omega)} - \Lambda \| U\|^2_{L^2(\Omega)}}{\|U|_{\pa}\|^2_{L^2(\pa)}},  \qquad k\in\mathbb{N}.
\end{equation}

Moreover, if $\Lambda$ is below the first Dirichlet eigenvalue of $\Omega$, that is, $\Lambda < \lambda_1^\Dir(\Omega)$, then the space of $\Lambda$-harmonic functions $\mathcal{H}_\Lambda(\Omega)$ in
\eqref{eq:muk_minimax} can be replaced by a (larger) Sobolev space $H^1(\Omega)$ to yield a simpler form
\begin{equation} \label{eq:muk_minimax1}
\sigma_k^{(\Lambda)} = \min_{\substack{\widetilde{\mathcal{L}} \subset H^1(\Omega)\\\dim\widetilde{\mathcal{L}} = k}}\  \max_{W \in \widetilde{\mathcal{L}}\setminus\{0\}} 
\frac{\|\nabla W\|^2_{L^2(\Omega)} - \Lambda \| W\|^2_{L^2(\Omega)}}{\|W|_{\pa}\|^2_{L^2(\pa)}},  \qquad k\in\mathbb{N}. 
\end{equation}
In order to prove \eqref{eq:muk_minimax1}, we  require the following useful decomposition result.

\begin{proposition} 
For any $\Lambda\in\mathbb{R}\setminus\Spec\left(-\Delta^\Dir_\Omega\right)$, the Sobolev space $H^1(\Omega)$ can be decomposed into the direct sum
\[
H^1(\Omega)=\mathcal{H}_\Lambda(\Omega) \oplus H^1_0(\Omega).
\]
Moreover,
\[
\myscals{U,V}_\Lambda=0\qquad\text{for any }U\in\mathcal{H}_\Lambda(\Omega), V\in H^1_0(\Omega).
\]
\end{proposition}

\begin{proof} 
Let $W\in H^1(\Omega)$. Set $U:=\mathcal{E}_\Lambda\left(W|_{\pa}\right)$ and $V:=W-U$. Then, obviously, $U\in\mathcal{H}_\Lambda(\Omega)$, and $V\in H^1_0(\Omega)$ since $W|_{\pa}=U|_{\pa}$. 

We further have, by definition \eqref{eq:bilform},  integration by parts, and using  $-\Delta U-\Lambda U=0$ and $V|_{\pa}=0$,
\[
\myscals{U,V}_\Lambda= \myscal{\nabla U, \nabla V}_{L^2(\Omega)}-\Lambda\myscal{U,V}_{L^2(\Omega)} =  \myscal{-\Delta U, V}_{L^2(\Omega)}-\Lambda\myscal{U,V}_{L^2(\Omega)}=0.
\]
\end{proof}

Returning now to the proof of the fact that the right-hand side of \eqref{eq:muk_minimax1} coincides with that of \eqref{eq:muk_minimax}, we substitute 
$W=U+V$ with  $U\in\mathcal{H}_\Lambda(\Omega)$, $V\in H^1_0(\Omega)$ into \eqref{eq:muk_minimax1}, whence the Rayleigh quotient becomes
\[
\frac{\myscals{W,W}_\Lambda}{\|W|_{\pa}\|^2_{L^2(\pa)}}=\frac{\myscals{U,U}_\Lambda+\myscals{V,V}_\Lambda}{\|U|_{\pa}\|^2_{L^2(\pa)}}\ge \frac{\myscals{U,U}_\Lambda+(\lambda_1^\Dir(\Omega)-\Lambda)\|V\|^2_{L^2(\Omega)}}{\|U|_{\pa}\|^2_{L^2(\pa)}}\
\] 
as $\|\nabla U\|^2_{L^2(\Omega)}\ge \lambda_1^\Dir(\Omega)\| U\|^2_{L^2(\Omega)}$ for all $U\in H^1_0(\Omega)$. Since the second term in the numerator is non-negative and $V$ does not appear in the denominator, the minimisation procedure requires taking $V=0$, thus giving the right-hand side of  \eqref{eq:muk_minimax}.

\begin{remark}\label{rem:spaces} 
It is important to emphasise that the minimax principle \eqref{eq:muk_minimax1} \emph{cannot} be used for $\Lambda>\lambda_1^\Dir(\Omega)$. To illustrate this, we use the following example due to L.~Friedlander: let $W=U_1^{\Dir}+\epsilon\in H^1(\Omega)$, where $U_1^{\Dir}$ is an $L^2(\Omega)$-normalised principal eigenfunction of the Dirichlet Laplacian, and $\epsilon\ne 0$ is a small constant. Then the numerator of the Rayleigh quotient in \eqref{eq:muk_minimax1} converges to $-\left(\Lambda-\lambda_1^\Dir\right)$ as $\epsilon\to 0$, whereas the denominator is $\epsilon^2 |\pa|_{d-1}$, and thus the Rayleigh quotient tends to $-\infty$ as $\epsilon\to 0$ when $\Lambda>\lambda_1^\Dir(\Omega)$, contradicting the fact that $\DtN_\Lambda$ is semi-bounded below. 
\end{remark}

As a direct application of the variational principle, we have
 
\begin{proposition}\label{prop:sigma1upper}
Let $\Omega\subset\mathbb{R}^d$ be a bounded Lipschitz domain with volume $|\Omega|_d$ and boundary area $|\pa|_{d-1}$. Then 
\[
\sigma_1^{(\Lambda)}(\Omega) \le -\frac{\Lambda |\Omega|_d}{|\pa|_{d-1}}\qquad\text{for all }\Lambda<\lambda_1^\Dir(\Omega).
\]
\end{proposition}

\begin{proof} 
Use the test function $U\equiv 1$ in the right-hand side of \eqref{eq:muk_minimax1}.
\end{proof}

\subsection{Basic properties of eigenfunctions}\label{subsec:basic}

The eigenfunctions $\left\{u_k^{(\Lambda)}\right\}_{k=1}^\infty$ of the Dirichlet-to-Neumann
operator $\DtN_\Lambda$ form a complete  basis in $L^2(\pa)$ which can be chosen to be orthonormal.
From now on, we assume that
\begin{equation}\label{eq:normaliseef}
\myscal{u_j^{(\Lambda)}, u_k^{(\Lambda)}}_{L^2(\pa)} = \delta_{j,k}, 
\end{equation}
where $\delta_{j,k}$ is the Kronecker symbol.

\begin{remark} We emphasise the following terminological dichotomy. For any $\Lambda$, the \emph{eigenfunctions of the Dirichlet-to-Neumann map} $\DtN_\Lambda$ are the functions $u_k^{(\Lambda)}$ defined on the boundary of the domain $\Omega$, and the solutions of the corresponding Helmholtz equation  
are their $\Lambda$-harmonic extensions $U^{(\Lambda)}_k=\mathcal{E}_\Lambda u_k^{(\Lambda)}$ (defined inside $\Omega$). However in the context of the Steklov problem (with $\Lambda=0$) the functions $U^{(0)}_k$ are usually referred to as \emph{eigenfunctions of the Steklov problem} (or simply as  \emph{Steklov eigenfunctions}). Slightly abusing terminology, we will refer to $U^{(\Lambda)}_k$ for arbitrary $\Lambda$ as \emph{bulk eigenfunctions} of the Dirichlet-to-Neumann map.
\end{remark}

\begin{remark}
For real $\Lambda$, the eigenfunctions $u_k^{(\Lambda)}$ can be chosen to be real. 
By elliptic regularity,  the bulk eigenfunctions $\left\{U_k^{(\Lambda)}\right\}$ are in this case 
real-analytic functions in $\Omega$. If, additionally, the boundary is real-analytic, then $\left\{u_k^{(\Lambda)}\right\}$ is similarly a family of real-analytic functions on $\partial\Omega$.  See also Theorem \ref{thm:regularity}.
\end{remark}

We have the following explicit representation of bulk eigenfunctions of $\DtN_\Lambda$  in terms of the corresponding eigenfunctions of $\DtN_\Lambda$ and of the spectral data of the Dirichlet Laplacian.

\begin{proposition}\label{prop:Eexpand} 
For any bounded Lipschitz domain $\Omega\subset\mathbb{R}^d$ and any $\Lambda\in\mathbb{R}$,  we have
\[
U^{(\Lambda)}_k = \sum_{m=1}^\infty \frac{\myscal{u^{(\Lambda)}_k, \partial_n U^\Dir_m}_{L^2(\pa)}}{\Lambda-\lambda_m^\Dir} U^\Dir_m,\qquad k\in\mathbb{N},
\]
where the equality is understood in the $L^2(\Omega)$-sense.
\end{proposition}

\begin{proof} Immediate by Proposition \ref{prop:Elambda} with $u:=u^{(\Lambda)}_k$.
\end{proof}

The functions $U^{(\Lambda)}_k$ are not, generally speaking, mutually orthogonal in $L^2(\Omega)$ (see also \cite{Wang} for some  related results on ``almost-orthogonality'' of Steklov eigenfunctions).
However, using formula \eqref{eq:bilform} with $u= u_k^{(\Lambda)}$ and $v=u_j^{(\Lambda)}$, and notation \eqref{eq:Lambdaprod}, together with our normalisation \eqref{eq:normaliseef}, we immediately obtain that 
\[
\myscals{U^{(\Lambda)}_k, U^{(\Lambda)}_j}_\Lambda = 
\begin{cases} 
0\qquad&\text{if }j\ne k,\\
\sigma_k^{(\Lambda)}\qquad&\text{if }j=k.
\end{cases}
\]

In addition, the integral of the Helmholtz equation in \eqref{eq:efexplicit}
over $\Omega$ together with the boundary conditions yields a useful relation
\[
\Lambda \int_\Omega U_k^{(\Lambda)} = - \sigma_k^{(\Lambda)} \int_{\pa} u_k^{(\Lambda)}.
\]

\subsection{Robin--Dirichlet-to-Neumann duality}\label{subsec:duality}

Even a cursory comparison of the Robin spectral problem \eqref{eq:Laplacian} and the spectral problem \eqref{eq:SteklovLambda} for the Dirichlet-to-Neumann map shows that these problems are closely linked together, see e.g.  \cite{GNP76}.  Indeed, replacing in the Robin problem \eqref{eq:Laplacian} the fixed parameter $-\gamma$ by the spectral parameter $\sigma$ and the spectral parameter $\lambda$ by a fixed parameter $\Lambda$, we arrive at  \eqref{eq:SteklovLambda}. Formally, this link is known as \emph{Robin--Dirichlet-to-Neumann duality}.

\begin{proposition}[{\cite{Arendt12,Hassannezhad22}}]\label{prop:DtNRduality}
Let $\Omega\subset\mathbb{R}^d$ be a bounded Lipschitz domain, and let $\Lambda, \sigma\in\mathbb{R}$. Then $\sigma\in\Spec\left(\DtN_\Lambda\right)$ if and only if $\Lambda\in\Spec\left(-\Delta^{\Rob,-\sigma}\right)$. Moreover, the multiplicities of $\sigma$ as an eigenvalue of $\DtN_\Lambda$ and $\Lambda$ as an eigenvalue of $-\Delta^{\Rob,-\sigma}$ coincide, and $\mathcal{E}_\Lambda u$ is an eigenfunction of the  Robin Laplacian $-\Delta^{\Rob,-\sigma}$ if and only if $u$ is an eigenfunction of $\DtN_\Lambda$.
\end{proposition}
This duality, which we will  later quantify further in Proposition \ref{prop:dualityfurther} and which also holds for complex-valued $\Lambda$ (see \S\ref{sec:complex}), allows one to translate many results known for the Robin Laplacian to Dirichlet-to-Neumann maps, as illustrated in particular in \S\ref{sec:lambda_large}.

It is instructive to see how the Robin--Dirichlet-to-Neumann duality manifests itself in relation to \emph{isospectrality}.  The study of isospectral (i.e. having the same spectra with the
account of multiplicities) non-isometric domains and Riemannian manifolds is one of
the central themes in spectral geometry, 
going back to
the celebrated question of Mark Kac ``Can one hear the shape of a
drum?'' \cite{Kac}. Very little is known about this problem in the
context of the Dirichlet-to-Neumann operator, as well as of the Robin
eigenvalue problem.  In particular, it is not known whether there
exist non-isometric Euclidean bounded domains $\Omega_1$ and $\Omega_2$ which
are $\DtN_\Lambda$-isospectral for some real $\Lambda$ not in the
union of the Dirichlet spectra of $\Omega_1$ and $\Omega_2$.
 
At the same time, as was shown in \cite{GHW}, there exist Riemannian
manifolds with boundary that are Dirichlet isospectral, and also
$\DtN_\Lambda$-isospectral for {\em any} real $\Lambda$ which is not a
Dirichlet eigenvalue. In particular, one can construct non-isometric
flat surfaces with this property which are embedded in
$\mathbb{R}^3$. Note that by the Robin--Dirichlet-to-Neumann duality
mentioned above, it also follows that these manifolds are Robin
isospectral for any choice of the Robin parameter.  We remark that if
the Robin parameter is nonzero, no Robin isospectral non-isometric
Euclidean domains are known.

\section{Examples}\label{sec:examples}

To gain some intuition, we mention several examples in which the
eigenvalues and eigenfunctions can be computed.

\subsection{Intervals}\label{sec:intervals}

For the interval $\Omega = \mathcal{I}_1: = \left(-\frac{1}{2},\frac{1}{2}\right) \subset \R$ of unit length, the
Dirichlet-to-Neumann operator is particularly simple.  A general
solution of the Helmholtz equation $-\frac{\dr^2}{\dr x^2}U^{(\Lambda)}(x)=\Lambda U^{(\Lambda)}(x)$, decomposed into symmetric and anti-symmetric (with respect to $x=0$) components, is
\[
U^{(\Lambda)}(x) = c_\sr U_\sr^{(\Lambda)}(x) + c_\ar U_\ar^{(\Lambda)}(x),
\]
where $c_\sr, c_\ar$ are arbitrary constants, and  
\begin{equation}\label{eq:U1d}
U_\sr^{(\Lambda)}(x) := \begin{cases}
\cosh(\sqrt{-\Lambda} x)\qquad&\text{if }\Lambda<0,\\
1\qquad&\text{if }\Lambda=0,\\
\cos(\sqrt{\Lambda} x)\quad&\text{if }\Lambda>0,
\end{cases}
\qquad
U_\ar^{(\Lambda)}(x) := \begin{cases}
\sinh(\sqrt{-\Lambda} x)\quad&\text{if }\Lambda<0,\\
x\quad&\text{if }\Lambda=0,\\
\sin(\sqrt{\Lambda} x)\quad&\text{if }\Lambda>0.
\end{cases}
\end{equation}

As the boundary  $\pa$
consists of two points $\left\{\pm \frac{1}{2}\right\}$, and the boundary conditions are 
\[
\mp \frac{\dr}{\dr x}U^{(\Lambda)}\left(\mp \frac{1}{2}\right)-\sigma U^{(\Lambda)}\left(\mp \frac{1}{2}\right)=0
\]  
(recall that the normal derivative $\partial_n$ is always taken with respect to the exterior normal), we easily deduce that the spectrum of $\DtN_\Lambda$ on $\mathcal{I}_1$ consists (generically) of the two eigenvalues
\begin{equation}\label{eq:sigmaLambda1d} 
\begin{gathered}
\sigma_\sr^{(\Lambda)} = f_\sr(\Lambda):=
\begin{cases}
\sqrt{-\Lambda}\tanh\left(\frac{\sqrt{-\Lambda}}{2}\right)\quad&\text{if }\Lambda<0,\\
0,\quad&\text{if }\Lambda=0,\\
-\sqrt{\Lambda}\tan\left(\frac{\sqrt{\Lambda}}{2}\right)\quad&\text{if }\Lambda>0,
\end{cases}
\\
\sigma_\ar^{(\Lambda)} = f_\ar(\Lambda):=
\begin{cases}
\sqrt{-\Lambda}\coth\left(\frac{\sqrt{-\Lambda}}{2}\right)\quad&\text{if }\Lambda<0,\\
2,\quad&\text{if }\Lambda=0,\\
\sqrt{\Lambda}\cot\left(\frac{\sqrt{\Lambda}}{2}\right)\quad&\text{if }\Lambda>0,
\end{cases}
\end{gathered}
\end{equation}
with the corresponding bulk eigenfunctions being  (up to multiplication by a constant) $U_\sr^{(\Lambda)}(x)$ and $U_\ar^{(\Lambda)}(x)$, respectively.

We show the behaviour of the eigenvalues as functions of $\Lambda$ in Figure \ref{fig:plot1dDtNevs}.

\begin{figure}[htb]
\centering
\includegraphics{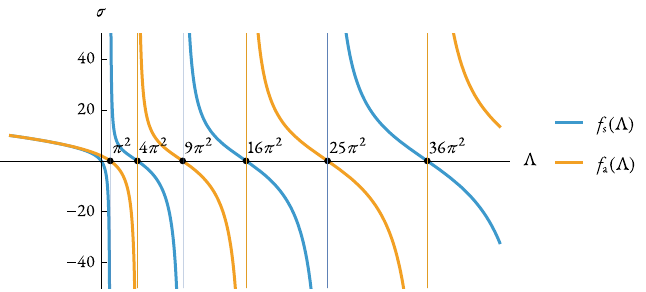}
\caption{Eigenvalues of the Dirichlet-to-Neumann map $\DtN_\Lambda$ for the interval $ \left(-\frac{1}{2},\frac{1}{2}\right)$ as functions of $\Lambda$.}
\label{fig:plot1dDtNevs}
\end{figure}

The following observations and comments are warranted. Before proceeding, let us re-label the eigenvalues of the Dirichlet Laplacian on $\mathcal{I}_1$, taking into account  symmetry/anti-symmetry of the corresponding eigenfunction, as 
\[
\lambda^\Dir_{\sr, m}(\mathcal{I}_1):= (2m-1)^2\pi^2,\qquad \lambda^\Dir_{\ar, m}(\mathcal{I}_1):= (2m)^2\pi^2,\qquad m\in\mathbb{N}.
\]
For further use, we formally denote
\begin{equation}\label{eq:lambda0}
\lambda^\Dir_{\sr, 0}(\mathcal{I}_1)=\lambda^\Dir_{\ar, 0}(\mathcal{I}_1):=-\infty.
\end{equation}

We first observe that when $\Lambda=\lambda^\Dir_{\oro, m}$, with some parity $\oro\in\{\sr,\ar\}$ and some $m\in\mathbb{N}$,  the corresponding Dirichlet-to-Neumann eigenvalue $\sigma^{(\Lambda)}_\oro$  blows up. At the same time, since these values of $\Lambda$, together with $\Lambda=0$, are also the Neumann Laplacian eigenvalues of the interval, by the Robin--Dirichlet-to-Neumann duality the other eigenvalue $\sigma^{(\Lambda)}_{-\oro}$  (with an opposite parity $-\sr:=\ar$ or $-\ar:=\sr$) vanishes.

Secondly, we have 
\[
\sigma_{\sr}^{(\Lambda)}\approx \sigma_{\ar}^{(\Lambda)}\approx \sqrt{-\Lambda}\qquad\text{as }\Lambda\to -\infty.  
\]

Thirdly,  note that on each interval of continuity $\left(\lambda^\Dir_{\oro, m-1}(\mathcal{I}_1), \lambda^\Dir_{\oro, m}(\mathcal{I}_1)\right)$, $m\in\mathbb{N}$ (with account of \eqref{eq:lambda0} if $m=1$), of the function $f_\oro(\Lambda)$,  this function is monotone decreasing from $+\infty$ to $-\infty$, and therefore has an inverse which we denote by
\[
f_{\oro, m}^{-1}: \mathbb{R}\to  \left(\lambda^\Dir_{\oro, m-1}(\mathcal{I}_1), \lambda^\Dir_{\oro, m}(\mathcal{I}_1)\right),\qquad \oro\in\{\sr,\ar\}, m\in\mathbb{N}.
\]  
Thus, the spectrum of the one-dimensional Robin Laplacian $-\Delta^{\Rob,\gamma}(\mathcal{I}_1)$ with $\gamma\in\mathbb{R}$ is given by
\[
\operatorname{Spec}\left(-\Delta^{\Rob,\gamma}(\mathcal{I}_1)\right)=\left\{f_{\oro, m}^{-1}(-\gamma), \oro\in\{\sr, \ar\}, m\in\mathbb{N}\right\}.
\]

Finally, if we consider an interval $\mathcal{I}_\alpha=\left(-\frac{\alpha}{2},\frac{\alpha}{2}\right)$ of length $\alpha$, the eigenvalues of the Dirichlet-to-Neumann operator can be expressed in terms of  \eqref{eq:sigmaLambda1d} using \eqref{eq:scaling}: they are
\[
\frac{1}{\alpha}f_\sr(\alpha^2\Lambda)\qquad\text{and}\qquad \frac{1}{\alpha}f_\ar(\alpha^2\Lambda).
\]
We preserve the notational convention \eqref{eq:lambda0} with $\mathcal{I}_\alpha$ replacing $\mathcal{I}_1$. 

We emphasise that this one-dimensional example, for which the spectrum of the Dirichlet-to-Neumann map 
consists of at most two eigenvalues, is very specific.  In the rest of the
paper we exclude such pathological situations and consider only bounded 
domains in $\R^d$ with $d \geq 2$. 

\subsection{Disk and balls}\label{sec:disk}

Let $\Omega = \mathbb{D}:=\{x\in \R^2 : |x|<1\}$ be the unit disk.  A
general solution of the Helmholtz equation $-\Delta U - \Lambda U =0$ in $\mathbb{D}$,  regular at the origin,  can be
written in polar coordinates $(r,\theta)$ as
\begin{equation}\label{eq:ULambdadisk}
U^{(\Lambda)}(r,\theta) = \sum_{m\in \Z} c_m \er^{\ir m\theta} F_{|m|}^{(\Lambda)}(r),
\end{equation}
where $c_m$ are constants,
\[
F_m^{(\Lambda)}(r) := 
\begin{cases}
I_{m}\left(\sqrt{-\Lambda}r\right)\qquad&\text{if }\Lambda<0,\\ 
r^m\qquad&\text{if }\Lambda=0,\\
J_{m}\left(\sqrt{\Lambda}r\right)\qquad&\text{if }\Lambda>0,
\end{cases}
\qquad m=0,1,2,\dots,
\]
and $J_m$ and $I_m$ are the Bessel and the modified Bessel functions, respectively, of the first kind of order $m$. 

Let $u(\theta)=\sum\limits_{m\in \Z} b_m \er^{\ir m\theta}$ be an $L^2(\partial\mathbb{D})$-function expanded into the Fourier series, with $b_m  := \frac{1}{\sqrt{2\pi}} \myscal{u, \er^{\ir m\theta}}_{L^2(\partial\mathbb{D})}$. Then $u\in H^{1/2}(\partial\mathbb{D})$ if and only if $\sum\limits_{m\in\Z} m b_m^2 <\infty$. It is easily checked that $\left.U^{(\Lambda)}\right|_{\partial\mathbb{D}}=u$ if we take $c_m=\frac{b_m}{F^{(\Lambda)}_{|m|}(1)}$ in  \eqref{eq:ULambdadisk}, and the Dirichlet-to-Neumann map applied to $u$ is
\[
\mathcal{D}_\Lambda u = \sum_{m\in \Z} b_m \er^{\ir m\theta} \frac{\left(F^{(\Lambda)}_{|m|}\right)'(1)}{F^{(\Lambda)}_{|m|}(1)}.
\]
It follows immediately that the eigenvalues of $\DtN_\Lambda$ are 
\begin{equation}\label{eq:sigmadisk}
\sigma^{(\Lambda)}_{(m)}:= \frac{\left(F^{(\Lambda)}_{m}\right)'(1)}{F^{(\Lambda)}_{m}(1)}=\begin{cases}
\frac{\sqrt{-\Lambda}I'_m(\sqrt{-\Lambda})}{I_m(\sqrt{-\Lambda})}\qquad&\text{if }\Lambda<0,\\ 
m\qquad&\text{if }\Lambda=0,\\
\frac{\sqrt{\Lambda}J'_m(\sqrt{\Lambda})}{J_m(\sqrt{\Lambda})}\qquad&\text{if }\Lambda>0,
\end{cases}
\qquad m=0,1,2,\dots,
\end{equation}
which are simple if $m=0$, with the corresponding constant eigenfunction $\frac{1}{\sqrt{2\pi}}$, and double if $m>0$, with the corresponding eigenspaces $\operatorname{Span}\left\{\frac{1}{\sqrt{2\pi}}  \er^{\pm \ir m\theta}\right\}= \operatorname{Span}\left\{\frac{\sin m\theta}{\sqrt{\pi}}, \frac{\cos m\theta}{\sqrt{\pi}}\right\}$. We show some of the curves $\sigma^{(\Lambda)}_{(m)}$ as functions of $\Lambda$ in Figure \ref{fig:plotDtNdisk}.

\begin{figure}[!htbp]
\centering
\includegraphics{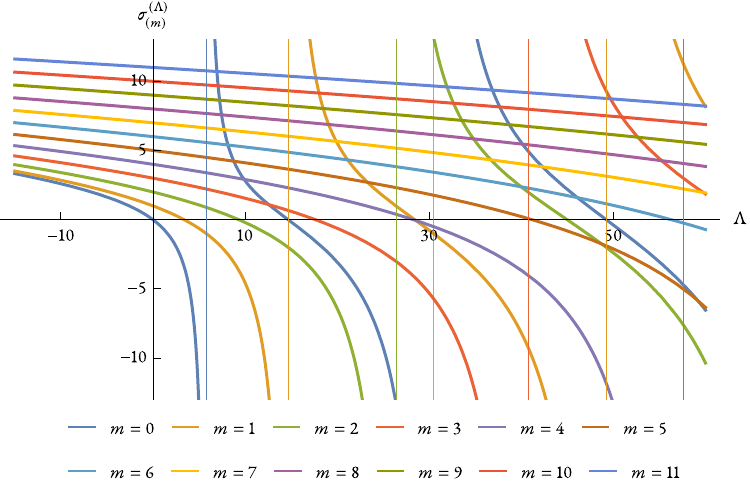}
\caption{Eigenvalues of the Dirichlet-to-Neumann map $\DtN_\Lambda$ for the unit disk as functions of $\Lambda$. The eigenvalue with $m=0$ is simple, all the others are double, except at the intersection points. The thin vertical lines indicate the positions of the Dirichlet eigenvalues. }
\label{fig:plotDtNdisk}
\end{figure}

\begin{remark}\label{rem:disk} 
Note that the expression \eqref{eq:sigmadisk} for $\sigma^{(\Lambda)}_{(m)}$ is invalid when $\sqrt{\Lambda}$ equals the $k$th positive zero $j_{m, k}$ of the Bessel function $J_m$. This happens since the eigenvalues of the Dirichlet Laplacian on $\mathbb{D}$  are exactly $j_{m, k}^2$ (simple ones for $m=0$ and double ones for $m>0$). 

The subscript $(m)$ in \eqref{eq:sigmadisk} does not indicate the position of the eigenvalue(s) in the spectrum but just refers to the corresponding eigenspace. Note that ordering eigenvalues in the non-decreasing order as $\sigma_1^{(\Lambda)}\le \sigma_2^{(\Lambda)}\le\dots$ is straightforward  for $\Lambda=0$: we have 
\[
\sigma_1^{(0)}=\sigma_{(0)}^{(0)}=0,\qquad \sigma_{2k-1}^{(0)}=\sigma_{2k}^{(0)}=\sigma_{(k)}^{(0)}=k,\qquad k\in\mathbb{N}.
\]
As shown in Appendix \ref{sec:ordering}, this
order of eigenvalues is preserved for any $\Lambda < \lambda_1^\Dir(\mathbb{D}) = j_{0,1}^2$: in this case we still have
\begin{equation}\label{eq:orderdisk}
\sigma_k^{(\Lambda)}=\sigma_{\left(\left\lfloor \frac{k+1}{2} \right\rfloor\right)}^{(\Lambda)},\qquad k\in\mathbb{N},
\end{equation}
where $\lfloor \cdot \rfloor$ denotes the integer part.

One can also see that the eigenfunctions of $\DtN_\Lambda$ for the disk  do not depend on
$\Lambda$; this situation is rather exceptional due to the rotational symmetry of the domain.
\end{remark}

Other examples of rotationally-invariant domains, for which the
eigenvalues and eigenfunctions of $\DtN_\lambda$ are known explicitly,
include a concentric annulus, a ball and a concentric spherical shell,
see \cite{Grebenkov20c}.  These explicit examples also rely on the
separation of variables in the Helmholtz equation.  For instance, for
the unit  ball $\Omega = \mathbb{B}^d=\{x\in\R^d :|x|<1\}$,  $d\ge 3$,
the eigenfunctions of  the Dirichlet-to-Neumann map $\DtN_\Lambda$ coincide with the eigenfunctions of the boundary Laplacian $-\Delta_{\mathbb{S}^{d-1}}$ (that is, they are restrictions of the harmonic homogeneous polynomials in $\R^{d}$ to the sphere $\mathbb{S}^{d-1}$), and the eigenvalues are
\[
\sigma^{(\Lambda)}_{(d, m)}:=\begin{cases}
\frac{\sqrt{-\Lambda}I'_{\frac{d}{2}-1+m}(\sqrt{-\Lambda})}{I_{\frac{d}{2}-1+m}(\sqrt{-\Lambda})}-\left(\frac{d}{2}-1\right)\qquad&\text{if }\Lambda<0,\\ 
m\qquad&\text{if }\Lambda=0,\\
\frac{\sqrt{\Lambda}J'_{\frac{d}{2}-1+m}(\sqrt{\Lambda})}{J_{\frac{d}{2}-1+m}(\sqrt{\Lambda})}-\left(\frac{d}{2}-1\right)\qquad&\text{if }\Lambda>0,
\end{cases}
\qquad m=0,1,2,\dots,
\]
where the multiplicity of $\sigma^{(\Lambda)}_{(d, m)}$ is $\binom{d+m-1}{d-1} - \binom{d+m-3}{d-1}$. 

\subsection{Rectangles and cuboids}\label{sec:cuboids}

The eigenvalues and eigenfunctions of the operator $\DtN_0$ of the Steklov problem in a square can be found semi-explicitly using separation of variables via solutions of some transcendental equations involving trigonometric and hyperbolic functions, see \cite{Girouard17} or  \cite[\S 7.1.2]{Levitin}. This method has been extended to the Steklov problem in cuboids in \cite{Girouard19}, and admits an easy generalisation to the case of an arbitrary $\Lambda\not\in\Spec\left(-\Delta^\Dir\right)$.

Let $\Omega=Q_{\boldsymbol{\alpha}}:=(-\alpha_1, \alpha_1)\times\dots\times(-\alpha_d, \alpha_d)\subset\mathbb{R}^d$, where  $\boldsymbol{\alpha}=(\alpha_1,\dots,\alpha_d)\in(0,+\infty)^d$. We are seeking solutions  of \eqref{eq:SteklovLambda} in $Q_{\boldsymbol{\alpha}}$ as products of one-dimensional Robin 
eigenfunctions in each direction, corresponding to unknown Robin eigenvalues $\Lambda_1,\dots,\Lambda_d\in\mathbb{R}$, that is, we set
\begin{equation}\label{eq:Ucuboid}
U(x_1,\dots,x_d)=\prod_{j=1}^d U^{(\alpha_j^2\Lambda_j)}_{\oro_j}\left(\frac{x_j}{\alpha_j}\right),\qquad \boldsymbol{\aleph}=(\oro_1,\dots,\oro_d)\in\{\sr, \ar\}^d,
\end{equation}
where $U^{(\alpha_j^2\Lambda_j)}_{\oro_j}$ are given by \eqref{eq:U1d}. For every choice of parameters $\oro_j$ (there are $2^d$ such choices altogether) this yields the following system of equations for the eigenvalues $\sigma^{(\Lambda)}$ of $\DtN_\Lambda$ and unknown parameters $\Lambda_1,\dots,\Lambda_d$:
\begin{equation}\label{eq:sigmacuboids}
\left\{
\begin{aligned}
\sigma^{(\Lambda)}&=\frac{1}{\alpha_1}f_{\oro_1}(\alpha_1^2\Lambda_1)=\dots=\frac{1}{\alpha_d}f_{\oro_d}(\alpha_d^2\Lambda_d),\\
\Lambda&=\Lambda_1+\dots+\Lambda_d,
\end{aligned}\right.
\end{equation}
where $f_{\oro}$, $\oro\in\{\sr,\ar\}$, are  defined by \eqref{eq:sigmaLambda1d}. We note that the same argument as in \cite{Girouard19} (using the Robin--Dirichlet-to-Neumann duality, cf.\ also \cite{Laugesen19}) allows one to conclude that \emph{all} solutions of \eqref{eq:SteklovLambda} in $Q_{\boldsymbol{\alpha}}$  can be represented as linear combinations of those in the form \eqref{eq:Ucuboid}.

We shall briefly describe the procedure for solving the system of equations \eqref{eq:sigmacuboids}. Set, for brevity, $\sigma=\sigma^{(\Lambda)}$, and assume that the vector of  parity  parameters $\boldsymbol{\aleph}=(\oro_1,\dots,\oro_d)\in\{\sr, \ar\}^d$ is fixed. Let us choose additionally a vector $\mathbf{m}=(m_1, \dots, m_d)\in \mathbb{N}^d$, and let us look for additional unknowns $\Lambda_1, \dots,\Lambda_d$ satisfying 
\[
\lambda_{\oro_j, m_j-1}^\Dir(\mathcal{I}_1)<\alpha_j \Lambda_j<\lambda_{\oro_j, m_j}^\Dir(\mathcal{I}_1), \qquad j=1,\dots,d.
\]
Then the first line of equations in \eqref{eq:sigmacuboids} yields
\[
\Lambda_j = \frac{1}{\alpha_j^2}f_{\oro_j, m_j}^{-1}\left(\alpha_j \sigma\right), 
\]
and substituting into the second line we get
\begin{equation}\label{eq:Lambdaimpl}
\Lambda = \sum_{j=1}^d \frac{1}{\alpha_j^2}f_{\oro_j, m_j}^{-1}\left(\alpha_j \sigma\right).
\end{equation}

Let us denote the right-hand side of \eqref{eq:Lambdaimpl} by $g_{\boldsymbol{\aleph},\mathbf{m}}(\sigma)$. It is easily seen from the properties of functions $f_{\oro, m}^{-1}$ discussed in \S\ref{sec:intervals} that $g_{\boldsymbol{\aleph},\mathbf{m}}$ is a monotone decreasing function on $\mathbb{R}$ with 
\[
\lim_{\sigma\to-\infty}g_{\boldsymbol{\aleph},\mathbf{m}}(\sigma) = \sum_{j=1}^d \frac{1}{\alpha_j^2} \lambda_{\oro_j, m_j}^\Dir(\mathcal{I}_1) =: \lambda_{\boldsymbol{\aleph}, \mathbf{m}}^\Dir(Q_{\boldsymbol{\alpha}}) 
\]
(that is, it tends to a Dirichlet eigenvalue of $Q_{\boldsymbol{\alpha}}$ as $\sigma\to-\infty$), and
\[
\lim_{\sigma\to\infty}g_{\boldsymbol{\aleph},\mathbf{m}}(\sigma) = \sum_{j=1}^d \frac{1}{\alpha_j^2} \lambda_{\oro_j, m_j-1}^\Dir(\mathcal{I}_1) = \lambda_{\boldsymbol{\aleph}, \mathbf{m}-\mathbf{1}}^\Dir(Q_{\boldsymbol{\alpha}}),
\]
where $\mathbf{1}:=(1,\dots,1)$, and we formally set  
\[
\lambda_{\boldsymbol{\aleph}, \mathbf{m}-\mathbf{1}}^\Dir(Q_{\boldsymbol{\alpha}}):=-\infty\qquad \text{if }\min_{j\in\{1,\dots,d\}} m_j = 1 
\]
(that is, $g_{\boldsymbol{\aleph},\mathbf{m}}(\sigma)$ tends to either another Dirichlet eigenvalue of $Q_{\boldsymbol{\alpha}}$ or to $-\infty$ as $\sigma\to+\infty$). Therefore, the inverse function
\[
g_{\boldsymbol{\aleph},\mathbf{m}}^{-1}: \left(\lambda_{\boldsymbol{\aleph}, \mathbf{m}-\mathbf{1}}^\Dir(Q_{\boldsymbol{\alpha}}), \lambda_{\boldsymbol{\aleph}, \mathbf{m}}^\Dir(Q_{\boldsymbol{\alpha}})\right)\to\mathbb{R}
\]
is well-defined, and we conclude from \eqref{eq:Lambdaimpl} that the eigenvalues $\sigma=\sigma^{(\Lambda)}$ of $\DtN_\Lambda$ are the values 
\[
g_{\boldsymbol{\aleph},\mathbf{m}}^{-1}(\Lambda),
\]
where $\boldsymbol{\aleph}$ and $\mathbf{m}$ are all the possible choices such that $\Lambda\in \left(\lambda_{\boldsymbol{\aleph}, \mathbf{m}-\mathbf{1}}^\Dir(Q_{\boldsymbol{\alpha}}), \lambda_{\boldsymbol{\aleph}, \mathbf{m}}^\Dir(Q_{\boldsymbol{\alpha}})\right)$.

\begin{example} Some eigenvalues of the Dirichlet-to-Neumann map for the square $\Omega=\left(-\frac{\pi}{2}, \frac{\pi}{2}\right)^2$, computed according to the algorithm above, are shown in  Figure  \ref{fig:plotDtNsquare}.
\begin{figure}[!htbp]
\centering
\includegraphics{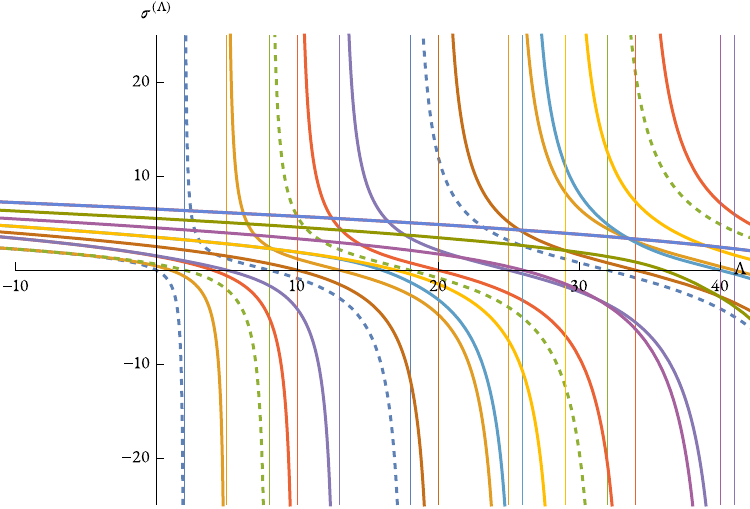}
\caption{Some eigenvalues $\sigma^{(\Lambda)}$, plotted as functions of $\Lambda$, for the square of side $\pi$. The dashed curves correspond to simple eigenvalues, and the solid curves to the double ones, except at the points of intersection. The thin vertical lines indicate the positions of the Dirichlet eigenvalues.}
\label{fig:plotDtNsquare}
\end{figure}
\end{example}

\begin{example} Unlike the case of the disk, cf. Remark \ref{rem:disk} and Appendix \ref{sec:ordering}, different eigenvalue branches for a cuboid can intersect even for $\Lambda<0$, see Figure \ref{fig:crossrect}.
\begin{figure}[!htbp]
\centering
\includegraphics{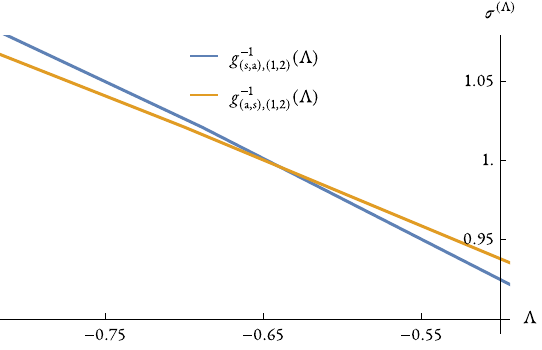}
\caption{An example of two analytic eigenvalue branches  for the rectangle $Q_{\left(\frac{\pi}{2}, \frac{27\pi}{16}\right)}$ intersecting at $\Lambda\approx -0.65$.}
\label{fig:crossrect}
\end{figure}
\end{example}

\subsection{A numerical example}

For illustrative purposes, we show some plots of bulk eigenfunctions of the Dirichlet-to-Neumann map for the asymmetric kite $\mathcal{K}\subset\mathbb{R}^2$ whose boundary  is given parametrically by
\[
\partial\mathcal{K} := \left\{\left(1.5 \cos t + 0.7 \cos 2t - 0.4, 1.5 \sin t - 0.3 \cos t\right) , t \in  [0, 2\pi)\right\},
\]
see Figure \ref{fig:kite}. The eigenvalues and eigenfunctions are computed numerically using the finite-element method, see \S\ref{sec:FEM}. For other illustrative numerical examples, see \cite{Chaigneau24} and \cite{Nigam25}.

\begin{figure}[!htbp]
\centering
\includegraphics{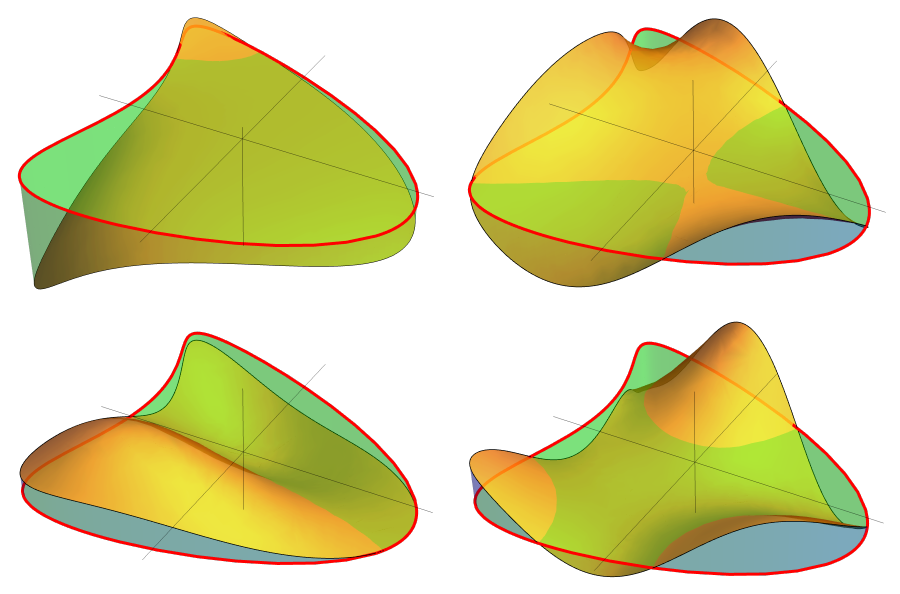}
\caption{Some bulk eigenfunctions for the kite $\mathcal{K}$. Top row, left to right, $\Lambda=-5$: the eigenfunctions $U^{(-5)}_2$ corresponding to the eigenvalue $\sigma^{(-5)}_2\approx 1.743$, and $U^{(-5)}_6$ corresponding to the eigenvalue $\sigma^{(-5)}_6\approx 2.740$. Bottom row, left to right, $\Lambda=5\in \left(\lambda_1^\Dir(\mathcal{K}), \lambda_2^\Dir(\mathcal{K})\right)$: the eigenfunctions $U^{(5)}_2$ corresponding to the eigenvalue $\sigma^{(5)}_2\approx -3.344$, and $U^{(5)}_{6}$ corresponding to the eigenvalue $\sigma^{(5)}_{6}\approx 0.784$.}
\label{fig:kite}
\end{figure}

\section{Spectral properties of $\DtN_\Lambda$}\label{sec:spectralproperties}
\subsection{Dependence of the eigenvalues  on the parameter $\Lambda$}\label{sec:lambda}

The general picture of the dependence of the eigenvalues of $\DtN_\Lambda$  on the parameter $\Lambda$ is similar to that in the specific case of the disk shown in Figure \ref{fig:plotDtNdisk}. Namely, we have the following 

\begin{theorem}\label{thm:lambdadep} Let $\Omega\subset \mathbb{R}^d$ be a bounded Lipschitz domain. The eigenvalues $\sigma_k^{(\Lambda)}$ of the Dirichlet-to-Neumann map $\DtN_\Lambda$ are strictly monotone decreasing functions of $\Lambda$ on any interval of the real line not containing the points of the Dirichlet spectrum $\Spec\left(-\Delta^\Dir_\Omega\right)$. If $\lambda^\Dir$ is a Dirichlet eigenvalue of multiplicity $m$, then the first $m$ eigenvalues of $\DtN_\Lambda$ tend to $-\infty$ as $\Lambda\to\left(\lambda^\Dir\right)^-$. Moreover, zero is in the spectrum of $\DtN_\Lambda$ if and only if $\Lambda\in \Spec\left(-\Delta^\Neu_\Omega\right)$.
\end{theorem}

Theorem \ref{thm:lambdadep} was first proved by L. Friedlander in the smooth case \cite{Friedlander91} and later extended to the Lipschitz case in \cite{Arendt12}. The proof essentially relies on the Robin--Dirichlet-to-Neumann duality and the monotonicity properties of Robin eigenvalues.

It should be noted from Figures \ref{fig:plotDtNdisk} and \ref{fig:plotDtNsquare} that the eigenvalues $\sigma_k^{(\Lambda)}$ (enumerated in increasing order for each $\Lambda$) are in general only continuous in $\Lambda$. However, for the disk the eigenvalue curves $\sigma_{(n)}^{(\Lambda)}$ (no longer ordered increasingly) are real-analytic whenever they are defined. This is a general property which holds if we agree to forsake the ordering of eigenvalues in increasing order.

\begin{theorem}\label{thm:analyticity}  For any bounded Lipschitz domain $\Omega\subset \mathbb{R}^d$, the union of spectra $\bigcup\limits_{\Lambda\in\mathbb{R}} \Spec\left(\DtN_\Lambda\right)$  can be decomposed into 
a family of real-analytic curves, each of them defined  either on a semi-infinite interval $\left(-\infty, \lambda^\Dir_k\right)$ or on a finite interval $\left(\lambda^\Dir_j, \lambda^\Dir_k\right)$ of the real line. 
If $\lambda_k^\Dir$ is an eigenvalue of the Dirichlet Laplacian of multiplicity $m$, then exactly $m$ analytic eigenvalue curves of $ \Spec\left(\DtN_\Lambda\right)$ tend to $-\infty$ as $\Lambda\to \left(\lambda_k^\Dir\right)^-$ and tend to $+\infty$ as $\Lambda\to \left(\lambda_k^\Dir\right)^+$.\end{theorem}

We postpone the proof of Theorem \ref{thm:analyticity} until \S\ref{sec:DtNDNL}.

\begin{remark} In practice, identifying individual analytic branches of eigenvalues, in particular in numerical examples, is highly non-trivial as one has to distinguish actual crossings of branches from avoided ``near-crossings",
see Figure \ref{fig:cross}. In reality, we are only able to fully identify all the branches for balls (see \S\ref{sec:disk}), annular domains, and cuboids (see \S\ref{sec:cuboids}). 
\end{remark}

\begin{figure}[!htbp]
\centering
\includegraphics{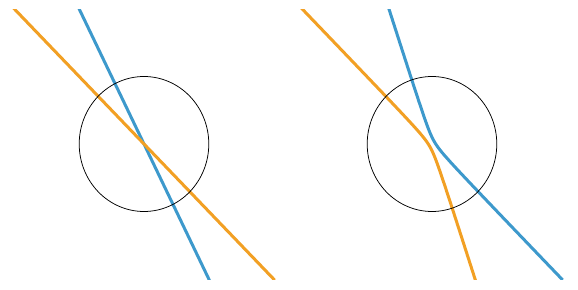}
\caption{A crossing (left) versus an avoided crossing (right).}
\label{fig:cross}
\end{figure}

We additionally have

\begin{proposition}\label{prop:dLambda} Let $\Omega\subset\mathbb{R}^d$ be a bounded Lipschitz domain, let $\Lambda_0\not\in\Spec\left(-\Delta^\Dir\right)$, and let us fix $k\in\mathbb{N}$. If $\sigma_{k}^{(\Lambda_0)}$ is a simple eigenvalue of $\DtN_{\Lambda_0}$, then
\begin{equation}\label{eq:firstder}
\left.\frac{\dr \sigma_k^{(\Lambda)}}{\dr \Lambda}\right|_{\Lambda=\Lambda_0}=-\left\|U_k^{(\Lambda_0)}\right\|^2_{L^2(\Omega)},
\end{equation}
and
\begin{equation}\label{eq:secondder}
\left.\frac{\dr^2 \sigma_k^{(\Lambda)}}{\dr \Lambda^2}\right|_{\Lambda=\Lambda_0}=-2\left(\myscal{\left(-\Delta^\Dir-\Lambda_0\right)^{-1}U_k^{(\Lambda_0)}, U_k^{(\Lambda_0)}}_{L^2(\Omega)}+\sum_{\substack{j\in\mathbb{N}\\j\ne k}} \frac{\myscal{U_j^{(\Lambda_0)}, U_k^{(\Lambda_0)}}^2_{L^2(\Omega)}}{\sigma_j^{(\Lambda_0)}-\sigma_{k}^{(\Lambda_0)}}\right).
\end{equation}
\end{proposition}

The proof of Proposition \ref{prop:dLambda} can be found in Appendix \ref{sec:proofdLambda}. 

\begin{remark}\label{rem:dLambda}
Formula \eqref{eq:firstder} is well-known, see, e.g., \cite[Lemma 2.3]{Friedlander91}; however, we have failed to find \eqref{eq:secondder} in full generality in the literature, although an alternatively presented formula in the case $\Lambda_0=0$ can be found in \cite[Appendix B]{Grebenkov19c}. The results of Proposition \ref{prop:dLambda} remain valid for $\Lambda_0\in \Spec\left(-\Delta^\Dir\right)$ with minor notational adjustments and can be sometimes extended to the case of multiple eigenvalues (for example, for disks and balls where eigenfunctions do not depend on $\Lambda$). The first term in brackets in the right-hand side of  \eqref{eq:secondder} (the quadratic form of the Dirichlet resolvent) can be also written as 
\[
\left\|\left(-\Delta^\Dir-\Lambda_0\right)^{-1/2}U_k^{(\Lambda_0)}\right\|^2_{L^2(\Omega)}=\sum_{m=1}^\infty \frac{\myscal{U_k^{(\Lambda_0)}, U_m^\Dir}^2_{L^2(\Omega)}}{\lambda_m^\Dir-\Lambda_0}.
\]
Using this, we immediately conclude from \eqref{eq:secondder} that the principal eigenvalue $\sigma_1^{(\Lambda)}$ of $\DtN_\Lambda$ is a concave function of $\Lambda$ for $\Lambda\in\left(-\infty, \lambda_1^\Dir\right)$. 
\end{remark}

Theorem \ref{thm:lambdadep} has a number of important consequences. Let us introduce, for a self-adjoint semi-bounded below operator $A$ with a discrete spectrum, the \emph{eigenvalue counting function}
\begin{equation}\label{eq:NlambdaA}
\mathcal{N}_A(L):=\#\left(\Spec(A)\cap(-\infty, L]\right).
\end{equation}
We immediately have
\begin{corollary}[{\cite[Lemma 1.2]{Friedlander91}, \cite[Proposition 4]{Arendt12}}]\label{cor:Ndiff} 
Let $\Omega\subset\mathbb{R}^d$ be a bounded Lipschitz domain, and let $\Lambda\in\mathbb{R}$. The number of non-positive eigenvalues of the Dirichlet-to-Neumann map $\DtN_\Lambda$ equals the difference of the eigenvalue counting functions of the Neumann and Dirichlet Laplacians in $\Omega$ evaluated at $\Lambda$:
\begin{equation}\label{eq:NDtN0}
\mathcal{N}_{\DtN_\Lambda}(0)=\mathcal{N}_{-\Delta^\Neu}(\Lambda)-\mathcal{N}_{-\Delta^\Dir}(\Lambda).
\end{equation}
\end{corollary}

In addition, we also have the following 
\begin{corollary}[{\cite{Friedlander91}, \cite{Arendt12}}]\label{cor:Nneg}
Let $\Omega\subset\mathbb{R}^d$, $d\ge 2$, be a bounded Lipschitz domain. Then the principal eigenvalue $\sigma_1^{(\Lambda)}$ of the Dirichlet-to-Neumann map is positive for $\Lambda<0$ and negative for $\Lambda>0$.  
\end{corollary}

\begin{proof} The first statement follows directly from Corollary \ref{cor:Ndiff}: if $\Lambda<0$, then $\mathcal{N}_{\DtN_\Lambda}(0)=0$. To prove the second statement, 
assume that $\Lambda\notin\Spec\left(-\Delta^\Dir_\Omega\right)$ and use the variational principle \eqref{eq:muk_minimaxk1} with the test function $U=\er^{\ir\mydotp{\omega, x}}$, where $\omega\in\mathbb{R}^d$ and $|\omega|^2=\Lambda$, so that $U\in\mathcal{H}_\Lambda(\Omega)$, which yields $\sigma_1^{(\Lambda)}\le [U,U]_\Lambda=0$. Since $\partial_n U=\ir \mydotp{\omega, n}U$ cannot vanish identically on the boundary, it follows that $\sigma_1^{(\Lambda)}<0$. The case when $\Lambda$ is a Dirichlet eigenvalue is dealt with similarly, with some minor modifications.   
\end{proof} 

Furthermore, we can quantify the Robin--Dirichlet-to-Neumann duality of Proposition \ref{prop:DtNRduality}, following  \cite[Proposition 5]{Arendt12} and \cite[Proposition 2.7]{Hassannezhad22}. 

\begin{proposition}\label{prop:dualityfurther} 
Let $\Omega\subset\mathbb{R}^d$, $d\ge 2$, be a bounded Lipschitz domain.  Assume that $\Lambda\not\in\Spec\left(-\Delta^\Dir_\Omega\right)$, and let $m:=\mathcal{N}_{-\Delta^\Dir_\Omega}(\Lambda)$ be the number of the Dirichlet eigenvalues less than $\Lambda$. Let also $\sigma=\sigma_k^{(\Lambda)}(\Omega)$, $k\in\mathbb{N}$. Then $\Lambda=\lambda^{\Rob, -\sigma}_{k+m}(\Omega)$.
\end{proposition}

Similarly to Corollary \ref{cor:Nneg}, we can obtain the expression for a number of eigenvalues of $\DtN_\Lambda$ in any fixed interval $(S_1, S_2]$ of the real line, this time expressed in terms of the Robin counting functions. 

\begin{corollary}
Let $\Omega\subset\mathbb{R}^d$ be a bounded  Lipschitz domain, and let $\Lambda\in\mathbb{R}$ and $S_1<S_2$. Then
\begin{equation}\label{eq:NDtNint}
\#\left(\Spec\left(\DtN_\Lambda\right)\cap(S_1, S_2]\right)=\mathcal{N}_{-\Delta^{\Rob,-S_2}}(\Lambda)-\mathcal{N}_{-\Delta^{\Rob,-S_1}}(\Lambda).
\end{equation}
\end{corollary}

We refer to \cite{Hassell17} for some related results concerning the asymptotics of  \eqref{eq:NDtNint} as $\Lambda \to +\infty$, with $S_j = s_j \Lambda$, and $s_1<s_2$ fixed.  See also \cite{Rudnick} for related bounds on the gaps between the corresponding Robin and Neumann eigenvalues.

Note also that \eqref{eq:NDtNint} formally yields  \eqref{eq:NDtN0} if we take in the former $S_1=-\infty$ and $S_2=0$ and use the facts that $-\Delta^{\Rob,0}=-\Delta^\Neu$ and $-\Delta^{\Rob,+\infty}=-\Delta^\Dir$.
Corollaries \ref{cor:Ndiff} and \ref{cor:Nneg} imply that for any bounded domain $\Omega\subset\mathbb{R}^d$, $d\ge 2$, and any $k\in\mathbb{N}$, we have 
\[
\lambda^\Neu_{k+1}(\Omega)<\lambda^\Dir_k(\Omega),
\]
see \cite{Filonov2003} for an alternative proof not involving the Dirichlet-to-Neumann map. This inequality can be improved to 
$\lambda^\Neu_{k+d}<\lambda^\Dir_k$ for \emph{convex} bounded domains in $\mathbb{R}^d$ \cite{Levine1986}, as well as for all simply-connected domains in the case $d=2$ \cite{Rohleder}. It is conjectured (see \cite[Conjecture 3.2.42]{Levitin}) that such an improvement holds for all bounded Euclidean domains.

\subsection{Dirichlet-to-Neumann map and the Dirichlet Laplacian}\label{sec:DtNDNL}

There are further relations between the Dirichlet-to-Neumann maps and the Laplacian with Dirichlet (or Neumann) boundary conditions  in $\Omega\subset\mathbb{R}^d$. In particular, we  
have the following

\begin{theorem}[{\cite{Behrndt15}}]\label{thm:DtNdiff}
Let $\Omega$ be  a bounded Lipschitz domain,  $\Lambda, \Lambda_0\in\mathbb{R}\setminus\Spec\left(-\Delta^\Dir\right)$, and $u\in H^{1/2}(\pa)$. Then
\begin{equation}\label{eq:Ddiff}
\left(\mathcal{D}_{\Lambda}-\mathcal{D}_{\Lambda_0}\right)u=(\Lambda - \Lambda_0)
\partial_n \left(\left(-\Delta^\Dir-\Lambda\right)^{-1}\mathcal{E}_{\Lambda_0}u \right).
\end{equation}
\end{theorem}

\begin{proof} Let us denote $U:=\mathcal{E}_{\Lambda} u$, $U_0:=\mathcal{E}_{\Lambda_0} u$,  and $W:=U-U_0$. Then $\left(\mathcal{D}_{\Lambda}-\mathcal{D}_{\Lambda_0}\right)u=\partial_n W$. 
At the same time, $W$ satisfies
\begin{equation}\label{eq:Wdiff}
\begin{cases}
-\Delta W -\Lambda W = (\Lambda-\Lambda_0) U_0\qquad&\text{in }\Omega,\\
W=0\qquad&\text{on }\pa.
\end{cases}
\end{equation}
Therefore $W= (\Lambda-\Lambda_0)\left(-\Delta^\Dir-\Lambda\right)^{-1}U_0$, and \eqref{eq:Ddiff} follows. 
\end{proof}

We refer to \cite{Behrndt15} for  further generalisations, including the removal of the condition $\Lambda, \Lambda_0\not\in\Spec\left(-\Delta^\Dir\right)$.

Additionally, there is an important class of results, known as \emph{Krein-type resolvent formulae}, which relate the difference $(-\Delta^\Neu-\Lambda)^{-1}-(-\Delta^\Dir-\Lambda)^{-1}$ of the resolvents of the Neumann and the Dirichlet Laplacian with the Dirchlet-to-Neumann map $\DtN_\Lambda$, its inverse (the \emph{Neumann-to-Dirichlet map}) $\DtN_\Lambda^{-1}$, the $\Lambda$-harmonic extension operator $\mathcal{E}_\Lambda$ and its Neumann analogue, see \cite{Behrndt15} and references therein.

Let $\Omega\subset\mathbb{R}^d$ be a bounded domain with Lipschitz boundary. Suppose we know \emph{all} the eigenvalues $\sigma_j^{(\Lambda_0)}$ and orthonormalised eigenfunctions, denoted for brevity $v_j:=u_j^{(\Lambda_0)}$, of the Dirichlet-to-Neumann map $\DtN_{\Lambda_0}$ for some $\Lambda_0\in\mathbb{R}\setminus\Spec\left(-\Delta^\Dir_\Omega\right)$. We aim to write down \emph{explicitly}, for any $\Lambda\in\mathbb{R}$,  the action of the Dirichlet-to-Neumann map $\DtN_{\Lambda}$ in the basis $\left\{v_j\right\}_{j=1}^\infty$, that is, to compute the semi-infinite matrix 
\[
\mathtt{D}_\Lambda := \left(\mathtt{d}_{\Lambda; i,j}\right)_{i,j=1}^\infty:=\left(\myscal{\DtN_\Lambda v_i, v_j}_{L^2(\pa)}\right)_{i,j=1}^\infty.
\]
It turns out that we can do this explicitly, if we additionally assume that we know all the eigenvalues $\lambda_k^\Dir$ and the corresponding eigenfunctions $U_k^\Dir$ of the Dirichlet Laplacian $-\Delta^\Dir_\Omega$.
The result additionally separates, in a sense, the dependence of $\mathtt{D}_\Lambda$ on the geometry and on varying values of $\Lambda$.

Let us define the semi-infinite matrices
\[
\mathtt{A}:= \left(\mathtt{a}_{i,k}\right)_{i,k=1}^\infty:=\left(\myscal{v_i, \partial_n U_k^\Dir}_{L^2(\partial\Omega)}\right)_{i,k=1}^\infty,
\]
and
\[
\mathtt{B}_\Lambda:= \left(\mathtt{b}_{\Lambda;k}\delta_{i,k}\right)_{i,k=1}^\infty:=\left(\frac{\Lambda_0-\Lambda}{\left(\Lambda_0-\lambda_k^\Dir\right)\left(\Lambda-\lambda_k^\Dir\right)}\delta_{i, k}\right)_{i, k=1}^\infty,
\]
the latter being diagonal.

\begin{theorem}\label{thm:Dmatrix} 
We have
\begin{equation}\label{eq:Dmatrix}
\mathtt{D}_\Lambda =\mathtt{D}_{\Lambda_0}+\mathtt{A}\mathtt{B}_\Lambda \mathtt{A}^{\mathtt{t}},
\end{equation}
where
\[
\mathtt{D}_{\Lambda_0}= \left(\mathtt{d}_{\Lambda_0; i,j}\right)_{i,j=1}^\infty=\left( \sigma_i^{(\Lambda_0)}\delta_{i,j}\right)_{i,j=1}^\infty
\]
is diagonal, and $\mathtt{A}^\mathtt{t}$ is the transpose of $\mathtt{A}$.
\end{theorem}

\begin{proof} 
We will show that for all $i,j=1,2,\dots$, 
\begin{equation}\label{eq:Dmatrixentry}
\mathtt{d}_{\Lambda; i,j} - \mathtt{d}_{\Lambda_0; i,j} = \sum_{k=1}^\infty \mathtt{a}_{i,k}\mathtt{b}_{\Lambda;k}\mathtt{a}_{j,k}
=(\Lambda_0-\Lambda)\sum_{k=1}^\infty\frac{\myscal{v_i, \partial_n U_k^\Dir}_{L^2(\partial\Omega)}\myscal{v_j, \partial_n U_k^\Dir}_{L^2(\partial\Omega)}}{\left(\Lambda_0-\lambda_k^\Dir\right)\left(\Lambda-\lambda_k^\Dir\right)},
\end{equation}
which is the entry-wise form of \eqref{eq:Dmatrix}.

For brevity, denote the bulk eigenfunctions of $\DtN_{\Lambda_0}$ by 
\[
V_i:=U^{(\Lambda_0)}_i=\mathcal{E}_{\Lambda_0} v_i.
\]

We apply Theorem \ref{thm:DtNdiff} to $u=v_i$, then \eqref{eq:Ddiff} reads
\begin{equation}\label{eq:Ddiff1}
\left(\mathcal{D}_{\Lambda}-\mathcal{D}_{\Lambda_0}\right) v_i =\partial_n W,
\end{equation}
where
$W:=(\Lambda-\Lambda_0)\left(-\Delta^\Dir-\Lambda\right)^{-1}V_i$ satisfies \eqref{eq:Wdiff} in which we take $U_0=V_i$ in the right-hand side.

We multiply both sides of \eqref{eq:Ddiff1} by $v_j$ in $L^2(\pa)$, and integrate by parts on the right, which leads to 
\begin{equation}\label{eq:Ddiff2}
\myscal{\left(\mathcal{D}_{\Lambda}-\mathcal{D}_{\Lambda_0}\right) v_i, v_j}_{L^2(\pa)} = \myscal{\partial_n W,  v_j}_{L^2(\pa)} = \myscal{\Delta W, V_j}_{L^2(\Omega)} -  \myscal{W, \Delta V_j}_{L^2(\Omega)}, 
\end{equation}
where we have taken into account the fact that $\left.W\right|_{\pa}=0$.  By \eqref{eq:Wdiff} and the standard equation for the bulk eigenfunction $V_j$, we have
\[
\Delta W = -\Lambda W - (\Lambda-\Lambda_0)V_i, \qquad \Delta V_j = - \Lambda_0 V_j,
\]
which after substitution into the right-hand side of \eqref{eq:Ddiff2} and evaluation of its left-hand side gives
\begin{equation}\label{eq:Ddiff3}
\mathtt{d}_{\Lambda; i,j} - \mathtt{d}_{\Lambda_0; i,j} = -(\Lambda-\Lambda_0)\myscal{W+V_i, V_j}_{L^2(\Omega)}.
\end{equation}
We now expand all the functions in the right-hand side of  \eqref{eq:Ddiff3} in the $L^2(\Omega)$ basis of the Dirichlet eigenfunctions, using also the standard resolvent representation 
\[
W=(\Lambda-\Lambda_0)\left(-\Delta^\Dir-\Lambda\right)^{-1}V_i = \sum_{k=1}^\infty \frac{\Lambda-\Lambda_0}{\lambda_k^\Dir -\Lambda}\myscal{V_i, U_k^\Dir}_{L^2(\Omega)} U_k^\Dir
\]
and
\[
V_i =  \sum_{k=1}^\infty  \myscal{V_i, U_k^\Dir}_{L^2(\Omega)} U_k^\Dir, \qquad V_j=  \sum_{k=1}^\infty  \myscal{V_j, U_k^\Dir}_{L^2(\Omega)} U_k^\Dir,
\]
which after minimal simplifications yields 
\begin{equation}\label{eq:Ddiff4}
\mathtt{d}_{\Lambda; i,j} - \mathtt{d}_{\Lambda_0; i,j} = (\Lambda_0-\Lambda) \sum_{k=1}^\infty \frac{\lambda_k^\Dir -\Lambda_0}{\lambda_k^\Dir -\Lambda}  \myscal{V_i, U_k^\Dir}_{L^2(\Omega)}  \myscal{V_j, U_k^\Dir}_{L^2(\Omega)}.
\end{equation}
Finally, by Proposition \ref{prop:Eexpand},
\[
\begin{split}
\myscal{V_i, U_k^\Dir}_{L^2(\Omega)} &= \frac{\myscal{v_i, \partial_n U^\Dir_k}_{L^2(\pa)}}{\Lambda_0-\lambda_k^\Dir}=\frac{\mathtt{a}_{i,k}}{\Lambda_0-\lambda_k^\Dir},\\
\myscal{V_j, U_k^\Dir}_{L^2(\Omega)} &= \frac{\myscal{v_j, \partial_n U^\Dir_k}_{L^2(\pa)}}{\Lambda_0-\lambda_k^\Dir}=\frac{\mathtt{a}_{j,k}}{\Lambda_0-\lambda_k^\Dir},
\end{split}
\]
and the substitution into \eqref{eq:Ddiff4} proves \eqref{eq:Dmatrixentry} and therefore \eqref{eq:Dmatrix}.
\end{proof}

\begin{example}
Let $\Omega=\mathbb{D}$ be the unit disk and let $\Lambda_0=0$. We fix the  standard basis of Steklov eigenfunctions    
$v_1=\frac{1}{\sqrt{2\pi}}$, $v_{2j}=\frac{\sin j\theta}{\sqrt{\pi}}$, $v_{2j+1}=\frac{\cos j\theta}{\sqrt{\pi}}$, $j\in\mathbb{N}$, and choose the basis of Dirichlet eigenfunctions with exactly the same angular dependencies. 
It is easy to check that in this case the matrix $\mathtt{D}_\Lambda$ is diagonal, which agrees with the fact that the eigenfunctions of $\DtN_\Lambda$ do not depend on $\Lambda$, see Remark \ref{rem:disk}. It is also relatively easy to see that the statement of Theorem \ref{thm:Dmatrix} is equivalent to the following identities
\[
k+2\Lambda \sum_{m=1}^\infty \frac{1}{\Lambda-j_{k,m}^2} = 
\begin{cases}
\frac{\sqrt{-\Lambda}I'_k(\sqrt{-\Lambda})}{I_k(\sqrt{-\Lambda})}\qquad&\text{if }\Lambda<0,\\[10pt]
\frac{\sqrt{\Lambda}J'_k(\sqrt{\Lambda})}{J_k(\sqrt{\Lambda})}\qquad&\text{if }\Lambda>0,
\end{cases}
\]
being valid for all $\Lambda\ne 0$ and all $k=0,1,2,\dots$. These are well known and are easily derived from \cite[\S15$\cdot$41, formula  (1)]{Watson}, see also \cite{Grebenkov21b}.
\end{example}

\begin{remark} 
A similar result representing the \emph{Neumann-to-Dirichlet map} $\DtN_\Lambda^{-1}$ in terms of Neumann eigenvalues and the boundary traces of Neumann eigenfunctions was obtained in \cite{LevMar}.
\end{remark}

\begin{remark} If $\Lambda$ is a Dirichlet eigenvalue of multiplicity $m$, $\Lambda=\lambda^\Dir_{\ell}=\dots=\lambda^\Dir_{\ell+m-1}$, then we additionally need to impose the orthogonality conditions $\myscal{u, \partial_n U^\Dir_k}_{L^2(\pa)}=0$, $k=\ell,\dots, \ell+m-1$ in the domain of $\DtN_\Lambda$. If $u=\sum_{i=1}^\infty c_i v_i \in \Dom(\DtN_\Lambda)$, the orthogonality conditions mean $\left(\mathtt{A}^\mathtt{t}\mathbf{c}\right)_k=0$, $k=\ell, \dots, \ell+m-1$, where $\mathbf{c}=(c_1, c_2,\dots)^\mathtt{t}$, thus removing the singular terms in $\mathtt{A}\mathtt{B}_\Lambda\mathtt{A}^\mathtt{t}\mathbf{c}$.
\end{remark}

We are now in the position to prove Theorem \ref{thm:analyticity}. 

\begin{proof}[Proof of Theorem \ref{thm:analyticity}]
It is immediately clear from Theorem \ref{thm:Dmatrix}  that $\mathtt{D}_{\Lambda}$, and therefore $\DtN_\Lambda$, is a type (A) analytic family of operators in the sense of Kato \cite{Kato} on every interval not including the Dirichlet eigenvalues of $\Omega$, and therefore the union of spectra of $\DtN_\Lambda$ can be decomposed into curves  which are real-analytic in $\Lambda$. The number of curves blowing up at a particular Dirichlet eigenvalue $\lambda_k^\Dir$ coincides with the multiplicity of that eigenvalue by  Theorem \ref{thm:lambdadep}, and the fact that the remaining curves are analytic at $\lambda_k^\Dir$ follows from the analyticity in $\gamma$ of the eigenvalue curves of the Robin Laplacian $-\Delta^{\Rob,\gamma}$ \cite{Bucur17} and the Robin--Dirichlet-to-Neumann duality of Proposition \ref{prop:DtNRduality}.
\end{proof}

As an additional consequence of Theorems \ref{thm:analyticity} and \ref{thm:Dmatrix}, we are able to give the positive answer to \cite[Open problem 4.11]{Bucur17}. 

\begin{theorem}\label{thm:Robincurves}
 Let $\Omega\subset\mathbb{R}^d$ be a bounded Lipschitz domain. Then every analytic branch $\lambda(\gamma)$ of eigenvalues of the Robin Laplacian $-\Delta^{\Rob,\gamma}_\Omega$ either tends to $-\infty$ or converges to a Dirichlet eigenvalue of $\Omega$ as $\gamma\to -\infty$.
\end{theorem}

\begin{proof} As the analytic  branches of eigenvalues of the Robin Laplacian are monotone increasing in $\gamma$, there exists a limit $L=\lim_{\gamma\to -\infty}  \lambda(\gamma)$ along a chosen branch, which is either finite or $-\infty$. Suppose, in contradiction with the statement of theorem, that there exists a branch for which $L$ is finite and $L\not\in \Spec\left(-\Delta^\Dir_\Omega\right)$. Then, by  Robin--Dirichlet-to-Neumann duality, there exists a branch $\sigma^{(\Lambda)}$ of eigenvalues of the Dirichlet-to-Neumann map which is analytic only in an interval $\left(L, \lambda_k^\Dir\right)$, with some $k\in\mathbb{N}$, which contradicts Theorem \ref{thm:analyticity}. Thus, either $L=-\infty$ or $L\in \Spec\left(-\Delta^\Dir_\Omega\right)$.
\end{proof}

\subsection{Domain monotonicity}\label{sec:dommon}

The eigenvalues of the Dirichlet Laplacian exhibit the domain
monotonicity property: if $\Omega_1 \subset \Omega_2$, then
$\lambda_k^\Dir(\Omega_1) \geq \lambda_k^\Dir(\Omega_2)$ for all $k\in\mathbb{N}$.  This property is
known to fail for the Neumann and Robin Laplacians.  Likewise, it does
not hold for the Steklov problem, i.e., the inclusion $\Omega_1
\subset \Omega_2$ does not in general imply
$\sigma_k^{(\Lambda)}(\Omega_1) \geq \sigma_k^{(\Lambda)}(\Omega_2)$, even
for $\Lambda = 0$ as might be naively expected from the scaling
\eqref{eq:scaling0} which guarantees $\sigma_k^{(0)}(\alpha\Omega) \geq \sigma_k^{(0)}(\Omega)$ for $\alpha<1$.  A dumbbell shape shown in
Fig. \ref{fig:domain_monotonicity} presents a simple
counter-example.  In fact, as the channel between two connected disks
gets narrower, the limit of the two isolated disks is approached, so that
the second eigenvalue $\sigma_2^{(0)}(\Omega_1)$ tends to $0$.

At the same time, the eigenvalues of the Dirichlet-to-Neumann map with $\Lambda\le 0$ obey monotonicity with
respect to exclusions: if $\Omega_1$, $\Omega_2$ are two open bounded sets such that $\overline{\Omega_1}\subset\Omega_2$, and $\Omega_1':=\Omega_2\setminus\overline{\Omega_1}$, see Figure \ref{fig:domain_monotonicity}, then for all $\Lambda\le 0$,
\begin{equation}\label{eq:muk_monotonicity0}
\sigma_k^{(\Lambda)}(\Omega_1') \le \sigma_k^{(\Lambda)}(\Omega_2),  \qquad k\in\mathbb{N}. 
\end{equation}
In the case $\Lambda=0$ this can be found in \cite{Bogosel17} but the proof translates directly to the case $\Lambda\le 0$: from the minimax relation
\eqref{eq:muk_minimax}, we have, if $U\in \mathcal{H}_\Lambda(\Omega_2)$ and $V=U|_{\Omega_1'}$, then $V\in  \mathcal{H}_\Lambda(\Omega_1')$, and also
\[
\|\nabla V\|^2_{L^2(\Omega_1')} - \Lambda \|V\|^2_{L^2(\Omega_1')} < \|\nabla U\|^2_{L^2(\Omega_2)} - \Lambda \| U\|^2_{L^2(\Omega_2)}
\]
and
\[
\|V|_{\pa_1'}\|^2_{L^2(\pa_1')} > \|U|_{\pa_2}\|^2_{L^2(\pa_2)}.
\]
Then \eqref{eq:muk_monotonicity0} follows immediately.

Note that for $\Lambda >0$, inequality \eqref{eq:muk_monotonicity0} is not necessarily true. Indeed, $\sigma_1^{(\Lambda)}(\Omega_2) \to -\infty$ as $\Lambda\nearrow \lambda_1^\Dir(\Omega_2)$. At the same time, since $\lambda_1^\Dir(\Omega_1')>\lambda_1^\Dir(\Omega_2)$ by strict domain monotonicity of the Dirichlet Laplacian 
(see \cite[Proposition 3.2.2]{Levitin}), 
$\sigma_1^{(\Lambda)}(\Omega_1')$  remains bounded below in this regime.

\begin{figure}[htb]
\centering
\includegraphics{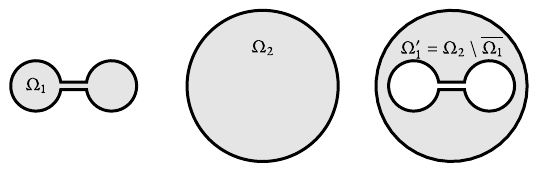}
\caption{An illustration of domain monotonicity and non-monotonicity.}
\label{fig:domain_monotonicity}
\end{figure}

\subsection{Isoperimetic inequalities for eigenvalues}\label{sec:iso}

Isoperimetric inequalities for Laplace and Steklov eigenvalues
have been actively studied in spectral geometry over more than a hundred years, see 
\cite[Chapters 5, 7]{Levitin} and references therein.  
Rather  little is known in this direction for eigenvalues of the Dirichlet-to-Neumann map for  $\Lambda \neq 0$. 
In what follows, we focus mostly on planar domains.

Let us start with the inequalities  for Steklov eigenvalues, i.e., for $\Lambda=0$.
If $\Omega\subset \R^2$ is a bounded simply connected domain with a
Lipschitz boundary $\pa$ of length $|\pa|_1$, then it was shown in \cite{Weinstock54} that
\begin{equation}\label{eq:Weinstock}
\sigma_2^{(0)} \leq \frac{2\pi}{|\pa|_1}, 
\end{equation}
with the equality attained if and only if $\Omega$ is a disk.  

This result  can be extended to all eigenvalues via the Hersch--Payne--Schiffer inequality \cite{Hersch75},
\begin{equation}\label{eq:Hersch2}
\sigma_k^{(0)} \leq \frac{2\pi (k-1)}{|\pa|_1},  \qquad k = 2, 3,\dots.
\end{equation}
Moreover, this inequality was shown in \cite{Girouard10} to be sharp for any $k$,  see also \cite{Girouard10b}.  
The equality is
attained in the limit by a union of $k-1$ identical disks touching each
other.  In fact, \eqref{eq:Hersch2} is strict for $k=3$, and as was shown in \cite{FraserSchoen20} (see also \cite[Remark 1.12]{Vinokurov}),
for any $k\ge 3$ the equality is not attained on a smooth domain. 
It would be interesting to extend this to Lipschitz domains and thus show that  inequality \eqref{eq:Hersch2} is in fact strict for all $k\ge 3$.

Note that originally Weinstock's and the Hersch--Payne--Schiffer inequalities  were proved
for a more general Steklov problem
\[
\begin{cases}
\Delta u = 0 \qquad&\text{in }\Omega,\\  
\partial_n u = \sigma \rho u  \qquad&\text{on }\pa,
\end{cases}
\]
with a positive mass density $\rho\in L^\infty(\pa)$ on the boundary
$\pa$, in which case $|\pa|_1$ is replaced by the total ``mass'' $M =
\int\nolimits_{\pa} \rho$. We also remark that there is a more general version of the Hersch--Payne--Schiffer inequality for pairwise products of Steklov eigenvalues which implies \eqref{eq:Hersch2} as a special case. 

As was highlighted in \cite{Girouard17}, Weinstock's inequality cannot be
generalised to non-simply connected domains: one can construct an annulus that yields a
counter-example.  In fact, for arbitrary bounded planar domains one has 
\begin{equation}
\label{eq:steklovhigher}
\sigma_k^{(0)} \leq \frac{8\pi (k-1)}{|\pa|_1}  \qquad k=2,3,\ldots.  
\end{equation}
This inequality was proved for $k=2$  in \cite{Kokarev14} and for $k\ge 2$ in  \cite{Girouard21}.
Moreover, \eqref{eq:steklovhigher} is sharp, and  there is a
sequence of bounded planar domains $\Omega_n$ with an increasing number of boundary
components, for which $\sigma_k^{(0)}(\Omega_n) |\pa_n|_1 \to 8\pi (k-1)$ 
\cite{Girouard21}. In fact, inequality \eqref{eq:steklovhigher} for $k=2$ can be improved provided the number of boundary
components stays bounded, see \cite{KarStern24}, \cite{KarLag24}.

Several extensions and analogues of Weinstock's inequality have been obtained in higher dimensions
in the case $\Lambda=0$. In particular, the ball maximises the first nonzero Steklov
eigenvalue $\sigma_2^{(0)}$ among {\em convex} domains of given surface
area $|\pa|_{d-1}$  \cite[Theorem 3.1]{Bucur21}:
\[
\sigma_2^{(0)} \leq \frac{(\omega_d\, d)^{\frac{1}{d-1}}}{|\pa|_{d-1}^{\frac{1}{d-1}}},
\]
where 
\begin{equation}\label{eq:omegad}
\omega_d = |\mathbb{B}|_d=\frac{\pi^{\frac{d}{2}}}{\Gamma\left(\frac{d}{2}+1\right)} 
\end{equation}
is the volume of the $d$-dimensional unit ball.  Actually,  a somewhat stronger  inequality
\[
\sigma_2^{(0)} \leq d (\omega_d)^{\frac{2}{d}} \frac{\left|\Omega\right|_d^{1-\frac{2}{d}}}{|\pa|_{d-1}}
\]
also holds in this setting \cite[Remark 5]{Bucur21}. 
At the same time, it was proved in \cite{Brock01} that the ball maximises $\sigma_2^{(0)}$ among all Euclidean
domains $\Omega\subset \R^d$ of given  volume $|\Omega|_d$:
\[
\sigma_2^{(0)} \leq \left(\frac{\omega_d}{|\Omega|_d}\right)^{\frac{1}{d}} \,.
\]
Note that  for simply connected  planar domains, Weinstock's inequality implies Brock's inequality via the usual  isoperimetric inequality.  We refer also to \cite{Bogosel17} for related results on isoperimetric inequalities for higher Steklov eigenvalues on planar domains under the area constraint.

Let us now present some recent results  in the case $\Lambda \neq 0$.
Let $\Omega$ be a bounded 
planar Lipschitz domain.
As was observed in \cite[Section 4.2]{Medvedev}, estimates obtained in  \cite[Section 5]{GNY} can be applied to Robin eigenvalues
and, together with the quantitative version of the Robin--Dirichlet-to-Neumann duality of Proposition \ref{prop:dualityfurther},  they imply
\begin{equation}\label{eq:Medvedev}
\sigma_k^{(\Lambda)} \le \frac{C(k-1+m) - \Lambda |\Omega|_2}{|\pa|_1}, \quad k \ge 1,
\end{equation}
where $C>0$ is a  universal constant (which does not depend on  $\Omega$ or $\Lambda$) and $m=\mathcal{N}_{-\Delta^\Dir_\Omega}(\Lambda)$.  

One can obtain a more explicit estimate involving the first two Dirichlet-to-Neumann eigenvalues provided $\Lambda<\lambda_1^\Dir(\Omega)$. It was shown in \cite[Theorem A]{Menezes} that under this condition,
for any bounded planar domain $\Omega$ with $l$ boundary components and  for any $\alpha\in (0, \pi/2)$,  one has
\begin{equation}\label{eq:Menezes}
\sigma_1^{(\Lambda)} \cos^2 \alpha + \sigma_2^{(\Lambda)} \sin^2 \alpha  \le \frac{4\pi (1-\cos \alpha)l -\Lambda |\Omega|_2}{ |\pa|_1}.
\end{equation}
It is instructive to compare \eqref{eq:Menezes} with the Weinstock inequality \eqref{eq:Weinstock}. Setting $\Lambda=0$ and $l=1$  in \eqref{eq:Menezes}, and taking into account that $\sigma_1^{(0)}=0$,  we get
\[
\sigma_2^{(0)} |\pa|_1 \le \frac{4 \pi}{(1+\cos \alpha)|\pa|_1}.
\]
Taking the limit as $\alpha \to 0$, we recover \eqref{eq:Weinstock}.

In conclusion, let us note that both inequalities \eqref{eq:Medvedev} and \eqref{eq:Menezes} admit generalisations to surfaces endowed with a Riemannian metric, see \cite{Medvedev}, \cite{Menezes} for details. We also remark that extensions of these results to higher dimensions have not been worked out yet.

\subsection{Other eigenvalue estimates}\label{sec:otherestimates}

We remark that in addition to isoperimetric inequalities some further bounds on all eigenvalues of $\DtN_\Lambda$ are known, in particular for $\Lambda<0$. For example, in \cite{Girouard22} the following result was proved.

\begin{theorem}\label{thm:Hor}
Let $\Omega \subset\R^d$ be a bounded domain with a boundary
$\pa$ of class $C^{1,1}$ if $d=2$ and of class $C^{2,\alpha}$ with  $\alpha>0$ if $d>2$. Denote by $0=\nu_1<\nu_2\le\dots$, the eigenvalues of the Laplace--Beltrami operator $-\Delta_{\pa}$
on the boundary $\pa$. Let $\Lambda \leq 0$.  Then, with some constant $C > 0$ depending only on $\Omega$, the
bounds
\begin{equation}\label{eq:muk_lambda_estimate}
\left|\sigma_k^{(\Lambda)} - \sqrt{- \Lambda+\nu_k} \right| \le  C
\end{equation}
hold uniformly over all $\Lambda \in (-\infty,0]$ and $k\in \mathbb{N}$.
\end{theorem}
A more  elaborate version  of Theorem \ref{thm:Hor} valid for positive $\Lambda$ can be also found in the same paper. 

Note that  since $\nu_1=0$, \eqref{eq:muk_lambda_estimate}  implies $\sigma_1^{(\Lambda)} \le C+\sqrt{-\Lambda}$. This bound can be further improved, without the restrictions on the smoothness of the boundary, to get the following
\begin{theorem}\label{thm:sqb} Let $\Omega\subset\mathbb{R}^d$ be a bounded Lipschitz domain. Then 
\begin{equation}\label{eq:Bucsigma1}
\sigma_1^{(\Lambda)}\le \sqrt{- \Lambda}\qquad\text{for }\Lambda< 0.
\end{equation}
\end{theorem}

\begin{proof} This is, in essence, an analogue of the bound on the first eigenvalue of  the Robin problem in \cite[Theorem 2.3]{GioSmi}, \cite[Lemma 2.1]{DanersKennedy}, and \cite[Proposition 4.12]{Bucur17}, and our proof is just an adaptation. 
In order to prove \eqref{eq:Bucsigma1}, we use the variational principle \eqref{eq:muk_minimax1} with the test function
\[
W(x):=\exp(\mydotp{x, \nu}),
\]
where a constant vector $\nu=(\nu_1,\dots,\nu_d)$ will be chosen later. As $|\nabla W(x)|=|\nu|\cdot |W(x)|$, we have
\[
\sigma_1^{(\Lambda)}\le \frac{(|\nu|^2-\Lambda)\|W\|^2_{L^2(\Omega)}}{\|W\|^2_{L^2(\pa)}}.
\]
As in \cite{GioSmi}, \cite{DanersKennedy}, \cite{Bucur17}, it is easy to show that $\|W\|^2_{L^2(\pa)}>2|\nu|\cdot \|W\|^2_{L^2(\Omega)}$, and thus
\[
\sigma_1^{(\Lambda)}\le \frac{|\nu|^2-\Lambda}{2|\nu|}.
\]
Choosing now any vector $\nu$ of length $\sigma_1^{(\Lambda)}>0$ immediately proves \eqref{eq:Bucsigma1}.
\end{proof}

On the basis of some numerical experiments, we suggest the following possible generalisation of Theorem \ref{thm:sqb}.

\begin{conjecture}\label{conj:sqroot} 
Let $\Omega\subset\mathbb{R}^d$ be a bounded Lipschitz domain, and let $\sigma^{(\Lambda)}$, $\Lambda\le 0$, be a fixed analytic eigenvalue branch of $\DtN_\Lambda$. Then 
\[
\sigma^{(\Lambda)}-\sigma^{(0)}\le \sqrt{-\Lambda}\qquad\text{for all }\Lambda\le 0.
\] 
\end{conjecture}

Conjecture \ref{conj:sqroot} holds for balls in $\mathbb{R}^d$, see \cite[formula (D5)]{Grebenkov25}.

\section{Eigenvalue asymptotics}\label{sec:asympt}
\subsection{Asymptotic behaviour of the principal eigenvalue as $\Lambda \to 0$}\label{sec:nearzero}

In the limit $\Lambda \to 0$, the lowest eigenvalue
$\sigma_1^{(\Lambda)}$ tends to $\sigma_1^{(0)}=0$.  The corresponding normalised eigenfunction is constant, $u_1^{(0)}= |\pa|_{d-1}^{-1/2}$, and thus $U_1^{(0)}= |\pa|_{d-1}^{-1/2}$. Therefore,  we immediately deduce from Proposition \ref{prop:dLambda} that
\[
\sigma_1^{(\Lambda)} = -\Lambda \frac{|\Omega|_d}{|\pa|_{d-1}} +O(\Lambda^2)  \qquad \text{as }\Lambda \to 0,
\]
cf. Proposition \ref{prop:sigma1upper}. 
The next term in the Taylor series expansion of $\sigma_1^{(\Lambda)}$ near zero can also be found by Proposition \ref{prop:dLambda}.  However the expression given in 
\cite[Appendix B]{Grebenkov19c} is easier: 
\begin{equation}\label{eq:d2mu0}
\sigma_1^{(\Lambda)} =-\Lambda \frac{|\Omega|_d}{|\pa|_{d-1}}  - \Lambda^2\frac{2|\Omega|_d^2}{|\pa|_{d-1}^3} \sum_{k=2}^\infty \frac{1}{\lambda_k^\Neu} \left\|U_k^\Neu\right\|^2_{L^1(\pa)} + O(\Lambda^3)  \qquad \text{as }\Lambda \to 0.
\end{equation}
In addition, it illustrates the concavity of $\sigma_1^{(\Lambda)}$, cf.\ Remark \ref{rem:dLambda}. The right-hand side of \eqref{eq:d2mu0} can be evaluated explicitly for the unit disk $\mathbb{D}$ or more generally for a unit ball $\mathbb{B}^d$, $d\ge 2$, yielding, for example,
\[
\begin{split}
\sigma_1^{(\Lambda)}(\mathbb{D}) &= -\frac{\Lambda}{2}-\frac{\Lambda^2}{8}+O(\Lambda^3)  \qquad \text{as }\Lambda \to 0,\\
\sigma_1^{(\Lambda)}(\mathbb{B}^3) &= - \frac{\Lambda}{3}-\frac{2\Lambda^2}{45}+O(\Lambda^3)  \qquad \text{as }\Lambda \to 0,
\end{split}
\]
for the unit disk and the three-dimensional unit ball, respectively. 

\subsection{Asymptotic behaviour as $\Lambda$ tends to a Dirichlet eigenvalue or to $-\infty$}\label{sec:lambda_large}
For any real-analytic  branch $\sigma^{(\Lambda)}$ of the eigenvalues of the Dirichlet-to-Neumann map, there are two distinguished asymptotic regimes. 

When $\Lambda$ is positive, increasing, and tends from below to an eigenvalue $\lambda^\Dir_\ell=\dots=\lambda^\Dir_{\ell+m-1}$ of multiplicity $m$ of the Dirichlet Laplacian, then the first $m$ eigenvalues $\sigma_k^{(\Lambda)}$, $k=1,\dots, m$, tend to $-\infty$, as illustrated in Figures \ref{fig:plotDtNdisk} and \ref{fig:plotDtNsquare}. Using the known asymptotics of the eigenvalues of the Robin Laplacian $-\Delta^{\Rob,\gamma}$ as 
$\gamma\to+\infty$, see, e.g., \cite{BelBBT18},  it is easily deduced that
\[
\sigma_k^{(\Lambda)}=-\frac{\eta_k}{\lambda^\Dir_\ell-\Lambda}+O(1)\qquad\text{as }\Lambda\to(\lambda_\ell^\Dir)^-,\quad k=1,\dots, m,
\]
where $\eta_1\ge\dots\ge \eta_m>0$ are the eigenvalues of the $m\times m$ Gram matrix
\[
\left(\myscal{\partial_n U^\Dir_i,\partial_n U^\Dir_j}_{L^2(\pa)}\right)_{i,j=\ell}^{\ell+m-1}
\]
of linearly independent functions $\partial_n U^\Dir_\ell,\dots,\partial_n U^\Dir_{\ell+m-1}$.

To study the asymptotic behaviour of the eigenvalues $\sigma_k^{(\Lambda)}$ as $\Lambda\to-\infty$, we first look at the behaviour of the eigenvalues of the Robin Laplacian $\lambda_k^{\Rob,\gamma}$ as $\gamma\to-\infty$. For any bounded Lipschitz domain $\Omega$, there exist constants $p_{\Omega}\ge 1$, $\Gamma_{\Omega,1}<0$,  such that 
\[
-p_\Omega\gamma^2\le \lambda_1^{\Rob,\gamma}\le -\gamma^2\qquad\text{for }\gamma<\Gamma_{\Omega,1},
\]
see \cite[Lemma 2.7]{Khalile20} for the lower bound, and  \cite[Proposition 4.12]{Bucur17} for the upper bound (valid with $\Gamma_{\Omega,1}=0$). Similarly, for any $k\ge 2$ there exist constants
$0<P_{\Omega,k}$, $\Gamma_{\Omega,k}<0$,  such that \cite{DiePanprivate}
\[
\lambda_k^{\Rob,\gamma}\le -P_{\Omega,k}\gamma^2 \qquad\text{for }\gamma<\Gamma_{\Omega,k}.
\]

Switching now to the asymptotics of $\sigma_k^{(\Lambda)}$ as $\Lambda$ goes to $-\infty$, we deduce by the Robin--Dirichlet-to-Neumann duality, that for a bounded Lipschitz domain $\Omega$ and  for each fixed $k$, we have
\[
\sigma_k^{(\Lambda)} = O\left(\sqrt{-\Lambda}\right)  \qquad\text{as }\Lambda \to -\infty.
\]
In many situations, discussed below, we further have
\begin{equation}\label{eq:muk_ck}
\sigma_k^{(\Lambda)} = c_k  \sqrt{-\Lambda} +o\left(\sqrt{-\Lambda}\right)  \qquad\text{as }\Lambda \to -\infty,
\end{equation}
with some constants  $c_k$ depending on the geometry of $\Omega$, or more precisely of its boundary $\pa$.  
The validity of \eqref{eq:muk_ck} and the values of the coefficients $c_k$ vary depending on the smoothness of $\pa$, but we note an interesting recent example \cite{DiePan}
which shows that \eqref{eq:muk_ck} may not hold even for the principal eigenvalue $\sigma_1^{(\Lambda)}$ if $\Omega$ is only assumed to be Lipschitz.

For the boundary $\pa$ of class $C^1$, 
we have, by  inverting the result of \cite{DanersKennedy}, 
\begin{equation}\label{eq:sigmakinfsmooth}
\sigma_k^{(\Lambda)} = \sqrt{-\Lambda} + o\left(\sqrt{-\Lambda}\right)\qquad\text{as }\Lambda\to-\infty,
\end{equation}
for any fixed $k$, and therefore we can take $c_k=1$ for all $k$ in \eqref{eq:muk_ck}.
Note that for the boundary of class $C^{1,1}$ if $d=2$ or of class $C^{2,\alpha}$ with  $\alpha>0$ if $d>2$, we can use Theorem  \ref{thm:Hor} to improve the remainder estimate in \eqref{eq:sigmakinfsmooth} to $O(1)$.

In order to discuss some improvements to \eqref{eq:sigmakinfsmooth} when the boundary has higher smoothness, we first recall  the following asymptotic behaviour of the eigenvalues of the Robin Laplacian in domains $\Omega\subset\R^d$ with  $C^3$ boundary \cite{Pankrashkin15}:
\begin{equation}\label{eq:lambdak_Pank}
\Lambda_k^{\Rob,\gamma} = - \gamma^2 + (d-1) H_{\max} \gamma + O\left(|\gamma|^{\frac{2}{3}}\right)  \qquad \text{as }\gamma \to -\infty,
\end{equation}
where $H_{\max}=\max\limits_{x\in\pa} H(x)$ is the maximum value of the mean curvature $H(x)$ at the point $x\in \pa$.  In the
case of bounded, star-shaped, $C^2$ smooth domains, one has $H_{\max}\geq
(\omega_d/|\Omega|_d)^{\frac{1}{d}}$, where $\omega_d$ is given by \eqref{eq:omegad}.
For simply-connected bounded planar
domains with $C^4$ boundary, two-sided two-term estimates for $\gamma\to-\infty$ of
the eigenvalues $\lambda_k^{\Rob,\gamma}$ that involve the maximal and
minimal curvatures of $\pa$, as well as eigenvalues of an auxiliary
Schr\"{o}dinger operator based on the curvature, were deduced in
\cite{Exner14}.  The above
asymptotic result was further refined  for bounded planar domains with $C^\infty$ boundary in 
\cite{Helffer18}. Parametrising the boundary $\pa$ by the arc-length $s$ and assuming that the maximal curvature is
attained in a single boundary point, say, $H_{\max}=H(s_0)$ and that this maximum
is non-degenerate, $H''(s_0) > 0$, they proved the following: for any fixed $k\in\mathbb{N}$, there exists a sequence
$\left(\beta_{k,j}\right)_{j=1}^\infty$ such that
\[
\Lambda_k^{\Rob,\gamma} =  - \gamma^2 + H_{\max}\gamma + (2k-1) \sqrt{H''(s_0)/2}\, |\gamma|^{\frac12} + \sum\limits_{j=0}^J \beta_{k,j} |\gamma|^{-j/2} + 
o\left(|\gamma|^{-J/2}\right)\qquad\text{as }\gamma\to -\infty
\]
for any positive $J \in \mathbb{N}$. 
This expansion also yields the leading order of the spectral gap,
\[
\Lambda_{k+1}^{\Rob,\gamma}  -\Lambda_k^{\Rob,\gamma}  =  \sqrt{H''(s_0)/2}\, |\gamma|^{\frac12}+O(|\gamma|) \qquad\text{as }\gamma\to -\infty.
\]

Returning to the case  $\Omega\subset\mathbb{R}^d$ with a $C^3$ boundary, and using the Robin--Dirichlet-to-Neumann duality, we can now invert \eqref{eq:lambdak_Pank} and get, in the leading order, 
\[
\sigma_k^{(\Lambda)} = \sqrt{-\Lambda} - \frac{(d-1)H_{\max}}{2} + O((-\Lambda)^{-\frac12})   \qquad \text{as } \Lambda \to -\infty.
\]

If the boundary is not smooth, e.g., it has corners like in polygonal
domains, the coefficients $c_k$ in \eqref{eq:muk_ck} may be smaller
than one for some $k$.  We are not aware of explicitly stated results
on $c_k$ for the Dirichlet-to-Neumann map; however, the Robin--Dirichlet-to-Neumann duality allows us to translate various asymptotic results
known for the Robin Laplacian in terms of eigenvalues of $\DtN_\Lambda$.  For instance, it was shown in   \cite{Levitin08}, see also \cite{Lacey}, that if $\Omega\subset \R^2$ is a curvilinear polygon with the
smallest angle $\alpha$, the first Robin eigenvalue 
behaves as $\Lambda_1^{\Rob,\gamma} = - C_1 \gamma^2 + o(\gamma^2)$ as $\gamma \to -\infty$, with 
\[
C_1 = \begin{cases}
\frac{1}{\sin^2(\alpha/2)}\quad&\text{if }\alpha < \pi,\\
1\quad&\text{otherwise}.
\end{cases}
\]
Inverting this relation, we get in \eqref{eq:muk_ck},  
\begin{equation}\label{eq:c1}
 c_1 = \begin{cases}
\sin(\alpha/2)\quad&\text{if }\alpha < \pi,\\
1\quad&\text{otherwise}.
\end{cases}
\end{equation}

As outlined in 
\cite{Helffer15}, each corner of the boundary can be intuitively
viewed as a ``geometric well'', and it is the deepest well that
determines the leading term of the lowest eigenvalue.  The authors addressed
a natural question of what happens if there are several wells of the
same depth, i.e., several corners with the same angle.  More precisely,
they considered a planar domain with two equal angles (an isosceles
triangle and an infinite bi-angle domain) separated by distance $L$,
and derived the large-$L$ asymptotic relation for the two first
eigenvalues.  The correction to the leading term $-C_1\gamma^2$
exponentially decays as $L$ increases, similarly to tunnelling effects
between two identical wells in quantum mechanics. 
Extensions of these
results, as well as a power-law decay of the correction term for
curvilinear polygons and an exponentially fast decay for polygonal
domains, were discussed in \cite{Khalile18}, \cite{Khalile18b}, \cite{Khalile20},  still within the
framework of the Robin Laplacian. At the same time, very little is known for
non-smooth boundaries in higher dimensions; some results in this direction can 
be found in \cite{Pankrashkin22}.

We finish this section by stating a conjectural extension of \eqref{eq:muk_ck} for all $k$ for an
arbitrary polygon with $n$ vertices and interior angles $\alpha_1, \alpha_2, \dots, \alpha_n$, which  
was proposed in \cite{Chaigneau24}. We reformulate it slightly. 
Let, for $\alpha\in (0,2\pi)$, 
\[
\overline{k}(\alpha):=\max\left\{k\in\{0\}\cup\mathbb{N}: 2k-1<\frac{\pi}{\alpha}\right\},
\]
and let
\[
\mathcal{C}(\alpha):=\begin{cases}
\varnothing\qquad&\text{if }\alpha \ge \pi\text{ (and therefore $\overline{k}(\alpha)=0$)},\\
\left\{\sin\frac{\alpha}{2},\sin\frac{3\alpha}{2}, \dots, \sin\frac{(2\overline{k}(\alpha)-1)\alpha}{2}\right\},\qquad&\text{if }k(\alpha) > 0.
\end{cases}
\]

Consider the multi-set 
\[
\widetilde{\mathcal{C}}=\bigcup_{i=1}^n \mathcal{C}(\alpha_i),
\]
with $K := \sum_{i=1}^n \overline{k}(\alpha_i)$ positive elements, rearranging them in non-decreasing order:
\[
\widetilde{\mathcal{C}}=\left\{\tilde{c}_1\le\dots\le \tilde{c}_K\right\}.
\]
Further, set $\tilde{c}_k:=1$ for $k>K$. 

\begin{conjecture}[\cite{Chaigneau24}]\label{conj:polygons}
For any curvilinear polygon, the asymptotics  \eqref{eq:muk_ck} holds with 
\[
c_k=\tilde{c}_k, \qquad k\in\mathbb{N}.
\]
\end{conjecture}

In order to explain this conjecture, consider a Robin Laplacian $-\Delta^{\Rob,1}_{S_\alpha}$ with the Robin parameter $\gamma=1$  in an infinite straight planar sector $S_\alpha$ of angle $\alpha$. This operator has the essential spectrum $[-1,+\infty)$ and may also   
have a finite number $m(\alpha)$ of discrete eigenvalues below $-1$. By \cite{Lyalinov21}, 
\begin{equation}\label{eq:Lyal}
m(\alpha)\ge \overline{k}(\alpha),
\end{equation}
and the $\overline{k}(\alpha)$ eigenvalues  are known explicitly, $\lambda^{\Rob,1}\in \left\{-\tilde{c}^{-2}, \tilde{c}\in \mathcal{C}(\alpha)\right\}$. An open question is whether there are any other eigenvalues of $-\Delta^{\Rob,1}_{S_\alpha}$ below the essential spectrum, or if bound \eqref{eq:Lyal} is in fact an equality. If we assume the latter, the validity of Conjecture \ref{conj:polygons} would follow from the results of  \cite{Khalile18} by the Robin--Dirichlet-to-Neumann duality.

There are special cases for which the equality in \eqref{eq:Lyal} is known to hold: if $\alpha\ge\pi$, then $m(\alpha)=\overline{k}(\alpha)=0$   \cite{Levitin08}, and if $\alpha\in\left[\frac{\pi}{3},\pi\right)$, then $m(\alpha)=\overline{k}(\alpha)=1$ \cite{Khalile18b}. Therefore we have the following result, cf.\ \cite[Corollary 1.3]{Khalile18},

\begin{proposition} 
If all angles of a curvilinear polygon are not less than $\frac{\pi}{3}$, Conjecture \ref{conj:polygons} holds for such a polygon. 
\end{proposition} 

\subsection{Weyl's asymptotic law}\label{sec:weyl}

Recalling notation \eqref{eq:NlambdaA}, we denote the eigenvalue counting function of the Dirichlet-to-Neumann map by
\[
\N^{(\Lambda)}(\sigma):=\N_{\DtN_\Lambda}(\sigma)=\#\left\{k: \sigma_k^{(\Lambda)}\le\sigma\right\},
\]
and introduce additionally the heat trace
\[
\mathcal{Z}^{(\Lambda)}(t) = \sum\limits_{k=1}^\infty \er^{-\sigma_k^{(\Lambda)}t},\qquad t>0.
\]

The asymptotic behaviour of the eigenvalues $\sigma_k^{(\Lambda)}$ for 
large $k$ is usually characterised in terms of the asymptotics of the eigenvalue counting
function $\N^{(\Lambda)}(\sigma)$ as $\sigma\to+\infty$ or of the asymptotics  the heat trace $\mathcal{Z}^{(\Lambda)}(t)$ as $t\to 0^+$.

If the boundary $\pa$ is smooth, and $\Lambda\notin\Spec\left(-\Delta^\Dir_\Omega\right)$,  then $\DtN_\Lambda$ is a pseudodifferential operator of order one, and 
\begin{equation}  \label{eq:Nlambda_mu}
\N^{(\Lambda)}(\sigma) = \frac{\omega_{d-1} \, |\pa|_{d-1}}{(2\pi)^{d-1}} \, \sigma^{d-1} + O\left(\sigma^{d-2}\right)  \qquad\text{as }\sigma \to +\infty,
\end{equation}
where $\omega_d$ is given by \eqref{eq:omegad}, by the standard Weyl's one-term law \cite{Horm}.

Inverting \eqref{eq:Nlambda_mu}, one gets the asymptotic behaviour for 
large eigenvalues, 
\begin{equation}  \label{eq:Weyl_muk}
\sigma_k^{(\Lambda)} = 2\pi \left(\frac{k}{\omega_{d-1} |\pa|_{d-1}}\right)^{\frac{1}{d-1}} + O\left(1\right)
\qquad\text{as }k\to \infty.
\end{equation} 
Using the Laplace transform, one can also obtain the short time asymptotics of  the heat trace  for any fixed
$\Lambda \in \R\backslash \Spec\left(-\Delta^\Dir_\Omega\right)$: 
\[
\mathcal{Z}^{(\Lambda)}(t) = \frac{\Gamma\left(\frac{d}{2}\right) \, |\pa|_{d-1}}{\pi^{\frac{d}{2}}} t^{1-d} + o\left(t^{1-d}\right)\qquad\text{as }t\to 0^+,
\]
see \cite{Polterovich15}, \cite{Liu15} for further terms of this expansion in the case $\Lambda=0$.

In fact, as follows from \cite{Horm}, for smooth domains $\Omega$ there is a pointwise version of Weyl's law \eqref{eq:Nlambda_mu}. Namely,  for any $x\in \pa$,
\begin{equation}
\label{eq:pointWeyl}
\sum_{\sigma_k^{(\Lambda)} \le \sigma} \left|u_k^{(\Lambda)}(x)\right|^2 = \frac{\omega_{d-1}}{(2\pi)^{d-1}} \, \sigma^{d-1} + O(\sigma^{d-2})  \qquad\text{as }\sigma \to +\infty.
\end{equation}
In particular, the error term estimate implies an upper bound
\begin{equation}\label{linftyeig}
\left\| u_k^{(\Lambda)}\right\|_{L^\infty(\pa)} = O\left(\left(\sigma_k^{(\Lambda)}\right)^{\frac{d-2}{2}}\right) = O\left(k^{\frac{d-2}{2d-2}}\right)\qquad\text{as }k\to\infty.
\end{equation}
Moreover, for $\Lambda=0$ a similar bound holds, by the maximum principle, for any bulk eigenfunction $U_k^{(0)}$. 

For smooth bounded planar domains, the eigenvalue asymptotics \eqref{eq:Weyl_muk}, which in this case reads
\[
\sigma_k^{(\Lambda)} = \frac{\pi k}{|\pa|_1} + O(1),
\]
can be made significantly more precise. 
If $\Lambda=0$, then the eigenvalues of the Dirichlet-to-Neumann map for a simply connected $\Omega\subset\mathbb{R}^2$ are superpolynomially close  as $k \to \infty$ to those of the disk with the same perimeter \cite{Rozenblyum86}, \cite{Edward93},
\begin{equation}\label{eq:muk_asympt_k}
\sigma_{2k}^{(0)} = \sigma_{2k+1}^{(0)} + O(k^{-\infty}) = \frac{2\pi k}{|\pa|_1} + O(k^{-\infty}),
\end{equation}
where the remainder term $r_k=O(k^{-\infty})$ means that $\lim\limits_{k\to \infty} k^n
r_k = 0$ for every $n\in {\mathbb N}$.  This result was further
extended to the case when the boundary $\pa$ is not connected in  \cite{Girouard14}. As a corollary, 
it was shown that the number of connected
components of the boundary, as well as their lengths, are invariants
of the Steklov spectrum.

Such a rapid convergence of the eigenvalues $\sigma_k^{(0)}$ to those for a
disk of perimeter $|\pa|_1$ is specific for dimension two. 
In a way, it is a consequence of the fact that the boundary $\pa$ is a smooth closed curve and hence 
is locally isometric to a circle. 
For similar reasons, the eigenfunctions
$u_k^{(0)}$ of smooth bounded simply-connected planar domains are also close to those for a disk, i.e. they can be expressed as trigonometric functions of the arc-length parameter plus a decaying remainder. A rigorous form of this statement with a $o(1)$ bound on the error term can be found in \cite{Shamma71} for smooth bounded domains. In fact, as follows from the proof in \cite{Shamma71}, $C^4$ regularity of the boundary is enough, and one expects that higher regularity implies better bounds. In particular, for simply-connected domains with real-analytic boundary it was shown in \cite[Section 3]{Polterovich19} that the error term for the eigenfunction asymptotics  is exponentially small, and similarly for the eigenvalue asymptotics \eqref{eq:muk_asympt_k}.

In \cite{Lagace20}, asymptotics  \eqref{eq:muk_asympt_k} was extended to the case $\Lambda
\ne 0$ to give  the complete asymptotic expansion for the
eigenvalues $\sigma_k^{(\Lambda)}$ for a simply-connected planar domain
$\Omega$ with a smooth boundary:
\[
\sigma_{2k}^{(\Lambda)} = \sigma_{2k+1}^{(\Lambda)} + O\left(k^{-\infty}\right) = \frac{2\pi k}{|\pa|_1} 
+ \frac{\Lambda}{k} \sum\limits_{n=0}^\infty \frac{s_k(\lambda;\Omega)}{k^n}\qquad\text{as }k\to \infty,
\]
where $s_n(\lambda;\Omega)$ are polynomials in $\lambda$ of degree at
most $n$, with
\[
s_0(\lambda;\Omega) = - \frac{|\pa|_1}{4\pi} \,,  \qquad 
s_1(\lambda;\Omega) = \frac{|\pa|_1}{4\pi}.
\]

If the boundary is not assumed to be smooth, the pseudodifferential techniques cannot be used. 
Historically, the first result on the asymptotics of the eigenvalue counting function of the Dirichlet-to-Neumann map was due to \cite{Sandgren55}  in the case $\Lambda = 0$ for $C^2$ boundary $\pa$:
\begin{equation}  \label{eq:N0_mu}
\N^{(0)}(\sigma) =  \frac{\omega_{d-1} \, |\pa|_{d-1}}{(2\pi)^{d-1}} \, \sigma^{d-1} + o(\sigma^{d-1})  \qquad\text{as }\sigma \to +\infty.
\end{equation}

In \cite[Theorem 11]{Girouard22}, the inequality \eqref{eq:muk_lambda_estimate} has been used to prove \eqref{eq:Nlambda_mu}  for bounded domains with $C^{2,\alpha}$ boundary for $d\ge 3$ or $C^{1,1}$ boundary in the planar case, provided $\Lambda \le 0$. 
 We also note that the asymptotic formula  \eqref{eq:N0_mu} was extended to bounded domains with piecewise $C^1$ boundary \cite{Agranovich06},  and  more recently  to bounded domains with Lipschitz boundary in two dimensions \cite{Karpukhin22} and in higher dimensions \cite{Rozenblum23}. It would be interesting  to verify if these results can be generalised to the case  $\Lambda \neq 0$.  Another compelling question is to investigate the analogues of the estimates \eqref{eq:pointWeyl} and 
 \eqref{linftyeig} for non-smooth boundaries.

If $\Omega$ is  a curvilinear polygon, then according to  \cite{Levitin22},  subject to some minor constraints, 
one can construct a sequence of quasi-eigenvalues $\tilde{\sigma}_k$ as the sequence of roots of  an explicit trigonometric function depending on the side-lengths and angles of the polygon (or, equivalently, as the sequence of eigenvalues of  a particular \emph{quantum graph}), such that 
\begin{equation}\label{eq:Spoly}
\sigma_k^{(0)} = \tilde{\sigma}_k  + O(k^{-\epsilon})\qquad\text{as }k\to \infty,
\end{equation}
with some $\epsilon>0$ depending on $\Omega$. This asymptotics is closely related to the so-called sloshing problem (a mixed Neumann--Steklov eigenvalue problem, see \S\ref{sec:mixed}), and the sloping beach problem, both  arising in hydrodynamics.

\section{Eigenfunctions}\label{sec:nodal}
\subsection{Regularity of bulk eigenfunctions}\label{sec:bulkregularity}

The following result holds concerning the  regularity of bulk eigenfunctions on Lipschitz domains.

\begin{theorem}\label{thm:regularity}
Let $\Omega \subset {\mathbb R}^d$ be a bounded Lipschitz domain.
The  bulk eigenfunctions $U_k^{(\Lambda)}$, $k=1,2,\dots$,  are real-analytic in $\Omega$ and continuous up to the boundary: $U_k^{(\Lambda)} \in C(\overline{\Omega})$. Moreover, if $\Omega$  has $C^{1,1}$ boundary, then $U_k^{(\Lambda)} \in C^1(\overline{\Omega})$.
\end{theorem}

The real-analyticity of the bulk eigenfunctions follows immediately from the fact that they are solutions of an elliptic equation with constant coefficients. The proof of the continuity up to the boundary for Lipschitz domains can be found in  \cite[Section 4]{Daners2009}:  while the result is stated there only for $\Lambda=0$, it holds for any $\Lambda$ as follows from the results on the regularity of Robin eigenfunctions.   
Another proof of continuity up to the boundary for $\Lambda=0$ can be found in \cite[Section 2]{Decio2022}. 
The last part of Theorem \ref{thm:regularity} is  contained in \cite[Proposition 2]{Decio2022}. Once again, it is stated only for $\Lambda=0$, but the argument extends in a straightforward way to the case of an arbitrary $\Lambda$.
 
\subsection{Courant's theorem}\label{sec:courant}

One of the fundamental results in spectral geometry is Courant's nodal domain theorem \cite[section 4.1]{Levitin}.
Recall that a nodal domain of a function is a connected component of the complement to its zero set.
The standard Courant's theorem states that  a Laplace eigenfunction (with Dirichlet, Neumann,  or Robin boundary conditions) corresponding to an eigenvalue with index $k$ has at most $k$ nodal domains. The proof of  this result is based on  the variational principle and the unique continuation property of eigenfunctions.  It can be extended essentially verbatim to the case of Steklov eigenfunctions  $U_k^{(0)}$ \cite{Kuttler69},  \cite{Karpukhin14}.  An extension to the general case of bulk eigenfunctions $U_k^{(\Lambda)}$ for an arbitrary $\Lambda$ is more subtle and has been recently obtained by Hassannezhad and Sher. It is a consequence of Courant's nodal domain theorem for Robin eigenfunctions and the Robin--Dirichlet-to-Neumann duality as stated in  Proposition \ref{prop:dualityfurther}.

\begin{theorem}[{\cite{Hassannezhad22}}]
\label{thm:Courant}
Let $\Omega\subset\R^d$ be a bounded domain with Lipschitz boundary
$\pa$, and let $U_k^{(\Lambda)}$ be a bulk eigenfunction of $\DtN_\Lambda$
corresponding to an eigenvalue $\sigma_k^{(\Lambda)}$.  Then $U_k^{(\Lambda)}$ has
at most $k+m$ nodal domains, where $m=\mathcal{N}_{-\Delta^\Dir_\Omega}(\Lambda)$.
\end{theorem}

Consider now the eigenfunctions $u_k^{(\Lambda)}$ of the
Dirichlet-to-Neumann operator $\DtN_\Lambda$ on $\partial \Omega$. Since these are the boundary traces of $U_k^{(\Lambda)}$, it follows from Theorem \ref{thm:regularity}  that $u_k^{(\Lambda)} \in C(\partial \Omega)$. In particular, one may 
ask whether a version of Courant's theorem holds for  $u_k^{(\Lambda)}$. As was noted in \cite{Girouard17},  
the proof of Courant's theorem cannot be generalised to the eigenfunctions $u_k^{(\Lambda)}$ because the
Dirichlet-to-Neumann operator $\DtN_\Lambda$ is nonlocal. Moreover, for $d \ge 3$,  the number of 
nodal domains of  $u_k^{(\Lambda)}$ cannot be controlled by the number of nodal domains of $U_k^{(\Lambda)}$.
In fact, as has been recently shown in \cite{Enciso24} there are no Courant-type bounds for $u_k^{(0)}$  in the 
Riemannian setting: given a compact manifold with boundary of dimension $d \ge 3$ and any numbers $K$ and $N$ one can construct a Riemannian metric such that the eigenfunctions  $u_k^{(0)}$, $k=1,\dots, K$, have all at least $N$ nodal domains. It would be interesting to adapt this construction to Euclidean domains, and to extend it to arbitrary $\Lambda$.  It is an open question whether  Courant's bound  holds  \emph{asymptotically} for $d \ge 3$, i.e. whether the number of nodal domains of $u_k^{(\Lambda)}$ is $O(k)$ as $k \to \infty$. If one ignores small oscillations and considers only ``deep''  nodal domains (i.e. nodal domains in which the absolute value of the eigenfunctions attains a certain fixed threshold), a bound of this type holds for smooth domains, see \cite{Buhovsky22} for details.

In two dimensions the situation is different.  First, it is easy to show that the number of endpoints of the nodal lines of 
the bulk eigenfunctions is bounded in terms of  the number of their nodal domains and the number of connected components of $\partial \Omega$ (cf. \cite[Lemma 3.4]{Alessandrini94} for  the case $\Lambda=0$ and $\Omega$ simply-connected).
Therefore, Theorem \ref{thm:Courant} would imply a Courant-type bound for  $u_k^{(\Lambda)}$ provided we can show that $u_k^{(\Lambda)}$ has no other zeros than the endpoints of the nodal lines of $U_k^{(\Lambda)}$. If $\Omega$
has $C^{1,1}$ boundary, the result follows from the last assertion of Theorem \ref{thm:regularity} and the Hopf--Oleinik lemma (cf. \cite[Chapter 3]{GilbargTrudinger}). 
Indeed, let $x \in \partial \Omega$ be such that  $u_k^{(\Lambda)}(x)=0$ and $x$ does not belong to the closure of the nodal set of $U_k^{(\Lambda)}$. Since  the boundary $\pa$ is $C^{1,1}$,   $\Omega$ satisfies the interior sphere property. Hence,  there exists a ball $B \subset \Omega$ such 
that $x \in \partial B$ and  $U_k^{(\Lambda)}$ does not change sign in $B$; we may assume without loss (multiplying it by $-1$ if needed) that it is negative. Then by Hopf--Oleinik lemma $\partial_n u_k^{(\Lambda)}(x)>0$, but at the same time 
since $U_k^{(\Lambda)} \in C^1(\overline{\Omega})$ we have $\partial_n u_k^{(\Lambda)}(x)=\sigma_k u_k^{(\Lambda)}(x)=0$, which is a contradiction.

However, if we only assume that $\Omega$ is Lipschitz, one cannot apply the Hopf--Oleinik lemma in a similar way.
 Therefore, the argument above does not work, and it is an interesting question to find an alternative approach.

\begin{remark}
We note that  the Dirichlet-to-Neumann map arises in the study of the nodal count for the eigenfunctions of the Laplace operator. We refer to \cite{Cox17}, \cite{Berkolaiko19}, \cite{Berkolaiko22} for details.
\end{remark}

\subsection{Positivity of the first eigenfunction and the heat semigroup}\label{sec:positivity}

It follows from Theorem \ref{thm:Courant} that the first bulk eigenfunction $U_1^{(\Lambda)}$ of the Dirichlet-to-Neumann operator $\DtN_\Lambda$ does not change sign in $\Omega$ for $\Lambda < \lambda_1^\Dir(\Omega)$. In fact, a stronger result holds, see  \cite{ArterElstGl2020}.
\begin{theorem}
Let $\Omega$ be a Lipschitz domain and  $\Lambda < \lambda_1^\Dir(\Omega)$. 
Then $U_1^{(\Lambda)} \in C(\overline{\Omega})$ and $U_1^{(\Lambda)}(x)>0$ for all $x \in \overline{\Omega}$. 
\end{theorem}
This result can be viewed as  two  separate statements.  The first one  is strict positivity of the bulk eigenfunction $U_1^{(\Lambda)}$ in $\Omega$, and the second one is the strict positivity of the boundary eigenfunction
$u_1^{(\Lambda)}$ on $\pa$. The first statement can be deduced from Courant's theorem and  a version of the maximum principle \cite[Theorem 3.5]{GilbargTrudinger} which implies that a solution of the equation $-\Delta U = \Lambda U$ must change sign in any neighbourhood of its zero. The second statement is more delicate. 
As was mentioned in the previous section, for $C^{1,1}$ boundaries it 
follows from the Hopf--Oleinik lemma, however in the general Lipschitz case one has to use other techniques. The result is a 
corollary of the positivity  properties of the  heat semigroup associated with the Robin Laplacian combined with the Robin--Dirichlet-to-Neumann duality. Related results concerning the positivity properties 
of the heat semigroup $\exp\left(-t \DtN_\Lambda\right)$ associated with the Dirichlet-to-Neumann map for
$\Lambda < \lambda_1^\Dir(\Omega)$ can be also found in \cite{Arendt07},  \cite{Arendt12}, \cite{terElst19}.
 
As a consequence of the strict positivity of $u_1^{(\Lambda)}$ for  $\Lambda <
\lambda_1^\Dir$,   we deduce using orthogonality 
that all other boundary Dirichlet-to-Neumann eigenfunctions must change sign on $\pa$, and, moreover, 
the first eigenvalue $\sigma_1^{(\Lambda)}$ is simple.
Note that the latter result requires only the connectedness of the
domain $\Omega$, not of the boundary $\pa$.

Let us note that the positivity of the heat semigroup and the positivity of the first
Dirichlet-to-Neumann (bulk) eigenfunction for $\Lambda >
\lambda_1^\Dir$ has been studied in \cite{Daners14}.  Using two
elementary examples (an interval and a disk), Daners showed that the
semigroup can be non-positive depending on the values of $\Lambda$,
and that there is no clear positivity criteria (both for the semigroup
and for the first eigenfunction) in the range $\Lambda >
\lambda_1^\Dir$.

\begin{remark}\label{rem:Sato}
It was realised in \cite[Theorems 9.1 and 9.1$'$]{Sato65} that the semigroup of the Dirichlet-to-Neumann
operator $\DtN_\Lambda$ with $\Lambda \leq 0$ defines a class of
Markov processes on the boundary $\pa$ that can be seen as the trace
on $\pa$ of a diffusion process in $\Omega$, up to an appropriate
stochastic time change $t\mapsto \ell_t$.  The positivity of this
semigroup plays the crucial role here. Indeed, if $p_0(x_0)$ is a
prescribed probability density of starting points $x_0 \in \pa$ of
the Markov process, the semigroup $\exp(-\ell\DtN_\Lambda)$ determines
the (possibly non-normalised) probability density $p(x)$ of the
position $x$ of that process on the boundary at the \emph{boundary local time} (see \cite{Ito}, \cite{Freidlin}),  $\ell$: $p = \exp(-\ell \DtN_\Lambda) p_0$. The positivity of
the semigroup ensures that, for any positive function $p_0$ and any
$\ell > 0$, the resulting function $p$ is also positive, as it should
be for its probabilistic interpretation.
When $\Lambda < 0$, the term $-\Lambda U$ in the Helmholtz
equation models the interaction of a reactive medium with a diffusing
particle, which, as a result, may disappear in the bulk with a reaction rate
proportional to $|\Lambda|$. We also note that the semigroup $\exp(-\ell \DtN_0)$ 
was used in the analysis of
the escape problem for the Cauchy process
\cite{Banuelos04}.
More recently, the heat kernel representing $\exp(-\ell\DtN_\Lambda)$ ,
\begin{equation}\label{eq:Sigma_def}
\Sigma^{(\Lambda)}(x,\ell; x_0) = \sum\limits_{k=1}^\infty u_k^{(\Lambda)}(x) \, u_k^{(\Lambda)}(x_0) \, \er^{-\ell \sigma_k^{(\Lambda)}},
\qquad x,x_0\in \pa,  \ell \geq 0,
\end{equation}
was used to build the encounter-based approach to diffusion-controlled
reactions \cite{Grebenkov20}, see \S \ref{sec:reactions}.  
\end{remark}

\subsection{Localisation near the boundary}\label{sec:localisation}

The bulk eigenfunctions $U_k^{(\Lambda)}$ of the Dirichlet-Neumann map with $\Lambda<\lambda_1^\Dir(\Omega)$ exhibit remarkable localisation properties near the boundary $\pa$ as $k\to+\infty$.

Consider first the case $\Lambda=0$ and let $\Omega \subset \R^d$ be a bounded domain with smooth
boundary.  It was shown in \cite{Hislop01} that for any compact $K\subset \Omega$,
$\|U_k^{(0)}\|_{H^1(K)} = O(k^{-\infty})$.  In other
words, the magnitude  of an eigenfunction $U_k^{(0)}$ on any interior 
subset of the domain rapidly vanishes as $k$ increases.
It was also  conjectured that the decay of $U_k^{(0)}$ is of order $\er^{-k \operatorname{dist}(K,\pa)}$ in the
case of a real-analytic boundary.

This conjecture was proved in \cite{Polterovich19} for planar
bounded domains with real-analytic boundary.  In
fact, an immediate consequence of their analysis was the existence of
constants $C$ and $\tau$ which depend only on the geometry of $\Omega$ such that
\begin{equation}  \label{eq:Vk_exp}
\left|U_k^{(0)}(x)\right| \leq C \, \er^{-\tau \sigma_k^{(0)} \operatorname{dist}(x, \pa)},  \qquad k\in\mathbb{N}.
\end{equation}
For instance, in the case of a unit disk $\mathbb{D}$, the bound \eqref{eq:Vk_exp} holds with $\tau = 1$ and
$C = \frac{1}{\sqrt{2\pi}}$.  Indeed, one has, for a bulk eigenfunction written in polar coordinates as $U_{(k)}^{(0)}(r,\theta)=\frac{1}{\sqrt{2\pi}} r^k\er^{\ir k \theta}$, and corresponding to an eigenvalue $\sigma^{(0)}_{(k)}=k$, 
\[
\left|U_{(k)}^{(0)}(r,\theta)\right| \le  \frac{1}{\sqrt{2\pi}} r^k =  \frac{1}{\sqrt{2\pi}} \er^{k\ln(1-\delta)}\le  \frac{1}{\sqrt{2\pi}}\er^{-\sigma^{(0)}_{(k)}\delta} ,
\]
where $\delta = 1-r = \operatorname{dist}\left((r,\theta),\partial\mathbb{D}\right)$ is the
distance to the boundary.
For domains with real-analytic boundary in $\mathbb{R}^d$, $d\ge 3$,  estimate \eqref{eq:Vk_exp} has been extended and made more precise in
\cite{Galkowski19}.

This analysis was further generalised  to the case
$\Lambda \ne 0$. 

\begin{theorem}[{\cite{Helffer22}, see also \cite{Daude21}}] 
Let  $\Omega\subset\R^d$ be a bounded domain with a
smooth connected boundary $\pa$, 
and let $U^{(\Lambda)}$ be a bulk eigenfunction of $\DtN_\Lambda$ with a unit $L^2(\pa)$ norm, corresponding to an eigenvalue $\sigma^{(\Lambda)}$. 
\begin{enumerate}
\item[\normalfont{(i)}] For any $p\in {\mathbb N}$ and any $\Lambda_0 <
\lambda_1^\Dir(\Omega)$, there exist constants $C=C_{p,\Lambda_0}$ and $S=S_{p, \Lambda_0}$ such that
for any $\Lambda \leq \Lambda_0$ and for any pair $\sigma^{(\Lambda)}$, $U^{(\Lambda)}$ with $\sigma^{(\Lambda)} >S$,
\begin{equation}\label{eq:HK1}
\left|U^{(\Lambda)}(x)\right| \leq \frac{C}{\left(\sigma^{(\Lambda)} \operatorname{dist}(x,\pa)\right)^{2p}}  \qquad \text{for all }x\in\Omega.
\end{equation} 
In other words, any such bulk eigenfunction
$U_k^{(\Lambda)}$ corresponding to a  sufficiently high eigenvalue $\sigma_k^{(\Lambda)}$ rapidly decays away
from the boundary.  
\item[\normalfont{(ii)}] If additionally the boundary $\pa$ is real-analytic, one can further improve
\eqref{eq:HK1} by switching from an arbitrary power decay to the
exponential decay.  Namely, for any $\eta > 0$ and any $\Lambda_0 < \lambda_1^\Dir$, there exist constants $C=C_{\eta,\Lambda_0}$, $\epsilon=\epsilon_{\eta,\Lambda_0}$ and $S=S_{\eta, \Lambda_0}$ such that for any pair $\sigma^{(\Lambda)}$, $U^{(\Lambda)}$ with $\sigma^{(\Lambda)} > S$, 
\[
\left|U^{(\Lambda)}(x)\right| \leq C \left(\sigma^{(\Lambda)}\right)^{\frac{d}{2} - \frac14} \exp\left(- (1-\eta)\sigma^{(\Lambda)} \min\{\epsilon,  \operatorname{dist}(x,\pa)\}\right) \quad \text{for all }x\in\Omega.
\]
\end{enumerate}
\end{theorem}
Helffer and Kachmar questioned whether the assumption of real-analytic
boundary was crucial for this statement; in particular, they posed the
question whether $C^\infty$ smoothness would be enough, which remains an open problem. A systematic
numerical study of the localisation for both smooth and polygonal
domains was provided in \cite{Chaigneau24}.

\section{Layer potentials and Green's functions}\label{sec:representation}

\subsection{The Dirichlet-to-Neumann map and  layer potentials}\label{subsec:laypot}

The Dirichlet-to-Neumann operator is closely related to the single and double layer potentials.
Indeed, a general solution $U(x)$ of the Helmholtz equation $-\Delta U - \Lambda U=0$ in $\Omega$ can be written as
\begin{equation}\label{eq:u_potential} 
U(x) = \int\limits_{\pa} \left( (\partial_{n_{y}} U)(y) \Phi_\Lambda(x-y) - U(y) \partial_{n_{y}} \Phi_\Lambda(x - y)\right) \dr y,
\qquad x\in \Omega,
\end{equation}
where $\Phi_\Lambda(x)$ is the fundamental solution of the Helmholtz
equation in $\R^d$, 
\[
\Phi_\Lambda(x) = 
\begin{cases}
\frac{K_{\frac{d}{2}-1}(|x|\sqrt{-\Lambda})}{(2\pi)^{\frac{d}{2}}} 
\left(\frac{|x|}{\sqrt{-\Lambda}}\right)^{1-\frac{d}{2}}\qquad&\text{if }\Lambda < 0,\\[3ex]
- \frac{\ln|x|}{2\pi}\qquad&\text{if }\Lambda = 0\text{ and }d=2,\\[2ex]
\frac{\Gamma(\frac{d}{2}-1)}{4\pi^{\frac{d}{2}} |x|^{d-2}}\qquad&\text{if }\Lambda = 0\text{ and }d\ge 3,\\[3ex]
\ir\frac{H_{\frac{d}{2}-1}^{(1)}(|x|\sqrt{\Lambda})}{4 (2\pi)^{\frac{d}{2}-1}} 
\left(\frac{|x|}{\sqrt{\Lambda}}\right)^{1-\frac{d}{2}}\qquad&\text{if }\Lambda > 0.
\end{cases}
\]
Here $H_n^{(1)}(z) = J_n(z) + \ir Y_n(z)$ is the Hankel function of the
first kind, $J_n(z)$ and $Y_n(z)$ are the Bessel functions of the
first and second kind, and $K_n(z)$ is the modified Bessel function of
the second kind.  For $\Lambda \leq 0$, the fundamental solution
decays at infinity (except for the planar case with $\Lambda = 0$, for
which its absolute value logarithmically increases to infinity).  In turn, for
$\Lambda > 0$, the fundamental solution satisfies the Sommerfeld
radiation condition at infinity \cite[p. 282]{McLean}, and the imaginary part of $\Phi_\Lambda$ is a solution of the homogeneous Helmholtz equation.

The first and second terms in \eqref{eq:u_potential} are referred
to as the single and double layer potentials, respectively \cite{steinbach2007numerical}.
The single layer potential  $\mathrm{V}: H^{-1/2}(\partial\Omega) \to H^{1/2}(\partial\Omega)$,  
is given by
\[
(\mathrm{V} \eta)(x) := \int_{\partial\Omega} \Phi_\Lambda(x-y) \eta(y) \, \dr y.
\]
The double layer potential $\mathrm{K}: H^{1/2}(\partial\Omega) \to H^{1/2}(\partial\Omega)$,  
is given by
\[
(\mathrm{K} \eta) (x):= \int_{\partial\Omega} (\partial_{n_{y}} \Phi_\Lambda(x-y)) \eta(y) \, \dr y.
\]
Using the limiting behaviour of the potentials, the standard representation \eqref{eq:u_potential} of 
a $\Lambda$-harmonic function $U(x)$ in terms of these potentials becomes on the boundary 
\[
U(x) = (\mathrm{V}(\partial_n U))(x) + \frac{U(x)}{2} - (\mathrm{K} U)(x),\qquad x\in \partial\Omega.
\]
If the single layer operator is invertible (which is always the case for $d \ge 3$ and, if needed, can be achieved by dilation for $d=2$), this equation yields a representation of the Dirichlet-to-Neumann 
operator
\begin{equation} \label{eq:DtN_VK}
\DtN_\Lambda = \mathrm{V}^{-1} \left(  \frac{1}{2} \mathrm{I} + \mathrm{K} \right): H^{1/2}(\partial\Omega) \to H^{-1/2}(\partial\Omega),
\end{equation}
where $\mathrm{I}$ denotes the identity operator.  

\begin{remark}\label{rem:fundsol}
Note that  $\Phi_\Lambda(x)$  is not a unique choice of a fundamental solution, which is defined up to an addition of a solution of the homogeneous Helmholtz equation. While the layer potential operators $\mathrm{K}$ and $\mathrm{V}$
depend on a particular choice of the fundamental solution, the Dirichlet-to-Neumann map $\DtN_\Lambda$ is independent of this choice. 
In particular, this explains that while for $\Lambda>0$ the function $\Phi_\Lambda$ is complex-valued (with its imaginary part being  a homogeneous solution), formula \eqref{eq:DtN_VK} defines an operator mapping real-valued functions to real-valued functions. 
\end{remark}

In \S\ref{sec:BEM}, we discuss how relation  \eqref{eq:DtN_VK} 
can be used for an efficient numerical computation of the eigenvalues
and eigenfunctions of $\DtN_\Lambda$ via boundary-element methods.

\subsection{The Dirichlet-to-Neumann map and Green's functions}\label{subsec:greens}

In various applications, the eigenfunctions of the Dirichlet-to-Neumann operator play an
important role as a basis for representing and approximating
$\Lambda$-harmonic functions and related quantities such as Green's
functions of the Laplace and Helmholtz equations (see
\cite{Auchmuty04}, \cite{Auchmuty06}, \cite{Auchmuty09}, \cite{Auchmuty13}, \cite{Auchmuty14}, \cite{Auchmuty15}, \cite{Auchmuty18}, \cite{Grebenkov19}
and references therein).  Such representations often exhibit a fast
convergence and are suitable for studying diffusion-controlled
reactions
\cite{Grebenkov20}, \cite{Grebenkov20c}, \cite{Grebenkov23a}.

Let $G^{\Dir}_\Lambda(x,x_0)$ denote the Dirichlet Green's
function for the Helmholtz equation (i.e., the kernel of the resolvent
$(-\Delta^{\Dir}_\Omega -\Lambda)^{-1}$) satisfying for any fixed
$x_0 \in \Omega$ and $\Lambda \notin \Spec\left(-\Delta^\Dir_\Omega\right)$
\[
\begin{cases}
-(\Delta + \Lambda) G^{\Dir}_\Lambda(x,x_0) = \delta(x-x_0), \qquad &x\in\Omega,  \\
G^{\Dir}_\Lambda(x,x_0) = 0,  \qquad &x\in\pa.
\end{cases}
\]
Similarly, one can introduce the Robin Green's function for any fixed
$x_0 \in \Omega$ and $\Lambda \notin \Spec\left(-\Delta^{\Rob,\gamma}_\Omega\right)$,
\begin{equation}\label{eq:tildeGq}
\begin{cases} 
-(\Delta + \Lambda) G^{\Rob,\gamma}_\Lambda(x,x_0)  = \delta(x-x_0), \qquad &x\in\Omega,\\
(\partial_n + \gamma) G^{\Rob,\gamma}_\Lambda(x,x_0)  = 0,  \qquad &x\in\pa, 
\end{cases}
\end{equation}
and Neumann Green's function for any $\Lambda \notin\Spec\left(-\Delta^{\Neu}_\Omega\right)$:
\[
\begin{cases} 
-(\Delta + \Lambda) G^{\Neu}_\Lambda(x,x_0)  = \delta(x-x_0), \qquad &x\in\Omega,\\
\partial_n G^{\Neu}_\Lambda(x,x_0)  = 0  \qquad &x\in\pa .
\end{cases}
\]

The Dirichlet Green's function $G^{\Dir}_\Lambda(x,x_0)$ for $\Lambda < 0$ can
be represented via the single layer potential as \cite[Theorem 5.2]{Auchmuty18},
\begin{equation}\label{eq:Ginf_fundam}
G^{\Dir}_\Lambda(x,x_0) = \Phi_\Lambda(x-x_0) - 
\sum\limits_{k=1}^\infty U_k^{(\Lambda)}(x) \int\limits_{\pa} \Phi_\Lambda(y - x_0) U_k^{(\Lambda)}(y)\, \dr y,
\end{equation}
where the convergence of the series is in the norm of $H^1(\Omega)$.  
In the same vein, the representation of the Robin Green's function
with $\Lambda < 0$ and $\gamma \geq 0$ involves both single and double layer potentials
\cite[Theorem 6.2]{Auchmuty18},
\begin{equation}\label{eq:Gq_fundam}
G^{\Rob,\gamma}_\Lambda(x,x_0) = \Phi_\Lambda(x-x_0) 
-\sum\limits_{k=1}^\infty \frac{U_k^{(\Lambda)}(x)}{\gamma + \sigma_k^{(\Lambda)}} 
\int\limits_{\pa} \left(\partial_{n_y} \Phi_\Lambda(y - x_0)  +  \gamma\, \Phi_\Lambda(y - x_0)\right) U_k^{(\Lambda)}(y)\, \dr y,
\end{equation}
where the convergence is again in the norm of $H^1(\Omega)$ (the
Neumann case is obtained by setting $\gamma = 0$).  It is also
applicable in the limit $\Lambda \to 0$ (for the Neumann case, one has
to remove the term with $k = 1$).

Despite the explicit form of the fundamental solution
$\Phi_\Lambda(x)$, the computation of the integrals in
\eqref{eq:Gq_fundam} is in general a difficult task.  Subtracting 
\eqref{eq:Gq_fundam} and \eqref{eq:Ginf_fundam} and re-arranging terms in 
the right-hand side, we get for $\Lambda < 0$ and $\gamma \ge 0$,
\[
\begin{split}
G^{\Rob,\gamma}_\Lambda(x,x_0) &- G^{\Dir}_\Lambda(x,x_0) \\  
&= 
\sum\limits_{k=1}^\infty \frac{U_k^{(\Lambda)}(x)}{\gamma + \sigma_k^{(\Lambda)}} 
\int\limits_{\pa} \left(\Phi_\Lambda(y - x_0) \partial_{n_y} U_k^{(\Lambda)}(y) -  U_k^{(\Lambda)}(y) \partial_{n_y} 
\Phi_\Lambda(y - x_0) \right) \dr y .
\end{split}
\]
According to \eqref{eq:u_potential}, the integral is just a
representation of the $\Lambda$-harmonic function
$U_k^{(\Lambda)}(x_0)$ in $\Omega$.  We therefore retrieve two
representations of the Robin Green's function proposed in
\cite{Grebenkov20}, for $\Lambda < 0$ and $\gamma \ge 0$:
\begin{equation}\label{eq:Gq_Ginf} 
\begin{split} 
G^{\Rob,\gamma}_\Lambda(x,x_0) & = G^{\Dir}_\Lambda(x,x_0) + 
\sum\limits_{k=1}^\infty \frac{U_k^{(\Lambda)}(x)  \, U_k^{(\Lambda)}(x_0)}{\gamma + \sigma_k^{(\Lambda)}}  \\  
&= G^{\Neu}_\Lambda(x,x_0) - 
\gamma \sum\limits_{k=1}^\infty \frac{U_k^{(\Lambda)}(x)  \, U_k^{(\Lambda)}(x_0)}{\sigma_k^{(\Lambda)}(\gamma + \sigma_k^{(\Lambda)})}  \,,
\end{split}
\end{equation}
where the second representation follows immediately from the first
one.  These representations are in terms of bulk eigenfunctions 
so that there is no
need for computing integrals with single and double layer potentials.
In turn, the explicit fundamental solution $\Phi_\Lambda(x)$ is
replaced by (usually unknown) Dirichlet or Neumann Green's function.
These two representations allow one to analyse the effect of the Robin
parameter $\gamma$ on the Robin Green's function and offer useful
probabilistic interpretations \cite{Grebenkov20}.  Setting $\gamma =
0$ and restricting $x_0 \in \pa$ and $x\in \pa$ cancels the term
$G^{\Dir}_\Lambda(x,x_0)$ in \eqref{eq:Gq_Ginf} and implies
\[
G^{\Neu}_\Lambda(x,x_0) = \sum\limits_{k=1}^\infty \frac{u_k^{(\Lambda)}(x_0)\, u_k^{(\Lambda)}(x)}{\sigma_k^{(\Lambda)}},
\qquad x,x_0\in\pa, \Lambda < 0, 
\]
i.e., the restriction of the Neumann Green's function
$G^{\Neu}_\Lambda(x,x_0)$ with $\Lambda < 0$ onto $\pa \times
\pa$ can be interpreted as the kernel of an integral operator on
$\pa$, which is the inverse of the Dirichlet-to-Neumann operator
$\DtN_\Lambda$.  This relation is used for further analysis and for
numerical computation of its spectral properties (see, e.g.,
\cite{Chaigneau24}).

\section{Extensions}\label{sec:extensions}

In this section, we briefly discuss several extensions of the basic
setting considered in  \S\ref{sec:basic}.  Some of these extensions are rather
straightforward, while others present additional challenges.

\subsection{Mixed boundary conditions}\label{sec:mixed}

While Dirichlet, Neumann and Robin boundary conditions are canonical
choices, many applications require more flexibility on surface properties
and often use {\em mixed} boundary conditions \cite{Sneddon}, \cite{Duffy}.
For instance, the so-called narrow escape problem is a special case of the mixed 
(Zaremba) problem, where we impose the Dirichlet boundary condition on a
subset $\Gamma_\Dir$ of the boundary $\pa$ and the Neumann boundary condition on
the remaining part $\Gamma_\Neu = \pa \setminus
\Gamma_\Dir$ \cite{Holcman14}, \cite{Schuss}, \cite{Holcman}.  When the subset $\Gamma_\Dir$ is
small, it can be considered as a {\em singular} perturbation of the
Neumann condition.  Many spectral properties have been studied for the
related Laplace operator (see
\cite{Ozawa81}, \cite{Mazya85}, \cite{Ward93}, \cite{Ward93b}, \cite{Kolokolnikov05}, \cite{Cheviakov11}, and
references therein).

In the same vein, one can deal with mixed boundary conditions for the
Dirichlet-to-Neumann operator.  One of the most studied settings is
the so-called \emph{sloshing problem} in fluid dynamics, see
\cite{Faltinsen}, \cite{Fox83}, \cite{Ibrahim05}, \cite{Mayer12}, \cite{Troesch65}, \cite{Mayrand20}, and also the summary in \cite[\S7.1.2]{Levitin}.  Let us
consider an inviscid, incompressible, irrotational fluid in an open
container $\Omega\subset \R^d$, $d=2$ or $d=3$, whose Lipschitz boundary $\pa$ is decomposed into the flat free surface $\Gamma$ and container ``walls'' $\mathcal{W}$, $\pa = \Gamma\sqcup  \mathcal{W}$, see Figure \ref{fig:sloshing}.  

\begin{figure}[!htbp]
\centering
\includegraphics{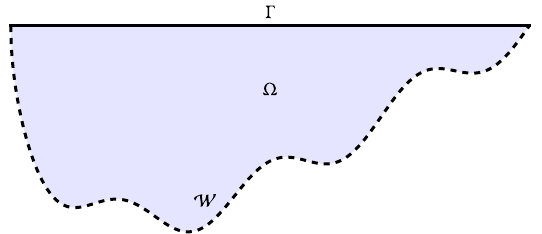}
\caption{A typical geometry of the sloshing problem in dimension $d=2$.}
\label{fig:sloshing}
\end{figure}

Considering small oscillations of the fluid due to gravity, and assuming that the surface tension on the free
surface $\Gamma$ is negligible, one deduces that the velocity potential $U$ satisfies the
mixed Steklov--Neumann eigenvalue problem
\begin{equation}\label{eq:sloshing}
\begin{cases}
\Delta U = 0\qquad&\text{in }\Omega,\\
\partial_n U = \sigma U\qquad&\text{on }\Gamma,\\
\partial_n U = 0\qquad&\text{on }\mathcal{W}.
\end{cases}
\end{equation}

The spectrum of \eqref{eq:sloshing} is discrete, $0=\sigma_1<\sigma_2\le\dots$, and the eigenpairs $\{ \sigma_k, U_k\}$ of
this  problem determine the modes of fluid
motion and its spatial frequencies.  
A lot of research has been dedicated to the study of the sloshing eigenvalues and eigenfunctions. 
For instance, Kozlov
and Kuznetsov investigated the dependence of the first nonzero sloshing
eigenvalue on the geometry of the free surface $\Gamma$
\cite{Kozlov04}, while the high spots (extrema) of the corresponding
sloshing mode in axisymmetric containers were numerically studied in
\cite{Kulczycki16}, see also \cite{Ammari20b}. In the planar case, the sharp asymptotics of the sloshing eigenvalues $\sigma_k$ as $k\to+\infty$, which takes into account the angles between the free surface $\Gamma$ and the walls $\mathcal{W}$, and which proves a long-standing conjecture of \cite{Fox83},  was recently obtained in \cite{LPPSslo}. This result plays a crucial role in the asymptotics \eqref{eq:Spoly} of the eigenvalues of the Steklov problem in curvilinear polygons.

More generally, let $\Omega\subset\mathbb{R}^d$ be a bounded Lipschitz domain, with its boundary decomposed into (up to) three subsets
\[
\pa =  \overline{\Gamma_\Dir \sqcup \Gamma_\Neu \sqcup \Gamma},
\]
possibly with $\Gamma_\Dir =\varnothing$ or $\Gamma_\Neu =\varnothing$. We will specify some further restrictions on this partition below. 
We define a \emph{partial Dirichlet-to-Neumann operator} $\DtN_{\Lambda, \Gamma}:  H^{1/2}(\Gamma)\to H^{-1/2}(\Gamma)$ which acts as 
\[
\DtN_{\Lambda, \Gamma}: u\mapsto \left.(\partial_n U)\right|_\Gamma,
\]
where $U=U^{(\Lambda)}$ is the solution of
\begin{equation}\label{eq:Helmholtz_mixed}
\begin{cases}
-\Delta U = \Lambda U\qquad&\text{in }\Omega,\\
U = 0\qquad&\text{on }\Gamma_\Dir,\\
U = u\qquad&\text{on }\Gamma,\\
\partial_n U = 0\qquad&\text{on }\Gamma_\Neu.
\end{cases}
\end{equation}

Without attempting to state the most general conditions on the boundary partition ensuring that the problem \eqref{eq:Helmholtz_mixed} is well-defined, see further \cite{Agranovich06}, \cite{Ott13}, we assume that each of the sets  $\Gamma$ and, if non-empty, $\Gamma_\mathrm{\#}$, $\mathrm{\#}\in\{\Dir,\Neu\}$, is open, has a finite number of disjoint components. If $d>2$, the interfaces between  the pairs of boundary components are assumed to be $(d-2)$-dimensional Lipschitz surfaces.  

Closely related to \eqref{eq:Helmholtz_mixed} is the homogeneous Zaremba spectral problem
\begin{equation}\label{eq:Zar_mixed}
\begin{cases}
-\Delta U = \Lambda U\qquad&\text{in }\Omega,\\
U = 0\qquad&\text{on }\widetilde{\Gamma}_\Dir:=\Gamma_\Dir \cup \Gamma,\\
\partial_n U = 0\qquad&\text{on }\Gamma_\Neu,
\end{cases}
\end{equation}
which has a discrete spectrum of eigenvalues $\left\{\lambda_k^{\Zar, \widetilde{\Gamma}_\Dir}\right\}$.

When defining the partial Dirichlet-to-Neumann operator $\DtN_{\Lambda, \Gamma}$, we first  additionally assume that $\Lambda$ is not an eigenvalue of the Zaremba Laplacian $-\Delta^{\Zar, \widetilde{\Gamma}_\Dir}_\Omega$ (that is, the one with the boundary conditions imposed as  in \eqref{eq:Zar_mixed}). This restriction can be lifted in the same manner as in \S\ref{subsec:DtNmap}: if $\Lambda=\lambda_k^{\Zar, \widetilde{\Gamma}_\Dir}$ for some $k$, we need to restrict the domain of  $\DtN_{\Lambda, \Gamma}$ to functions from $H^1(\Gamma)$ which are orthogonal in $L^2(\Gamma)$ to the normal derivatives,  restricted to $\Gamma$,  of the functions in the corresponding eigenspace of the Zaremba problem.

Informally, the spectral problem for the partial Dirichlet-to-Neumann map $\DtN_{\Lambda, \Gamma}$ can be stated as a \emph{weighted} boundary-value problem
\[
\begin{cases}
-\Delta U = \Lambda U\qquad&\text{in }\Omega,\\
\partial_n U = \sigma \rho U\qquad&\text{on }\partial\Omega,
\end{cases}
\]
where the ``weight function'' $\rho$ is defined on the boundary $\partial\Omega$ by
\[
\rho(x):=\begin{cases}
\infty,\qquad&\text{if }x\in\Gamma_\Dir,\\
0,\qquad&\text{if }x\in\Gamma_\Neu,\\
1,\qquad&\text{if }x\in\Gamma,
\end{cases}
\]
and the eigenfunctions are $u=U|_\Gamma$. 
Physical applications involving these problems were discussed in
\cite{Grebenkov20c}, \cite{Grebenkov23a}
(see also \S\ref{sec:applications}).  

Many basic properties of the Dirichlet-to-Neumann operator
$\DtN_\Lambda$ discussed in \S\ref{sec:basic} remain valid  for the partial Dirichlet-to-Neumann operator $\DtN_{\Lambda, \Gamma}$.  In particular, for
any $\Lambda \in \R$,
$\DtN_{\Lambda, \Gamma}$ is the self-adjoint operator, semi-bounded below, and with the discrete
spectrum of eigenvalues $\left\{\sigma_{\Gamma,k}^{(\Lambda)}\right\}_{k=1}^\infty$
accumulating to $+\infty$.  The eigenfunctions $\left\{u_{\Gamma,k}^{(\Lambda)}\right\}_{k=1}^\infty$
can be chosen to form a complete orthonormal basis in $L^2(\Gamma)$.  Additionally, if $\pa$ is smooth and $\Gamma$ has no boundary (that is, it is an isolated boundary component of $\pa$), then $\DtN_{\Lambda, \Gamma}$ is an elliptic pseudodifferential operator of order one on $\Gamma$. 

In a direct analogy with \eqref{eq:muk_minimax}, we have the minimax relation for the eigenvalues,
\begin{equation}\label{eq:muk_minimax2}
\sigma_{\Gamma,k}^{(\Lambda)} = \min_{\substack{\mathcal{L} \subset \mathcal{H}_\Lambda\left(\Omega, \Gamma_\Dir\right)\\\dim\mathcal{L} = k}}\  \max_{U \in \mathcal{L}\setminus\{0\}} 
\frac{\|\nabla U\|^2_{L^2(\Omega)} - \Lambda \| U\|^2_{L^2(\Omega)}}{\|U|_{\Gamma}\|^2_{L^2(\Gamma)}},  \qquad k\in\mathbb{N},
\end{equation}
where 
\[
\mathcal{H}_\Lambda\left(\Omega, \Gamma_\Dir\right)=\left\{U\in H^1(\Omega): -\Delta U-\Lambda U=0\text{ in }\Omega, \left.U\right|_{\Gamma_\Dir}=0|\right\}
\]
is the subspace of
$\Lambda$-harmonic functions in $H^1(\Omega)$ that vanish on $\Gamma_\Dir$.
Moreover, if $\Lambda < \lambda_1^{\Zar, \widetilde{\Gamma}_\Dir}$, then $\mathcal{H}_\Lambda\left(\Omega, \Gamma_\Dir\right)$
in \eqref{eq:muk_minimax2} can be replaced by the space
$H^1_0(\Omega, \Gamma_\Dir)$ of functions from $H^1(\Omega)$ that vanish on
$\Gamma_\Dir$.

Additionally, the analogues of Theorems \ref{thm:lambdadep}  and \ref{thm:analyticity} exist for the eigenvalues of the partial Dirichlet-to-Neumann maps: one needs only to replace in their statements Dirichlet eigenvalues $\lambda^\Dir$ by Zaremba eigenvalues $\lambda^{\Zar, \widetilde{\Gamma}_\Dir}$.

Many results presented in the previous sections can be easily
generalised for mixed boundary conditions, including Weyl's
law, asymptotic relation for the eigenvalues in the limits
$\Lambda\to 0$ and $\Lambda\to -\infty$, Courant's theorem, positivity
of the ground eigenfunction $U_1^{(\Lambda)}$, etc. For instance, the following modification of Corollary \ref{cor:Ndiff} holds for the number of non-positive eigenvalues of $\DtN_{\Lambda, \Gamma}$,
\begin{equation}\label{eq:Nlambda_0_mixed}
\mathcal{N}_{\DtN_{\Lambda,\Gamma}}(0)=\mathcal{N}_{-\Delta^{\Zar, \Gamma_\Dir}}(\Lambda)-\mathcal{N}_{-\Delta^{\Zar, \widetilde{\Gamma}_\Dir}}(\Lambda).
\end{equation}

At the same time, the presence of Dirichlet and/or Neumann boundary
conditions affects some basic properties of the spectrum.
For instance, if $\Gamma_\Dir \ne \varnothing$, the smallest eigenvalue
$\sigma_{\Gamma, 1}^{(0)}$ is strictly positive and not zero. Moreover, we have 
$\sigma_{\Gamma, 1}^{(\Lambda)}>0$ for all $\Lambda<\lambda_1^{\Zar, \Gamma_\Dir}$. At the same time, we cannot guarantee that 
$\sigma_{\Gamma, 1}^{(\Lambda)}$ is negative for all $\Lambda>\lambda_1^{\Zar, \widetilde{\Gamma}_\Dir}$ as it may happen that the right-hand side of \eqref{eq:Nlambda_0_mixed} vanishes, for example if $\Gamma_\Neu=\varnothing$ and $\Gamma$ is small.

In turn, mixed boundary conditions bring additional properties that
were not present in the conventional setting.  Let us look at  the domain
monotonicity for mixed problems on the basis of the following example of  \cite{Banuelos10}.  For three bounded simply-connected domains $D_0 \subset D_1
\subset D_2 \subset \R^d$, we define $\Omega_1 = D_1 \backslash D_0$
and $\Omega_2 = D_2 \backslash D_0$, see Fig. \ref{fig:domain_monotone}.
Clearly, $\Omega_1 \subset \Omega_2$; moreover, denoting $\Gamma = \partial D_0$, we have, for $i=1,2$, 
$\partial \Omega_i = \partial D_i \sqcup \Gamma$.  

\begin{figure}[!htbp]
\centering
\includegraphics{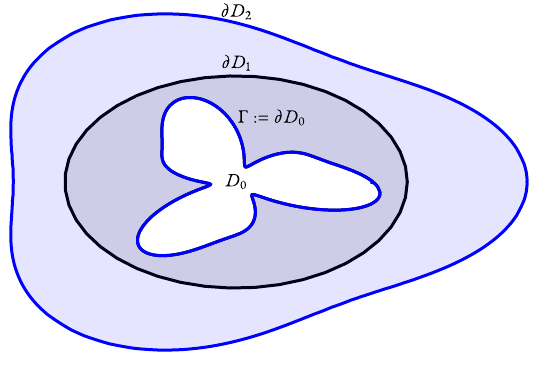}
\caption{
Domain monotonicity for the case of mixed boundary
conditions: three bounded domains $D_0\subset D_1\subset D_2$ and two bounded domains $\Omega_1 := D_1 \backslash D_0$ and
$\Omega_2 := D_2\backslash D_0$, $\Omega_1 \subset \Omega_2$.
\label{fig:domain_monotone}}
\end{figure}

We impose the Steklov
condition on $\Gamma$ and either Dirichlet or Neumann boundary
condition on $\partial D_i$.  Then the eigenvalues
$\sigma_{\Gamma, k}^{(\Lambda)}(\Omega_i)$ of the partial Dirichlet-to-Neumann operator
$\DtN_{\Gamma,\Lambda}$  satisfy the following domain
monotonicity properties, see \cite{Banuelos10} for the case $\Lambda=0$ which also covers more general geometries.

\begin{theorem}\label{thm:mixedDM}\ 
\begin{enumerate}
\item[\normalfont(i)] Let us impose the Dirichlet boundary conditions on $\partial D_1$ and $\partial
D_2$, therefore setting $\Gamma_\Dir = \partial D_i$ and $\Gamma_\Neu=\varnothing$ for both domains $\Omega_i$, $i=1, 2$. Assume additionally that $\Lambda < \lambda_1^\Dir(\Omega_2)$. Then
\[
\sigma_{\Gamma, k}^{(\Lambda)}(\Omega_1) \geq \sigma_{\Gamma,k}^{(\Lambda)}(\Omega_2)\qquad\text{for all }k \in \mathbb{N}.
\]
\item[\normalfont(ii)] Let us impose the Neumann boundary conditions on $\partial D_1$ and $\partial
D_2$, therefore setting $\Gamma_\Neu= \partial D_i$ and $\Gamma_\Dir=\varnothing$ for both domains $\Omega_i$, $i=1, 2$. Assume additionally that $\Lambda < 0$. Then
\[
\sigma_{\Gamma, k}^{(\Lambda)}(\Omega_1) \leq \sigma_{\Gamma,k}^{(\Lambda)}(\Omega_2)\qquad\text{for all }k \in \mathbb{N}.
\]
\end{enumerate}
\end{theorem}

The proof of Theorem \ref{thm:mixedDM} follows almost directly the arguments of \cite{Banuelos10}  using the minimax relations \eqref{eq:muk_minimax2}.

\subsection{Exterior problems}\label{sec:exterior}

The Dirichlet-to-Neumann operators  and the corresponding Steklov-type problems are usually considered in the interior of bounded domains. However, in certain applications, in particular to diffusion problems, the exterior Dirichlet-to-Neumann map naturally arises, see \cite{Grebenkov21a, Grebenkov25}. We briefly overview this case below.
Let  $\Omega = \R^d \setminus \overline{\Omega_0}$,
where $\Omega_0$ is a bounded open set with Lipschitz boundary.
The major issue in defining the exterior Dirichlet-to-Neumann map for the Helmholtz equation with parameter $\Lambda$ is
to find the appropriate conditions satisfied by the solutions at infinity. Indeed, without further constraints, the $\Lambda$-harmonic extension to the exterior is not uniquely defined. As it turns out, the situation is quite different depending on the sign of $\Lambda$.
 
The case $\Lambda < 0$  has been studied  in \cite{Auchmuty13}, \cite{Auchmuty14}.
It was shown that the exterior Dirichlet-to-Neumann operator $\DtN_\Lambda^{\mathrm{ext}}$ can
be introduced using  the solutions of the Helmholtz equation $-(\Delta +
\Lambda)U = 0$  in the standard Sobolev space  $H^1(\Omega)$.
Moreover, as was proved in \cite{BardosMerigot}  (see also \cite[\S 3.3]{Bundrock25}) 
these solutions
exhibit exponential decay at infinity.  In short, for $\Lambda<0$ the exterior problem can be defined in  a similar way  as the interior one  (in particular, the variational characterisation \eqref{eq:muk_minimax1} holds) and no substantial difficulties arise. 

The case  $\Lambda = 0$ is in a way the most challenging.  It has been considered in detail in \cite{Bundrock25}. 
Indeed, the  solutions of the Laplace equation $\Delta U = 0$
may not decay at all or decay too slowly to ensure that $U \in H^1(\Omega)$.
For instance, the first Steklov eigenfunction of an exterior domain in two dimensions is constant.
Different approaches were developed to  define the exterior problem for the Dirichlet-to-Neumann operator $\DtN_0$.
One possibility is to introduce a parameter $\Lambda<0$ as above and consider the limit as $\Lambda \to 0^{-}$ \cite{Christiansen22}, \cite{Grebenkov25}.
Another option is to  consider the Steklov problem in a bounded
domain $\Omega_R = \Omega \cap B_R$, where $B_R$ is a  ball of radius $R$, with either Dirichlet or Neumann
boundary condition imposed on the sphere $\partial B_R$, and
consider  the limits of the associated partial Dirichlet-to-Neumann operators $\DtN_{\Lambda, \partial\Omega}$ in $\Omega_R$ as $R\to\infty$ \cite{Arendt15}. Note that in dimension two the limits for both Dirichlet and Neumann condition coincide, and in higher dimensions the limiting operators differ by a rank one perturbation. As follows from the results of \cite{Bundrock25}, the Dirichlet boundary condition leads to  a more ``canonical'' choice of the exterior operator.  In particular,  the spectra of the operators  $\DtN_\Lambda^{\mathrm{ext}}$ as  $\Lambda  \to 0^{-}$ converge to the spectrum of the exterior operator corresponding to the Dirichlet conditions in the construction above.

In fact, there are several other constructions that lead to an equivalent definition of the exterior problem. One can 
use the layer potential approach, cf. \S \ref{sec:representation}. In dimensions $ d \geq 3$ one can also consider  a modification of the variational characterisation   \eqref{eq:muk_minimax1} with the space  $H^1(\Omega)$ replaced by a certain space of  functions
with finite energy  (i.e., functions with their gradients in $L^2(\Omega)$ but not necessarily in $L^2(\Omega)$ themselves), see
\cite{Auchmuty14}, as well as  \cite{Auchmuty18}, \cite{Xiong23}. In dimension two,  an application of  a conformal map (e.g.\ a composition of an inversion and a reflection) transforms an exterior problem into a weighted interior problem in a (punctured) bounded domain.  We refer to \cite{Bundrock25} for further details and a comparison of different approaches.

Finally, if $\Lambda>0$, in order to define an exterior Dirichlet-to-Neumann operator   one  needs  to consider the solutions of the Helmholtz equation satisfying the Sommerfeld radiation condition, cf.  \S \ref{sec:representation}.  This leads to a non-self-adjoint eigenvalue problem, which is beyond the scope of this survey.
The non-self-adjointness can be easily  seen from the layer potential representation. Note that unlike the interior case (see  Remark \ref{rem:fundsol}), the radiation condition at infinity determines the choice of a complex-valued fundamental solution.

\begin{remark}
In  the  cases $\Lambda \le 0$  discussed  above, the Dirichlet-to-Neumann operators in
an exterior domain have a discrete spectrum. In particular, when solving
diffusion or heat transfer problems in exterior domains, one can still
employ spectral decompositions over eigenfunctions $u_k^{(\Lambda)}$ on the boundary 
and use their $\Lambda$-harmonic extensions $U_k^{(\Lambda)}$ inside the domain, despite
the fact that $\Omega$ is unbounded, see
\S\ref{sec:representation}.
\end{remark}

\subsection{Complex parameter $\Lambda$}\label{sec:complex}

Our previous discussion was focused on real $\Lambda$, for which the
Dirichlet-to-Neumann operator $\DtN_\Lambda$ is self-adjoint.  Allowing the parameter
$\Lambda$ in the Helmholtz equation to be complex can represent damping effects onto
wave propagation.  In diffusion-oriented applications, the complex
$\Lambda$ naturally emerges in the computation of inverse Laplace
transforms \cite{Grebenkov20}.  In this section, we
briefly mention a few 
results on $\DtN_\Lambda$ in this
setting.

The Dirichlet-to-Neumann operator for a
complex parameter was studied in \cite{Bogli22}. 
Let $\Omega \subset \R^d$ be a bounded Lipschitz domain and $\Lambda
\in \C \backslash \Spec\left(-\Delta^\Dir_\Omega\right)$.  The operator
$\DtN_\Lambda$ is closed, densely defined and 
$m$-sectorial, with a
compact resolvent in $L^2(\pa)$. 
In particular, its spectrum consists
of eigenvalues of finite algebraic multiplicity.  Nevertheless, $\DtN_\Lambda$ is no longer self-adjoint for $\Lambda\not\in\mathbb{R}$.

Additionally, the Dirichlet-to-Neumann operator $\DtN_\Lambda$ is
meromorphic with respect to the parameter $\Lambda\in \C$.  Its
singularities are poles of finite order, are real,  and coincide with the Dirichlet  eigenvalues of $\Omega$.
For $\Lambda\in \C \setminus \mathbb{R}$,
$\DtN_{\Lambda^*} = \left(\DtN_\Lambda\right)^*$, and the corresponding
quadratic forms are holomorphic of type A (see \cite{Bogli22} for
details).
Even for complex $\Lambda$, the Robin--Dirichlet-to-Neumann duality  still holds:
for any $\sigma\in \C$ and any $\Lambda\in \C \backslash
\Spec\left(-\Delta^\Dir_\Omega\right)$, we have that $\Lambda \in
\Spec\left(-\Delta_\Omega^{\Rob,-\sigma}\right)$ if and only if $\sigma \in
\Spec\left(\DtN_\Lambda\right)$.  Moreover, $U$ is an eigenfunction of
$-\Delta_\Omega^{\Rob,-\sigma}$ corresponding to $\Lambda$ if and
only if $U|_{\pa}$ is an eigenfunction of $\DtN_\Lambda$ corresponding
to $\sigma$.

As an example, let us consider the spectral problem for $\DtN_\Lambda$ with complex $\Lambda$ in the unit disk $\mathbb{D}$. The eigenvalues are computed using the same formula as in 
\eqref{eq:sigmadisk},
\[
\sigma^{(\Lambda)}_{(n)}= \frac{\sqrt{\Lambda}J'_n(\sqrt{\Lambda})}{J_n(\sqrt{\Lambda})},\qquad n=0,1,2,\dots,
\]
and have multiplicity one if $n=0$ and two if $n>0$. The top plot in Figure \ref{fig:muk_diskC} shows the modulus of $\sigma^{(\Lambda)}_{(0)}$ plotted as a function of a complex $\Lambda$: as expected, we see blow-ups at real eigenvalues $j_{0,1}^2\approx 5.78319$, $j_{0,2}^2\approx 30.4713$, and $j_{0,3}^2\approx 74.887$ of the Dirichlet Laplacian in the disk. The plots of the real and imaginary parts of this eigenvalue are very similar. The bottom figures show the trajectories swept by $\sigma^{(\Lambda)}_{(0)}$ as $\Lambda$ changes along a circle of a given radius in a complex plane. The macro view in the left image shows that the eigenvalue curves resemble circles, however zooming in near the origin in the right image demonstrates a more complicated behaviour. 

\begin{figure}[!htbp]
\centering
\includegraphics{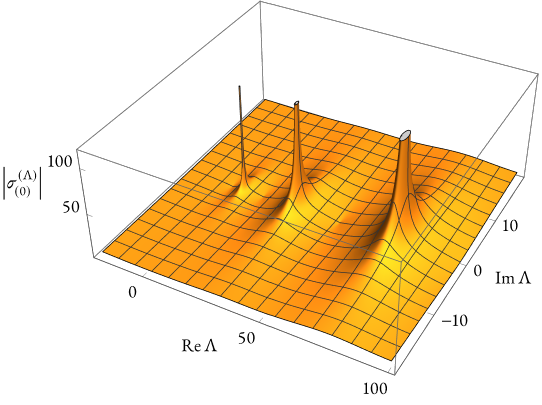}

\includegraphics{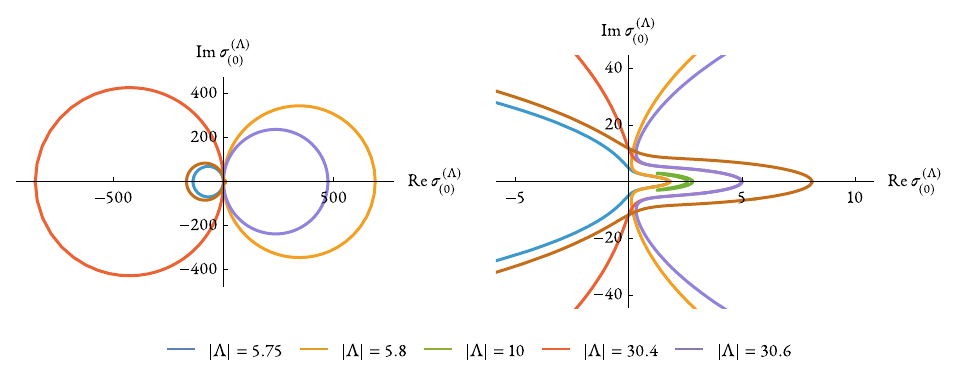}
\caption{The absolute value of $\sigma^{(\Lambda)}_{(0)}$ plotted as a function of complex $\Lambda$ (top). The images of circles of a given radius under the map $\Lambda\mapsto \sigma^{(\Lambda)}_{(0)}$ (bottom left), and their zoom near the origin (bottom right).
\label{fig:muk_diskC}}
\end{figure}

\section{Numerical computation and applications}\label{sec:applications}

In this section, we briefly recall some usual numerical techniques for computing the eigenvalues and
eigenfunctions of the Dirichlet-to-Neumann operator. After that, we discuss two applications: the domain decomposition, see \S\ref{sec:decomp}, and the
encounter-based approach to diffusion-controlled reactions, see \S\ref{sec:reactions}.

\subsection{Finite-element methods}\label{sec:FEM}

Finite-difference and finite-element methods allow one to discretise
the spectral problem \eqref{eq:SteklovLambda} on a regular lattice or
a mesh and thus to reduce the Helmholtz equation to a large set of
linear algebraic equations.  For this purpose, one can either employ
finite-difference approximations for the Laplace operator $\Delta$ and
the normal derivative $\partial_n$, or a weak formulation of the
Helmholtz equation with appropriate basis functions.  In both cases,
one constructs a matrix representing the spectral problem for
$\DtN_\Lambda$ at a given spatial resolution, that needs to be
diagonalised to approximate the eigenvalues $\sigma_k^{(\Lambda)}$ and
the eigenfunctions $U_k^{(\Lambda)}$.  Flexibility is one of the
advantages of these methods that allows one to deal in the same way
with more general second-order elliptic operators.  One potential drawback
is the need to carefully refine the mesh near the
boundary to resolve bulk eigenfunctions localised there.

Various improvements
have been proposed to amend these limitations.  For instance, an
isoparametric variant of the finite-element method for solving Steklov
eigenvalue problems in $\R^d$ for second-order self-adjoint elliptic
differential operators was developed in \cite{Andreev04}.  Other
modifications and improvements include a virtual element method in
planar domains \cite{Mora15}, a two-grid discretisation scheme
\cite{Li11}, \cite{Bi11}, a finite element multi-scale discretisation with an
adaptive algorithm based on the shifted inverse iteration \cite{Bi16},
an iterative multilevel approach \cite{Xie13}, nonconforming
finite-element methods \cite{Yang09}, \cite{Li13}, and
discontinuous Galerkin methods for non-selfadjoint Steklov
eigenvalue problems \cite{MengMei}, \cite{Wang22}.  Note that the numerical
efficiency of these methods is typically tested in the planar case only.  An
extension of the finite-element method to exterior problems by using a
transparent boundary condition was discussed in \cite{Grebenkov25}.

\subsection{Boundary-element methods and layer potentials}\label{sec:BEM}

Another common numerical approach to finding the spectrum of the
Dirichlet-to-Neumann operator $\DtN_\Lambda$ is to  reformulate
the corresponding eigenvalue problem in terms of an equivalent boundary
integral equation (see \cite{Tang98},  \cite{Chen20},  \cite{Bruno20},  \cite{Akhmetgaliyev17}
and references therein). This approach extends to a more general
non-self-adjoint setting of complex $\Lambda$ \cite{Ma22}, \cite{Nigam25}.
Using  \eqref{eq:DtN_VK}, one can rewrite the Steklov condition as
\begin{equation}\label{eq:BIE}
\frac12 u(x) + \int\limits_{\pa} \partial_{n_{y}} \Phi_\Lambda(x - y) u(y) \, \dr y 
= \sigma \int\limits_{\pa} \Phi_\Lambda(x-y) u(y) \, \dr y,  \qquad x\in \pa,
\end{equation}
where $u(x)$ is the restriction of $U(x)$ onto $\pa$.
The boundary integral equation \eqref{eq:BIE} represents a generalised
eigenvalue problem  
which determines the eigenvalues and eigenfunctions of the
Dirichlet-to-Neumann operator $\DtN_\Lambda$.  In practice, this
equation needs to be discretised and regularised to amend
singularities of the fundamental solution and its derivatives.  In
this way, the eigenvalue problem for the Dirichlet-to-Neumann map is
reduced to a system of linear algebraic equations; the associated
matrix can then be diagonalised to approximate the eigenvalues
$\sigma_k^{(\Lambda)}$, whereas the eigenvectors of that matrix allow
one to approximate the eigenfunctions $u_k^{(\Lambda)}$.  In turn, the
bulk eigenfunctions $U_k^{(\Lambda)}$ are then found via
\eqref{eq:u_potential}.  The advantage of this method is the
dimensionality reduction when the partial differential equation in
$d$-dimensional domain $\Omega$ is replaced by an integral equation on
the $(d-1)$-dimensional boundary $\pa$; as a consequence, the sizes of
the resulting matrices are substantially reduced, but they are no longer
sparse.  Moreover, these matrices are in
general ill-conditioned,  which requires additional computational
treatment (e.g., a singular value decomposition).
For planar domains, several improvements were proposed such as
weighting with the fundamental solutions of the Laplace equation
\cite{Turk} or a fast wavelet collocation method \cite{Long}.  

\subsection{Method of fundamental solutions}\label{sec:fundsol}

Yet another technique relies on the method of fundamental solutions,
see \cite{Kupradze64}, and employed by Antunes and Alves in their
study of various eigenvalue problems \cite{Alves09}, \cite{Alves13},
\cite{Antunes11}.  In this method, the solution of the Laplace (or
Helmholtz) equation inside the domain is approximated by a linear
combination of fundamental solutions $\Phi_\Lambda(x-x_k)$ at some
points $x_k$.  The number of these points, their locations, and the
weights of terms are adjusted to approximate the boundary conditions
with required accuracy.  The method is mesh-free, which allows fast
and accurate computations, even though the underlying matrices
can again be ill-conditioned (see \cite{Barnett08} for stability and
convergence analysis of the method of fundamental solutions for
Helmholtz problems on analytic domains).  This method was adopted to solving the Steklov and related spectral problems  in \cite{Bogosel16}.

\begin{remark} 
It is worth noting that in all methods discussed in \S\S\ref{sec:FEM} --\ref{sec:fundsol}, it is
computationally challenging to maintain accuracy or even identify negative eigenvalues of $\DtN_\Lambda$ for  $\Lambda$  near the Dirichlet
eigenvalues  (see \cite{Nigam25} for more details). 
\end{remark}

\subsection{Domain decomposition}\label{sec:decomp}

Dirichlet-to-Neumann operators are often used as a tool for analysing
and solving spectral and scattering problems in unbounded domains via
the method of \emph{domain decomposition}.  A typical example of such
an approach is a waveguide-like domain $\Omega\subset\mathbb{R}^d$
with regular ends, defined as an interior of $\overline{\Omega_0\sqcup
\Omega_1\sqcup\dots\sqcup \Omega_m}$, where $\Omega_0$ is a bounded
Lipschitz domain, and $\Omega_j$, $j=1,\dots,m$, are semi-infinite
cylinders with the straight bases
$\Gamma_j:=\overline{\Omega_0}\cap\overline{\Omega_j}$, see Figure
\ref{fig:decomp}.

\begin{figure}[!htbp]
\centering
\includegraphics{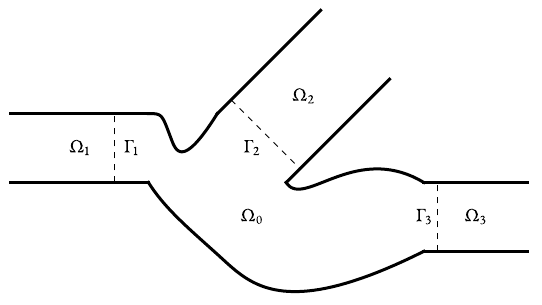}
\caption{An example of a domain decomposition}
\label{fig:decomp}
\end{figure}

We consider the spectral problem for the Laplacian
$-\Delta^{\mathcal{B}}$ in $\Omega$ with some self-adjoint boundary
condition $\left.\mathcal{B} U\right|_{\pa}=0$ (which may be a Dirichlet,
a Neumann, or a Robin one, or a mix of those imposed on parts of the boundary), and some appropriate
conditions at infinity which are determined by whether we are looking
for eigenvalues, either below or embedded into the essential spectrum,
or resonances (the poles of the scattering matrix).  We reduce the
problem to the one in a bounded domain $\Omega_0$. First, we
construct, for $j=1,\dots,m$, the partial Dirichlet-to-Neumann maps
$\DtN_{\Lambda, \Gamma_j}$ on $\Gamma_j$ with respect to an unbounded
domain $\Omega_j$ and subject to the boundary conditions on
$\pa_j\setminus\Gamma_j$ and at infinity inherited from the original
problem. This is typically done by separation of variables.  Then,
using elliptic regularity, we reformulate the original spectral
problem as
\begin{equation}\label{eq:domdec}
\begin{cases}
-\Delta U =\Lambda U\qquad&\text{in }\Omega_0,\\
 \mathcal{B} U = 0&\text{on }\pa_0\setminus\Gamma,\\
 \partial_n U + \DtN_{\Lambda, \Gamma_j}\left(U|_{\Gamma_j}\right)=0&\text{on }\Gamma_j, j=1,\dots,m,
 \end{cases}
\end{equation}
where $\Gamma=\cup_{j=1}^m \Gamma_j$ denotes the ``total''
interface.  Equivalently, problem \eqref{eq:domdec} can be re-stated
as a problem on $\Gamma$ only,
\[
\left(\DtN_{\Lambda, \Gamma}  + \sum_{j=1}^m \DtN_{\Lambda, \Gamma_j}\chi_j\right) u  = 0\qquad\text{on }\Gamma, 
\]
where $u:=U|_{\Gamma}$, $\chi_j$ is the characteristic function of
$\Gamma_j$, and $\DtN_{\Lambda, \Gamma}$ is the partial
Dirichlet-to-Neumann map on $\Gamma$ with respect to the interior
domain $\Omega_0$, with the inherited boundary condition on
$\pa_0\setminus\Gamma$.

For various examples of this approach, see \cite{Smith96},  \cite{LevMar},  \cite{Exner},  \cite{Delitsyn18},  \cite{Delitsyn22},  \cite{LevSto}. 

\subsection{Application to diffusion-controlled reactions}\label{sec:reactions}

The eigenvalues and eigenfunctions of the Dirichlet-to-Neumann
operator $\DtN_\Lambda$ have recently found immediate applications in
the theory of diffusion-controlled reactions in chemical physics
\cite{Grebenkov20}, \cite{Grebenkov23c}.  For instance, in most biochemical
reactions in living cells, the reactants have first to diffuse towards
their partners, and this diffusion step can be the rate-limiting
factor \cite{Berg85}, \cite{Lindenberg}.  This is also a typical scenario of
heterogeneous catalysis when the molecules need to reach a catalytic
boundary to be chemically transformed.  The most basic setting of
ordinary diffusion in a bounded domain $\Omega$ with partially
reactive boundary $\pa$ is usually described by the heat kernel
$K^{\Rob,\gamma}(x, x_0; t)$ (also known as the propagator in
physics literature) which satisfies the heat equation with Robin
boundary condition,
\begin{equation}\label{eq:heat}
\begin{cases}
\partial_t K^{\Rob,\gamma}(x, x_0; t) = \Delta K^{\Rob,\gamma}(x, x_0; t),\qquad &x\in \Omega,\\
\partial_n K^{\Rob,\gamma}(x, x_0; t) + \gamma K^{\Rob,\gamma}(x, x_0; t)  = 0, \qquad &x\in \pa,
\end{cases}
\end{equation}
where the diffusion coefficient is set to $1$, $\gamma > 0$ is a
constant reactivity of the boundary, and $x_0 \in \Omega$ is a source
point entering through the initial condition $K^{\Rob,\gamma}(x,
x_0; 0) = \delta(x-x_0)$.  The heat kernel admits the standard
probabilistic interpretation: if a particle started from a point
$x_0$ at time $0$, it is found in a $\dr x$ vicinity of a point $x$
at time $t$ with probability $K^{\Rob,\gamma}(x, x_0; t)\,\dr x$.
As a consequence, the heat kernel was used to describe various
properties of diffusion-controlled reactions such as time evolution of
the concentration of particles, the survival probability of a single
particle, the distribution of its first-reaction time, etc., see
\cite{Redner}, \cite{Metzler}, \cite{Grebenkov}.
The common way to analyse heat kernels consists in reducing the
parabolic problem \eqref{eq:heat} to the elliptic problem \eqref{eq:tildeGq}
via the Laplace transform:
\[
G^{\Rob,\gamma}_\Lambda(x, x_0) = \int\limits_0^\infty \er^{\Lambda t} \, K^{\Rob,\gamma}(x, x_0; t) \, \dr t,
\qquad \Lambda \leq 0.
\]
The Robin Green's function $G^{\Rob,\gamma}_\Lambda(x, x_0)$ with
negative $\Lambda$ also appears in the description of steady-state
diffusion in a reactive medium, in which the diffusing particles
spontaneously disappear in the bulk with the rate proportional to
$|\Lambda|$.
 
In the above conventional description via the heat kernel or Green's function, the dependence on the reactivity
parameter  $\gamma$, which enters through the Robin boundary condition, is implicit; in particular, if one needs to optimise
the system with respect to this parameter, Green's function or related quantity needs to be computed numerically
for each value of $\gamma$. 
This fundamental issue is resolved when using the spectral
expansions \eqref{eq:Gq_Ginf} over the bulk eigenfunctions
$U_k^{(\Lambda)}$, in which $\gamma$ appears {\em explicitly}.  This
spectral form clarifies the role of the reactivity and opens a
possibility for addressing various optimisation problems in chemical
engineering.
This key feature was systematically employed to investigate the role
of the boundary and its reactivity onto efficiency of
diffusion-controlled reactions \cite{Grebenkov23c}.  Moreover, the
encounter-based approach introduced in \cite{Grebenkov20} allows one
to generalise the theory to deal with more sophisticated boundary
reactions, far beyond the conventional ones described by Robin
boundary condition.   
In fact, in most applications, the surface reaction occurs after a
``sufficient number'' of arrivals of the particle on the reactive
boundary.  As the statistics of such arrivals is characterised by the
boundary local time $\ell_t$  (see Remark \ref{rem:Sato}), the reaction event is triggered at a
random time $\tau = \inf\left\{t > 0: \ell_t > \hat{\ell}\right\}$ when
$\ell_t$ exceeds an independent random threshold $\hat{\ell}$.  The
position of the reaction event can thus be obtained by averaging
the heat trace $\Sigma^{(\Lambda)}(x,\hat{\ell};x_0)$ given by \eqref{eq:Sigma_def} over random realisations of
the threshold $\hat{\ell}$ so that the distribution of this threshold
determines the nature of modelled surface reactions.  For instance, if
$\hat{\ell}$ obeys an exponential distribution with the rate $\gamma$,
the average yields
\[
\int\limits_0^\infty \gamma \er^{-\gamma \ell} \, \Sigma^{(\Lambda)}(x,\ell;x_0)\, \dr\ell
= \sum\limits_{k=1}^\infty u_k^{(\Lambda)}(x) \, u_k^{(\Lambda)}(x_0) \frac{\gamma}{\gamma + \sigma_k^{(\Lambda)}} 
= \gamma G_\Lambda^{\Rob,\gamma}(x,x_0) 
\]
for $x, x_0 \in \pa$, where the second equality follows from the first relation in \eqref{eq:Gq_Ginf}. This is precisely the probability flux density
that determines the position of the reaction event on the boundary.
One sees that the exponential law of the threshold corresponds to the
conventional setting of diffusion-controlled reactions with a constant
reactivity $\gamma$ described by the Robin boundary condition.  In
turn, other distributions of the threshold allows one describe more
sophisticated surface reactions with encounter-dependent reactivity
such as, e.g., progressive activation or passivation of the boundary
by diffusing particles \cite{Grebenkov20}, \cite{Grebenkov23b}.  In this
general setting, the conventional tools such as the Robin heat kernel
or Green's function are not suitable anymore, whereas the spectral
expansion \eqref{eq:Sigma_def} in terms of the eigenfunctions
$u_k^{(\Lambda)}$ 
plays the
central role. The semigroup of the Dirichlet-to-Neumann operator, $\exp(-\ell
\DtN_\Lambda)$, and the spectral expansion \eqref{eq:Sigma_def} of the
associated heat kernel $\Sigma^{(\Lambda)}(x,\ell;x_0)$ are thus the
main ``building blocks'' of this framework \cite{Grebenkov20c}.
In addition, one gets access to the statistics of encounters between
diffusing particles and to their boundary local time
\cite{Grebenkov19c},  \cite{Grebenkov21a}.

\appendix
\section{Ordering the Dirichlet-to-Neumann eigenvalues of the disk for $\Lambda<0$}\label{sec:ordering}

In this Appendix, we provide an elementary proof that the eigenvalue curves
$\sigma_{(n)}^{(\Lambda)}$ of the Dirichlet-to-Neumann operator for the unit disk $\mathbb{D}$, given by  \eqref{eq:sigmadisk}, do not cross each other for $\Lambda < \lambda_1^\Dir(\mathbb{D})= j_{0,1}^2$. More precisely, we prove the following
\begin{theorem}
Let $n, m\in\{0,1,2,\dots\}$, with $n<m$.  Then $\sigma_{(n)}^{(\Lambda)}\ne \sigma_{(m)}^{(\Lambda)}$ for all $\Lambda<j_{n,1}^2$, and the order of eigenvalues 
\eqref{eq:orderdisk} is preserved for all $\Lambda < j_{0,1}^2$.
\end{theorem}

\begin{proof}
The statement for $\Lambda=0$ is obviously true. Let us first consider the case $\Lambda < 0$. Define, for $\nu\ge 0$ and $z>0$,  the function 
\[
\Phi(\nu, z):=\frac{z I'_\nu(z)}{I_\nu(z)}=z\frac{\dr\left(\ln I_\nu(z)\right)}{\dr z},
\]
so that 
\[
\sigma_{(n)}^{(\Lambda)} = \Phi\left(n, \sqrt{-\Lambda}\right).
\]
We will show that for a fixed $z>0$, the function $\Phi(\nu, z)$ is strictly monotone increasing in $\nu$ by evaluating 
\[
\frac{\partial \Phi(\nu, z)}{\partial\nu} = z \frac{\partial^2\left(\ln I_\nu(z)\right)}{\partial\nu\partial z}. 
\]
Using the following representation \cite{Gonzalez-Santander18} of the derivatives of the modified Bessel functions with respect to order,  
\[
\frac{\partial I_\nu(z)}{\partial \nu} = -2\nu \left(I_\nu(z) \int\limits_z^\infty \frac{K_\nu(t) I_\nu(t)}{t}\,\dr t
+ K_\nu(z) \int\limits_0^z \frac{I_\nu^2(t)}{t}\, \dr t\right),
\]
we get
\[
\begin{split}
\frac{\partial \Phi(\nu, z)}{\partial\nu} & = z\frac{\partial\left(\frac{1}{I_\nu(z)}\,\frac{\partial I_\nu(z)}{\partial\nu}\right)}{\partial z}=
 -2\nu z \frac{\partial}{\partial z} 
 \left(\int\limits_z^\infty \frac{K_\nu(t) I_\nu(t)}{t}\,\dr t
+ \frac{K_\nu(z)}{I_\nu(z)} \int\limits_0^z \frac{I_\nu^2(t)}{t}\,\dr t\right) \\
& = -2\nu z \left(\int\limits_0^z \frac{I_\nu^2(t)}{t}\,\dr t\right)  \frac{\partial}{\partial z}\left(\frac{K_\nu(z)}{I_\nu(z)}\right) =
\frac{2\nu}{I_\nu^2(z)} \int\limits_0^z \frac{I_\nu^2(t)}{t}\, \dr t,
\end{split}
\]
where 
we used the value of the Wronskian $K'_\nu(z)
I_\nu(z) - I'_\nu(z) K_\nu(z) = -1/z$ in the last equality.  As $I_\nu(z)>0$ for all $z > 0$, we conclude that $\frac{\partial \Phi(\nu, z)}{\partial\nu}>0$, which
implies
that $\sigma_{(n)}^{(\Lambda)} < \sigma_{(m)}^{(\Lambda)}$ for any $n < m$ and any
$\Lambda \le 0$.

Let now $\Lambda>0$. Define, for $\nu\ge 0$ and $z>0$ such that $J_\nu(z)\ne 0$,  the function 
\[
\Psi(\nu, z):=\frac{z J'_\nu(z)}{J_\nu(z)},
\]
so that in this case
\[
\sigma_{(n)}^{(\Lambda)} = \Psi\left(n, \sqrt{\Lambda}\right).
\]
Using the representation \cite{Dunster17} of the derivatives of the Bessel functions with respect to order,
\[
\frac{\partial J_\nu(z)}{\partial \nu} = \pi \nu \left(J_\nu(z) \int\limits_z^\infty \frac{J_\nu(t) Y_\nu(t)}{t}\,\dr t
+ Y_\nu(z) \int\limits_0^z \frac{J_\nu^2(t)}{t}\,\dr t\right),
\]
we deduce that
\[
\frac{\partial \Psi(\nu, z)}{\partial\nu} = \frac{2\nu}{J_\nu^2(z)} \int\limits_0^z \frac{J_\nu^2(t)}{t}\,\dr t >0.
\]  
The statement of the theorem for $\Lambda<j_{n,1}^2$ now follows as a consequence and with account of $j_{n,1}<j_{m,1}$ for $0\le n<m$.
\end{proof}

\begin{remark}
A  similar argument shows that the eigenvalues of $\DtN_\Lambda$ in a higher-dimensional ball $ \mathbb{B}^d$, $d\ge 3$, also preserve their order for $\Lambda < \lambda_1^\Dir(\mathbb{B}^d)$.
\end{remark}

\section{Proof of Proposition \ref{prop:dLambda}}\label{sec:proofdLambda}

Set $\sigma_0:=\sigma_k^{(\Lambda_0)}$, $u_0:=u_k^{(\Lambda_0)}$, and $U_0:=U_k^{(\Lambda_0)}$. 
Using the analyticity of eigenvalues of $\DtN_\Lambda$ in $\Lambda$, we seek its eigenvalue $\sigma$, the corresponding eigenfunction $u$ with $\|u\|_{L^2(\pa)}=1$, and its $\Lambda$-harmonic extension $U$ for $\Lambda=\Lambda_0+\epsilon$, $\epsilon\to 0$, in the form
\begin{equation}\label{eq:series}
\begin{split}
\sigma &= \sigma_0 + \epsilon \tilde{\sigma}_1 + \epsilon^2 \tilde{\sigma}_2+O(\epsilon^3),\\
u &= u_0+ \epsilon\tilde{u}_1+\epsilon^2 \tilde{u}_2+O(\epsilon^3),\\
U &= U_0+ \epsilon\tilde{U}_1+\epsilon^2 \tilde{U}_2+O(\epsilon^3).
\end{split}
\end{equation}
Once we have found the coefficients of $\sigma$ expansion, we would immediately get $\left.\frac{\dr \sigma_k^{(\Lambda)}}{\dr \Lambda}\right|_{\Lambda=\Lambda_0}= \tilde{\sigma}_1$ and $\left.\frac{\dr^2 \sigma_k^{(\Lambda)}}{\dr \Lambda^2}\right|_{\Lambda=\Lambda_0}= 2\tilde{\sigma}_2$.
 
 Substituting \eqref{eq:series} into \eqref{eq:efexplicit}, and collecting terms with $\epsilon$ and $\epsilon^2$, we arrive at
 \begin{equation}\label{eq:eps1}
 \begin{cases}
 -\Delta\tilde{U}_1-\Lambda_0\tilde{U}_1 =U_0\qquad&\text{in  }\Omega,\\
 \partial_n \tilde{U}_1-\sigma_0\tilde{u}_1 = \tilde{\sigma}_1 u_0\qquad&\text{on  }\pa,
 \end{cases}
 \end{equation}
 and
  \begin{equation}\label{eq:eps2}
 \begin{cases}
 -\Delta\tilde{U}_2-\Lambda_0\tilde{U}_2 =\tilde{U}_1\qquad&\text{in  }\Omega,\\
 \partial_n \tilde{U}_2-\sigma_0\tilde{u}_2 = \tilde{\sigma}_1 \tilde{u}_1+\ \tilde{\sigma}_2 u_0\qquad&\text{on  }\pa,
 \end{cases}
 \end{equation}
 respectively. 
 
Since $\Lambda_0$ and $U_0$ are the eigenvalue and the eigenfunction of the Robin Laplacian $-\Delta^{\Rob,-\sigma_0}$, the solvability condition for \eqref{eq:eps1} reads
\[
\myscal{U_0, U_0}_{L^2(\Omega)}+\myscal{\tilde{\sigma}_1 u_0,u_0}_{L^2(\pa)}=0,
\]
which after taking into account the normalisation $\|u_0\|_{L^2(\pa)}=1$ becomes \eqref{eq:firstder}.
 
Similarly, the solvability condition for \eqref{eq:eps2} is
\begin{equation}\label{eq:solveps2}
\myscal{\tilde{U}_1, U_0}_{L^2(\Omega)}+\myscal{\tilde{\sigma}_1\tilde{u}_1+\tilde{\sigma}_2  u_0,u_0}_{L^2(\pa)}=0.
\end{equation}
We first simplify it by taking into account  the normalisation conditions on $u_0$ and $u$, with the latter being
\[
\|u\|^2_{L^2(\pa)}=\|u_0\|^2_{L^2(\pa)}+2\epsilon\myscal{\tilde{u}_1,u_0}_{L^2(\pa)}+O(\epsilon^2)=1,
\] 
which means that 
\begin{equation}\label{eq:u1u0prod}
\myscal{\tilde{u}_1,u_0}_{L^2(\pa)}=0.
\end{equation}
Hence, \eqref{eq:solveps2} becomes
\begin{equation}\label{eq:solveps22}
\tilde{\sigma}_2 = -\myscal{\tilde{U}_1, U_0}_{L^2(\Omega)}, 
\end{equation}
but we still need to find the solution $\tilde{U}_1$ from \eqref{eq:eps1} in order to evaluate this. We note that any solution of \eqref{eq:eps1} is defined modulo addition of $U_0$ and therefore seek it in the form 
\begin{equation}\label{eq:tU1}
\tilde{U}_1=c_0 U_0 + V + W, 
\end{equation}
where 
\begin{equation}\label{eq:Vdef}
V:=\left(-\Delta^\Dir-\Lambda_0\right)^{-1}U_0,
\end{equation}
 and $W$ satisfies
\[
\begin{cases}
-\Delta W-\Lambda_0 W =0\qquad&\text{in  }\Omega,\\
\partial_n W - \sigma_0 W = \tilde{\sigma}_1 u_0 - \partial_n V\qquad&\text{on  }\pa,\\
\myscal{W|_{\pa}, u_0}_{L^2(\pa)}=0,\qquad& 
\end{cases}
\]
or, equivalently, with $w=W|_{\pa}$, 
\begin{equation}\label{eq:eps1W2}
\DtN_{\Lambda_0}w=\sigma_0 w + \tilde{\sigma}_1 u_0 - \partial_n V.
\end{equation}
Expanding $w$ into the basis of eigenfunctions of the Dirichlet-to-Neumann map $\DtN_{\Lambda_0}$ as 
\[
w=\sum_{\substack{j\in\mathbb{N}\\j\ne k}} c_j  u_j^{(\Lambda_0)},\qquad c_j=\myscal{w, u_j^{(\Lambda_0)}}_{L^2(\pa)},
\]
and therefore with 
\begin{equation}\label{eq:Wseries}
W=\sum_{\substack{j\in\mathbb{N}\\j\ne k}} c_j  U_j^{(\Lambda_0)},
\end{equation}
substituting into \eqref{eq:eps1W2}, and integrating by parts easily gives 
\begin{equation}\label{eq:cj}
c_j =  \frac{\myscal{U^{(\Lambda_0)}_j, U_0}_{L^2(\Omega)}}{\sigma_j^{(\Lambda_0)}-\sigma_0},
\end{equation}
whereas a combination of \eqref{eq:u1u0prod} and \eqref{eq:tU1} gives $c_0=0$. Together, \eqref{eq:solveps22}--\eqref{eq:Vdef} and \eqref{eq:Wseries}, \eqref{eq:cj} yield \eqref{eq:secondder}.

\section*{Acknowledgements}\addcontentsline{toc}{section}{Acknowledgements}\phantomsection\label{sec:aknow}

We thank Lennie Friedlander, Mikhail Karpukhin, Marco Marletta, Nilima Nigam, Konstantin Pankrashkin, David Sher, and Denis Vinokurov for valuable comments and suggestions.

DSG acknowledges the support of the Simons Foundation for a sabbatical stay at the CRM in Montr\'eal. Research of ML was partially supported by EPSRC. Research of IP was partially supported by NSERC and FRQNT. 

The authors thank the Isaac Newton Institute for Mathematical Sciences, Cambridge, for support and hospitality during the programme \emph{Geometric spectral theory and applications}, where some work on this paper was undertaken, supported by the EPSRC grant EP/Z000580/1.


\section*{Data availability statement}\addcontentsline{toc}{section}{Data availability statement}\phantomsection\label{sec:data}

Some accompanying \texttt{Mathematica} scripts and their printouts are
available for download at \url{https://www.michaellevitin.net/DtN.html} or at
\url{https://github.com/michaellevitin/DtN}.

\phantomsection
\end{document}